\documentclass[a4paper, 11 pt, twoside, notitlepage]{amsart}
\usepackage[left = 3.5cm, right = 3.5cm, headsep = 6mm,
footskip = 10mm, top = 35mm, bottom = 35mm, footnotesep=5mm, headheight =
2cm]{geometry}

\usepackage{multicol}

\usepackage{makecell}
\usepackage{mathtools,leftindex,tensor}
\usepackage[version=4]{mhchem}
\usepackage{longtable}
\usepackage{booktabs}
\usepackage{tabularx}

\RequirePackage{amsmath}
\RequirePackage{amssymb, amsthm, amsfonts}
\RequirePackage[T1]{fontenc}
\usepackage{comment}

\RequirePackage[british]{babel}

\RequirePackage{booktabs}
\RequirePackage{multirow}

\RequirePackage{enumitem}

\RequirePackage[usenames, dvipsnames, pdftex]{xcolor}
\RequirePackage{tikz, tikz-cd, tcolorbox}

\usetikzlibrary{shapes,backgrounds}
\usetikzlibrary{arrows, positioning, intersections}
\usetikzlibrary {through, chains}

\tikzset{every picture/.style={>=Straight Barb}} 
\tikzset{commutative diagrams/arrow style=tikz}

\newtheoremstyle{Lehn-it}
{}
{}
{\itshape}
{}
{\bfseries}
{$\;$\textmd{---}}
{ }
{}

\newtheoremstyle{Lehn-up}
{}
{}
{\upshape}
{}
{\bfseries}
{$\;$\textmd{---}}
{ }
{}

\newtheoremstyle{up-list}
{}
{}
{\upshape}
{}
{\bfseries}
{ }
{ }
{}

\newtheoremstyle{Lehn-Bemerkung}
{}
{}
{}
{}
{\itshape}
{$\;$\textmd{---}}
{ }
{}

\newtheoremstyle{citing}
{}
{}
{\itshape}
{}
{\bfseries}
{$\;$\textmd{---}}
{.5em}
{\thmnote{#3}}

\usepackage{chngcntr}
\usepackage{apptools}

\AtAppendix{\numberwithin{equation}{chapter}}
\AtAppendix{\numberwithin{equation}{section}}
\AtAppendix{\numberwithin{equation}{subsection}}
\AtAppendix{\numberwithin{equation}{subsubsection}}

\numberwithin{equation}{section}

\newenvironment{customthm}[1]
{\renewcommand\theinnercustomthm{#1}\innercustomthm}
{\endinnercustomthm}

\newenvironment{customprop}[1]
{\renewcommand\theinnercustomprop{#1}\innercustomprop}
{\endinnercustomprop}

\newenvironment{customcor}[1]
{\renewcommand\theinnercustomcor{#1}\innercustomcor}
{\endinnercustomcor}

{\theoremstyle{Lehn-it}
	
	\newtheorem{thm}[equation]{Theorem}
	\newtheorem{lem}[equation]{Lemma}
	
	\newtheorem{prop}[equation]{Proposition}
	\newtheorem{cor}[equation]{Corollary}
	
	\newtheorem{innercustomthm}{Theorem}
	\newtheorem{innercustomprop}{Proposition}
	\newtheorem{innercustomcor}{Corollary}
}
{\theoremstyle{Lehn-up}
	\newtheorem{defin}[equation]{Definition}
	\newtheorem{exam}[equation]{Example}
	
}
{\theoremstyle{Lehn-Bemerkung}
	\newtheorem{rem}[equation]{Remark}
	
}

{\theoremstyle{up-list}
	
}

{\theoremstyle{citing}
	}

\RequirePackage{tikz-cd}
\RequirePackage{amsmath, amssymb, amsthm, amsfonts}

\makeatletter

\def\chaptermark#1{}

\def\chapter{%
	\if@openright\cleardoublepage\else\clearpage\fi
	\thispagestyle{plain}\global\@topnum\z@
	\@afterindenttrue \secdef\@chapter\@schapter}

\def\@chapter[#1]#2{\refstepcounter{chapter}%
	\ifnum\c@secnumdepth<\z@ \let\@secnumber\@empty
	\else \let\@secnumber\thechapter \fi
	\typeout{\chaptername\space\@secnumber}%
	\def\@toclevel{0}%
	\ifx\chaptername\appendixname \@tocwriteb\tocappendix{chapter}{#2}%
	\else \@tocwriteb\tocchapter{chapter}{#2}\fi
	\chaptermark{#1}%
	\addtocontents{lof}{\protect\addvspace{10\p@}}%
	\addtocontents{lot}{\protect\addvspace{10\p@}}%
	\@makechapterhead{#2}\@afterheading}
\def\@schapter#1{\typeout{#1}%
	\let\@secnumber\@empty
	\def\@toclevel{0}%
	\ifx\chaptername\appendixname \@tocwriteb\tocappendix{chapter}{#1}%
	\else \@tocwriteb\tocchapter{chapter}{#1}\fi
	\chaptermark{#1}%
	\addtocontents{lof}{\protect\addvspace{10\p@}}%
	\addtocontents{lot}{\protect\addvspace{10\p@}}%
	\@makeschapterhead{#1}\@afterheading}
\newcommand\chaptername{Chapter}

\def\@makechapterhead#1{\global\topskip 7.5pc\relax
	\begingroup
	\fontsize{\@xivpt}{18}\bfseries\centering
	\ifnum\c@secnumdepth>\m@ne
	\leavevmode \hskip-\leftskip
	\rlap{\vbox to\z@{\vss
			\centerline{\normalsize\mdseries
				\uppercase\@xp{\chaptername}\enspace\thechapter}
			\vskip 3pc}}\hskip\leftskip\fi
	#1\par \endgroup
	\skip@34\p@ \advance\skip@-\normalbaselineskip
	\vskip\skip@ }
\def\@makeschapterhead#1{\global\topskip 7.5pc\relax
	\begingroup
	\fontsize{\@xivpt}{18}\bfseries\centering
	#1\par \endgroup
	\skip@34\p@ \advance\skip@-\normalbaselineskip
	\vskip\skip@ }
\def\appendix{\par
	\c@chapter\z@ \c@section\z@
	\let\chaptername\appendixname
	\def\thechapter{\@Alph\c@chapter}}

\newcounter{chapter}

\newif\if@openright

\makeatother

\newcommand{\del}{\partial}
\newcommand{\dbar}{\overline\del}

\DeclareMathOperator{\id}{\mathsf{id}}
\DeclareMathOperator{\Hom}{\mathsf{Hom}}
\DeclareMathOperator{\Mor}{\mathsf{Mor}}
\DeclareMathOperator{\End}{\mathsf{End}}

\newcommand{\Wedge}{\textstyle\bigwedge}

\DeclareMathOperator{\Aut}{\mathsf{Aut}}
\DeclareMathOperator{\rk}{\mathsf{rk}}
\DeclareMathOperator{\im}{\mathsf{im}}

\DeclareMathOperator{\coker}{\mathsf{coker}}

\DeclareMathOperator{\GL}{\mathsf{GL}}

\DeclareMathOperator{\Kur}{Kur}
\DeclareMathOperator{\obs}{obs}
\DeclareMathOperator{\Diff}{Diff}
\DeclareMathOperator{\Ann}{Ann}
\DeclareMathOperator{\Gr}{Gr}

\DeclareMathOperator{\para}{par}

\renewcommand{\epsilon}{\varepsilon}

\DeclareFontFamily{OT1}{rsfs}{}
\DeclareFontShape{OT1}{rsfs}{n}{it}{<-> rsfs10}{}
\DeclareMathAlphabet{\curly}{OT1}{rsfs}{n}{it}

\newcommand{\kc}{{\mathcal C}}
\newcommand{\kd}{{\mathcal D}}

\newcommand{\kf}{{\mathcal F}}
\newcommand{\kg}{{\mathcal G}}
\newcommand{\kh}{{\mathcal H}}
\newcommand{\ki}{{\mathcal I}}

\newcommand{\kk}{{\mathcal K}}
\newcommand{\kl}{{\mathcal L}}
\newcommand{\km}{{\mathcal M}}

\newcommand{\kp}{{\mathcal P}}

\newcommand{\ks}{{\mathcal S}}
\newcommand{\kt}{{\mathcal T}}
\newcommand{\ku}{{\mathcal U}}
\newcommand{\kv}{{\mathcal V}}

\newcommand{\kx}{{\mathcal X}}
\newcommand{\ky}{{\mathcal Y}}
\newcommand{\kz}{{\mathcal Z}}

\newcommand{\IC}{{\mathbb C}}

\newcommand{\IN}{{\mathbb N}}

\newcommand{\IP}{{\mathbb P}}
\newcommand{\IQ}{{\mathbb Q}}
\newcommand{\IR}{{\mathbb R}}

\newcommand{\IZ}{{\mathbb Z}}
\newcommand{\gotha}{{\mathfrak a}}
\newcommand{\gothb}{{\mathfrak b}}

\newcommand{\frg}{{\mathfrak g}}
\newcommand{\frh}{{\mathfrak h}}

\newcommand{\gotht}{{\mathfrak t}}

\DeclareMathOperator{\ad}{\mathsf{ad}}

\usepackage[unicode,bookmarks, pdftex]{hyperref}
\hypersetup{colorlinks=true, citecolor=NavyBlue ,linkcolor=NavyBlue,urlcolor=Orange, pdfpagemode=UseNone, breaklinks=true}

\usepackage{enumitem}
\newlist{prooflist}{description}{1}
\setlist[prooflist]{font=\normalfont \itshape, labelindent = \parindent, leftmargin = 0pt}

\usepackage{comment}

\makeatletter
\g@addto@macro\@floatboxreset\centering
\makeatother

\usepackage[backend=biber, style=alphabetic, url= false, doi= false, maxbibnames= 99, isbn= false]{biblatex}
\usepackage{csquotes}

\title{Teichmüller and moduli spaces for complex nilmanifolds}
\author{Konstantin Wehler}

\address{Konstantin Wehler\\FB 12/Mathematik und Informatik\\
	Philipps-Universit\"at Marburg\\
	Hans-Meerwein-Str. 6\\
	35032 Marburg\\
	Germany}
\email{konstantin.wehler@uni-marburg.de}

\begin{document}
	
	\begin{abstract}
		We develop several tools to study the local and global structure of the Teichmüller and moduli spaces for complex nilmanifolds. As a starting point, we prove a structure theorem for holomorphic maps between complex nilmanifolds, which allows for an explicit description of the Teichmüller and moduli spaces. We then introduce two period maps, and give general criteria for the existence of a complex structure on these spaces.
		
		To illustrate the theory, we prove a global Torelli theorem for various classes of nilmanifolds, including principal torus bundles over tori and almost abelian nilmanifolds. In particular, we show that there are nilmanifolds of arbitrary dimension and nilpotency index for which the Teichmüller space admits the structure of a complex manifold.
		
	\end{abstract}

	\maketitle

	\tableofcontents

	\addtocontents{toc}{\protect\setcounter{tocdepth}{-1}}
	\section*{Introduction}
	\addtocontents{toc}{\protect\setcounter{tocdepth}{1}}
	\addcontentsline{toc}{chapter}{Introduction}

	Given a class $\kc$ of complex manifolds with a fixed underlying smooth manifold, we can ask the following general question: is there a complex space whose points uniquely correspond to the biholomorphism classes of complex manifolds in $\kc$? Such a parameter space is called a moduli space and there are several classes of complex manifolds for which the existence of a moduli space has already been established, especially in algebraic geometry. Among the simplest examples are the classes of elliptic curves and Kodaira surfaces \cite{Borcea84}, which belong to a much larger class of complex manifolds, known as complex nilmanifolds.
	
	A complex nilmanifold is a compact quotient of a simply connected nilpotent Lie group $G$ by a lattice $\Gamma \subset G$, endowed with a left-invariant complex structure.
	
	Unfortunately, it was already shown by Kodaira and Spencer in \cite{Kodaira58} that a moduli space no longer exists for complex tori of higher dimension, while the existence of a moduli space for the higher-dimensional analogues of Kodaira surfaces, known as Kodaira manifolds, was ruled out in \cite{grantcharov2004deformations}.
	This suggests that one should instead consider another variant of a moduli space, known as the Teichmüller space.

	 So far, complex tori and Kodaira manifolds are the only complex nilmanifolds for which the existence of a moduli space has been studied. We will establish general criteria for the existence of a complex structure on the Teichmüller space, and show that this space admits a complex structure for large classes of complex nilmanifolds.

	\subsection*{Teichmüller and moduli spaces}
	
	Let $M = \Gamma \backslash G$ be an even-dimensional nilmanifold and let $\frg$ be the Lie algebra of $G$. By left-translations, any almost complex structure $J\colon\frg\to \frg$ extends to an almost complex structure on $G$. This almost complex structure is invariant under the action of the lattice $\Gamma$ and thus descends to an almost complex structure on $M$, also denoted by $J$. If $J$ is integrable, then the compact complex manifold $(M,J)$ is called a complex nilmanifold.
	
	The lattice $\Gamma$ and the Lie algebra $\frg$ determine the Lie group $G$ up to canonical isomorphism, so we will often omit the Lie group $G$ from the notation and just say $M$ is a nilmanifold of type $(\frg, \Gamma)$. 
	The main objects of our study are the following spaces: we define 
	\[
	\km(\frg, \Gamma) = \left\{ \textup{complex nilmanifolds $(M,J)$}\right\}/ \cong
	\] 
	as the moduli space of biholomorphism classes of complex nilmanifolds of type $(\frg, \Gamma)$.
	The Teichmüller space
	\[
	\kt(\frg) =  \left\{ \textup{complex nilmanifolds $(M,J)$}\right\}/ \sim
	\]
	is defined by the equivalence relation $\sim$ identifying two complex nilmanifolds if they are biholomorphic via a diffeomorphism of $M$ isotopic to the identity. Following the results of Section \ref{sect: affine maps}, we will see that the structure of the Teichmüller space only depends on the underlying Lie algebra, so we omit the lattice from the notation. Our main goal is to identify pairs $(\frg, \Gamma)$ for which $\km(\frg, \Gamma)$ or $\kt(\frg)$ can be endowed, in a natural way, with the structure of a complex space.
	
	In the same way as above, one can also define the moduli and Teichmüller spaces of all possible complex structures on $M$. However, it was shown in \cite{Cat04DefLargeI} that even for the real six-dimensional torus these spaces have countably many connected components of unbounded dimension, so we will restrict our considerations to left-invariant complex structures. It turns out that this is not a serious limitation since we locally stay in the realm of complex nilmanifolds, that is, every sufficiently small deformation of a complex nilmanifold is again a complex nilmanifold \cite{rollenske09b, Has26}. Under additional assumptions the space $\kt(\frg)$ even makes up a union of connected components in the Teichmüller space of all complex structures on $M$ \cite{rollenske09}.

	\subsection*{Holomorphic maps between complex nilmanifolds} 
	
	Since the spaces $\km(\frg, \Gamma)$ and $\kt(\frg)$ are defined by the existence of certain holomorphic maps between complex nilmanifolds, our first goal will be to provide a classification of such maps. It is a classical result that a holomorphic map $f \colon \Gamma\backslash \IC^n \to \Lambda\backslash\IC^m$ between complex tori lifts to an affine map $F \colon \IC^n \to \IC^m$. Concretely, there is a linear map $\Phi \colon \IC^n \to \IC^m$  with $\Phi(\Gamma) \subset \Lambda$ and an element $\lambda \in \Lambda$ such that $F = \Phi + \lambda$. The following result generalises this fact to holomorphic maps between arbitrary complex nilmanifolds.

	\begin{customthm}{A}[Theorem \ref{thm: affine maps}]\label{Theorem A}
		Let $f \colon X \to Y$ be a holomorphic map between complex nilmanifolds $X = (\Gamma \backslash G, J)$ and $Y = (\Lambda\backslash H, I)$. Then there is a holomorphic lift $F \colon (G,J) \to (H,I)$ of $f$ such that $F= R_h \circ \Phi$, where $R_h$ is a right-translation on $H$ and $\Phi\colon G \to H$ is a group homomorphism with $\Phi(\Gamma) \subset \Lambda$.
	\end{customthm}

	As a consequence of Theorem \ref{Theorem A}, we have a much simpler description of the Teichmüller and moduli spaces. Let us denote by 
	\[
	\kc(\frg) = \{ J \colon \frg \to \frg \, | \,  J^2= - \id, \,\textup{$J$ integrable}\}
	\]
	the space of complex structures on $\frg$; for any lattice $\Gamma$ the space $\kc(\frg)$ contains all left-invariant complex structures on a nilmanifold of type $(\frg, \Gamma)$. The space $\kc(\frg)$ can be embedded as a locally closed subspace of a complex Grassmannian and, contrary to the Teichmüller and moduli spaces, it always comes equipped with the structure of a (possibly singular) complex space. The group $G$ acts on $\kc(\frg)$ by right-translations and since every automorphism of the lattice $\Gamma \subset G$ extends uniquely to an automorphism of the Lie group $G$ \cite{Vinberg00}, we have an action of  $\Aut(\Gamma)$ on $\kc(\frg)$, which descends to an action on the quotient $\kc(\frg)/G$.
	\begin{customcor}{B}[Corollary \ref{cor: simple description}]\label{Corollary B}
		For any lattice $\Gamma\subset G$ we have $\kt(\frg) = \kc(\frg)/G$ and $\km(\frg, \Gamma) = \kt(\frg)/\Aut(\Gamma)$.
	\end{customcor}
	
	The space $\kc(\frg)$ also admits an action by the automorphism group $\Aut(\frg)$, and the quotient $\kc(\frg)/\Aut(\frg)$ has already been studied for several classes of Lie algebras \cite{ceballos2016invariant, Latorre23, andrada2025almost}. While understanding the space $\kc(\frg)/\Aut(\frg)$ is of interest in its own right, it does not parametrise complex structures on any nilmanifold since the elements of $\Aut(\frg)$ usually do not descend to the quotient $\Gamma\backslash G$. As already illustrated by the moduli space of Kodaira surfaces \cite{Borcea84}, the situation becomes much more complicated as soon as the lattice starts to play a role.
	
	Unfortunately, the simpler description of the Teichmüller and moduli spaces given in Corollary \ref{Corollary B} certainly does not imply that they have any nice geometric structure. In Section \ref{chap: Teichmüller and moduli spaces} we will give a precise meaning to the question whether or not $\kt(\frg)$ or $\km(\frg, \Gamma)$ admits a natural complex structure. In this case, we will just say that $\kt(\frg)$, respectively $\km(\frg, \Gamma)$, exists as a complex space or a complex manifold.
	
	We will see that one of the main obstructions for these spaces to admit a complex structure is the existence of too many automorphisms. More precisely, to any complex nilmanifold parametrised by $\kc(\frg)$ one can attach the dimension of its automorphism group, that is, the group of its biholomorphisms. We will show that if this dimension is not constant on the connected components of $\kc(\frg)$, then the Teichmüller space cannot exist as a complex space. For example, this occurs if the Lie algebra $\frg$ is non-abelian and admits a complex parallelisable structure. In fact, both $\kt(\frg)$ and $\km(\frg, \Gamma)$ are far from Hausdorff in this case (see Proposition \ref{prop: Tg non hausdorff complex par}).

	\subsection*{Period maps}
	One of the most important geometric features of complex nilmanifolds is that their canonical bundle is always trivial \cite{Barberis09, salamon01}. While complex tori are the only nilmanifolds that admit a Kähler metric, the Kähler case might give some insight into what one could hope for when considering deformations of complex nilmanifolds.
	
	Two central results in the deformation theory of Calabi--Yau manifolds, i.e., Kähler manifolds with trivial canonical bundle, are the Tian--Todorov theorem and the local Torelli theorem. The former states that for a Calabi--Yau manifold $Y$, the Kuranishi space $\Kur(Y)$ is always smooth. Moreover, since every small deformation of a Calabi--Yau manifold is again a Calabi--Yau manifold one can define a local period map on $\Kur(Y)$ by mapping a point $t \in \Kur(Y)$ to the de Rham class of a trivialising section of the canonical bundle of $Y_t$. The local Torelli theorem states that this map is always a holomorphic embedding.
	
	Unfortunately, both of these statements are no longer true for complex nilmanifolds. As a consequence of a result by Rollenske \cite{rollenske09b}, it was recently shown that the Kuranishi family of a complex nilmanifold only contains complex nilmanifolds \cite{Has26}. However, the Kuranishi space of a complex nilmanifold is not necessarily smooth. For instance, it is almost always singular for complex parallelisable nilmanifolds \cite{Rol11Kuranishi, paulsen2024verbal}. Secondly, and perhaps more surprisingly, we will see that in a holomorphic family of complex nilmanifolds there can be infinitesimal directions which no longer have trivial canonical bundle. Due to this phenomenon a holomorphic local period map can only be defined on a suitable reduction of the Kuranishi space of a complex nilmanifold which, unlike in the Kähler case, is not necessarily injective. Nevertheless, it turns out that such a local period map is an important tool in the study of the Teichmüller space. For example, based on a result of Catanese \cite{catanese2011superficial} we obtain the following useful criterion.

	\begin{customprop}{C}[Proposition \ref{prop: catanese criterion}, Lemma \ref{lem: local period map}]
		Let $X$ be an $n$-dimensional complex nilmanifold with underlying Lie algebra $\frg$. If every left-invariant $d$-exact $(n-1,1)$-form is $\dbar$-exact, then the natural map
		\[
		\Kur(X) \to \kt(\frg)
		\] 
		is a local homeomorphism.
	\end{customprop}
	
	Furthermore, since we consider complex nilmanifolds with a fixed underlying Lie algebra $\frg$, we can extend the local period map to a map on the space $\kc(\frg)$. Concretely, if $M$ is a nilmanifold of dimension $2n$ with underlying Lie algebra $\frg$, then, up to scaling, the canonical bundle of a complex nilmanifold $(M,J)$ can be trivialised by a unique closed left-invariant $(n,0)$-form. So we obtain a map $\kc(\frg) \to \IP H^n(\frg, \IC)$ by sending a complex structure $J \in \kc(\frg)$ to the class of an associated $(n,0)$-form on $(M,J)$. It turns out that this map is holomorphic on a reduction of the complex space $\kc(\frg)$, and it descends to a map
	\[
	\kp \colon \kt(\frg) \to \IP H^n(\frg, \IC).
	\]
	We call $\kp$ the global period map. The image of this map is contained in the so-called period domain $\kd$, which is an open subset of a projective variety in $\IP H^n(\frg, \IC)$. One of our main goals is to understand when $\kp$ is injective, as this has strong implications for the structure of the Teichmüller space.
	
	\subsection*{Applications}
	It is a well-known fact that every real nilmanifold $M$ has the structure of an iterated real principal torus bundle, that is, there exists a tower of maps
	\[
	\begin{tikzcd}
		M= M_1 \rar{\pi_1} & M_2 \rar{\pi_2} & {\cdots} \rar{\pi_{\nu-1}} & M_{\nu},
	\end{tikzcd} 
	\]
	where every $M_i$ is a real nilmanifold, the maps $\pi_i\colon M_i \to M_{i+1}$ are real principal torus bundles, and $M_\nu$ is itself a real torus. Conversely, every smooth manifold admitting such a bundle structure is a nilmanifold. Unfortunately, this correspondence fails in the complex setting. While it is true that every iterated holomorphic principal bundle of complex tori is a complex nilmanifold, the converse no longer holds. In order to obtain a global description of the Teichmüller space we often rely on the existence of what is known as a (stable) principal torus bundle series, a notion that was first introduced by Rollenske in \cite{rollenske09}.
	
	Roughly speaking, a principal torus bundle series for a left-invariant complex structure $J$ on $M = \Gamma \backslash G$ is a filtration on the Lie algebra $\frg$ of $G$ compatible with both the complex structure $J$ and the lattice $\Gamma$ (see Definition \ref{def: stbs}). Geometrically, such a filtration induces a tower of principal bundles as above such that the complex structure $J$ induces a complex structure on every nilmanifold $M_i$ and the maps $\pi_i$ are holomorphic with respect to these complex structures.
	
	In \cite{rollenske09} Rollenske proved the existence of a principal torus bundle series for several classes of complex nilmanifolds. Furthermore, he showed that such a torus bundle series is often stable in the sense that it is compatible with all complex structures on a given Lie algebra $\frg$ and also with every lattice. As a first application we study the Teichmüller spaces for principal torus bundles over a torus.
	
	\begin{customthm}{D}[Theorem \ref{thm: fibre map not injective}]\label{Theorem C}
		Let $\frg$ be a $2$-step nilpotent Lie algebra of dimension $2n$ with a stable principal torus bundle series of length $2$. If $\dim \Aut(X) =1$ for every complex nilmanifold $X$ of type $(\frg, \Gamma)$, then the global period map $\kp\colon \kt(\frg) \to \kd$ is injective.
		
		Moreover, the space $\kc(\frg)$ is a complex manifold of dimension $\frac{1}{2}n(n+1)$ and the Teichmüller space $\kt(\frg)$ exists as a complex manifold of dimension $\frac{1}{2}n(n-1) +1$.
	\end{customthm}
	
	We will see several examples of nilmanifolds which satisfy the assumptions of this theorem, among them is the class of Kodaira manifolds.
	Unfortunately, if one loosens any of the assumptions in Theorem \ref{Theorem C}, then it only holds for special classes of Lie algebras. One of these classes are Lie algebras whose commutator $\kc^1\frg$ is at most two-dimensional, in particular, they are at most 3-step nilpotent. In \cite{rollenske09} Rollenske provides a classification of nilpotent Lie algebras with at most two-dimensional commutator that admit a stable torus bundle series. Based on his classification we obtain a fairly complete picture for this class.
	
	\begin{customthm}{E}[Theorem \ref{thm: Teichmüller 2-dim commutator}]\label{Theorem D}
		Let $\frg$ be a nilpotent Lie algebra admitting complex structures. If  $\dim \kc^1\frg \leq 2$, then we have the following cases.
		\begin{enumerate}
			\item If $\dim \kc^1\frg=1$, then $\kt(\frg)$ exists as a complex manifold and the global period map $\kp$ is an embedding.
			\item If $\frg$ is $2$-step nilpotent and the centre $\kz\frg$ is two-dimensional, then $\kt(\frg)$ does not exist as a complex space in general, but $\kp$ is generically injective on the connected components of $\kt(\frg)$ containing complex parallelisable structures.
			\item If $\frg$ is $2$-step nilpotent and $\dim \kz\frg$ is odd, then $\kt(\frg)$ does not exist as a complex space in general.
			\item If $\frg$ is $3$-step nilpotent, then $\kt(\frg)$ exists as a complex manifold and the global period map $\kp$ is an embedding.
		\end{enumerate}
	\end{customthm}
	
	The existence of a stable principal torus bundle series $\ks$ on a Lie algebra $\frg$ also allows us to define another period map $\kp^\ks \colon \kt(\frg) \to \ku$, where $\ku$ is an open subset of a product of Grassmannians. This map roughly measures to what extent the complex tori making up the fibres and base of an iterated torus bundle determine the isomorphism class of a complex nilmanifold. Even though there is generally no relation between the two period maps $\kp$ and $\kp^\ks$ we will see that it is much rarer for $\kp^\ks$ to be injective. In fact, we will often have a precise characterisation of Lie algebras for which $\kp^\ks$ is injective.
	
	Another class for which the existence of a stable principal torus bundle series $\ks$ has recently been proved is the class of so-called almost abelian Lie algebras \cite{andrada2025almost}.
	An almost abelian Lie algebra is a non-abelian Lie algebra admitting a codimension one abelian ideal, and there exists a complete classification of nilpotent almost abelian Lie algebras admitting a complex structure~\cite{ABDGH24}.
	Even though neither the nilpotency index nor the dimension of the commutator is bounded in this class, it turns out that almost abelian Lie algebras are very well suited for the study of our moduli problems. 
	
	\begin{customthm}{F}[Theorem \ref{thm: global period map almost abelian}, Proposition \ref{prop: fibre map almost abelian}]
		Let $\frg$ be a nilpotent almost abelian Lie algebra admitting complex structures. 
		\begin{enumerate}
			\item The Teichmüller space $\kt(\frg)$ exists as a complex manifold and the period map $\kp \colon \kt(\frg) \to \kd$ is an embedding.
			\item The period map $\kp^\ks \colon \kt(\frg) \to \ku$ is injective if and only if $\frg$ is the real Lie algebra underlying a Kodaira surface.
		\end{enumerate}
	
	\end{customthm}

	Finally, we apply the previous results to study the Teichmüller and moduli spaces for nilmanifolds of dimension six. It was shown by Salamon in \cite{salamon01} that there are 18 isomorphism classes of six-dimensional nilpotent Lie algebras admitting a complex structure. While a complete description of all Teichmüller and moduli spaces in this class requires a more detailed analysis, we already cover some parts as an immediate consequence of our previous results. For instance, we obtain
	
	\begin{customthm}{G}[Theorem \ref{thm: dimension six}]
		Let $\frg$ be a nilpotent Lie algebra of dimension six admitting complex structures. If $\dim \kc^1\frg \leq 2$, then $\kp \colon \kt(\frg) \to \kd \subset \IP H^3(\frg, \IC)$ is injective and $\kt(\frg)$ exists as a complex manifold, unless $\frg$ is the real Lie algebra underlying the Iwasawa manifold in which case $\kt(\frg)$ is non-Hausdorff.
	\end{customthm}
	
	A different approach to understanding the Teichmüller and moduli spaces, which avoids the problem of dealing with non-Hausdorff spaces, was taken by Meersseman in \cite{Meersseman19}. Instead of considering the Teichmüller and moduli spaces as topological spaces, Meersseman constructed them as stacks for an arbitrary smooth compact manifold. Moreover, he showed that under a mild condition on the dimension of the automorphism group, these stacks are analytic \cite{Meersseman19}, which is far from true on the level of topological spaces. Although the Teichmüller and moduli stacks are in general difficult to describe, our results suggest that an explicit description of some connected components should be possible for nilmanifolds admitting left-invariant complex structures.

	\subsection*{Acknowledgements}
	This work is part of my PhD thesis \cite{Wehler26}. I would like to thank my advisers Sönke Rollenske and Nicolina Istrati for their constant support, and many valuable comments and discussions. I would also like to thank Stefan Nemirovski for a helpful conversation regarding Proposition \ref{prop: polynomial growth}. This work was supported by the DFG through the grant RO 3734/4-1. 
	
	\chapter*{Part I: General theory}
	
	\section{Preliminaries}\label{chap: geometry of complex nilmanifolds}
	
	This section collects some definitions and basic results about complex nilmanifolds. 
	
	\subsection{Complex structures on nilpotent Lie algebras}\label{sect: complex structures on nilpotent Lie algebras}
	
	Let $\frg$ be a real nilpotent Lie algebra of dimension $2n$. An almost complex structure on $\frg$ is an endomorphism $J \colon\frg \to \frg$ such that $J^2 = -\id$, which is said to be integrable if its Nijenhuis tensor vanishes, that is, we have
	\begin{equation*}\label{eq: nijenhuis}
		[x,y]-[Jx,Jy]+J[Jx,y]+J[x,Jy] =0
	\end{equation*}
	 for all $x,y \in \frg$. In this case, the pair $(\mathfrak{g}, J)$ is called a nilpotent Lie algebra with complex structure.

	 Alternatively, we have a decomposition $\frg_\IC = \frg^{1,0}\oplus\frg^{0,1}$ of the complexification of $\frg$ into the $\pm i$-eigenspaces of the complex linear extension of the almost complex structure $J$, and $J$ is integrable if and only if $[\frg^{1,0}, \frg^{1,0}] \subset \frg^{1,0}$. Conversely, every decomposition $\frg_\IC=V \oplus \overline{V}$ defines a unique almost complex structure on $\frg$ whose $i$-eigenspace is $V$.

	It is often useful to describe complex structures on $\frg$ in terms of differential forms. We define $d\colon \frg^* \to \Wedge^2\frg^*$ by
	\begin{equation}\label{eq: differential}
		d\alpha(x,y) = -\alpha([x,y])
	\end{equation}
	for $\alpha \in \frg^*$ and $x,y\in \frg$. The map $d$ induces a map on the exterior algebra $\Wedge^\bullet \frg^*$ and makes $(\Wedge^\bullet \frg^*, d)$ into a differential graded algebra. The Jacobi identity on $\frg$ corresponds to $d^2 =0$.
	If $e^1, \dots, e^{2n}$ is a basis of $\frg^*$, then the Lie algebra $\frg$ can be described by the real structure equations
	\[
	de^k = \sum_{i<j} c_{ij}^k e^i \wedge e^j,
	\]
	where the constants $c_{ij}^k\in \mathbb{R}$ are called the structure constants with respect to the chosen basis of $\mathfrak{g}^*$.
	
	An almost complex structure $J$ also induces a decomposition $\mathfrak{g}_\mathbb{C}^* = \mathfrak{g}^{*1,0} \oplus \mathfrak{g}^{*0,1}$, and thus for $1 \leqslant k \leqslant 2n$ the exterior algebra of $\mathfrak{g}_\IC^*$ decomposes as
	\begin{equation*}
		\Wedge^k\mathfrak{g}_\mathbb{C}^* = \bigoplus_{p+q = k} \Wedge^{p,q}\mathfrak{g}^*,
	\end{equation*}
	where $\Wedge^{p,q}\mathfrak{g}^* = \Wedge^p\mathfrak{g}^{*1,0} \otimes \Wedge^q\mathfrak{g}^{*0,1}$ and $\overline{\Wedge^{p,q}\mathfrak{g}^*} = \Wedge^{q,p}\mathfrak{g}^*$. With respect to this bi-grading we obtain maps  $\partial \colon \Wedge^{p,q}\frg^* \to \Wedge^{p+1,q}\frg^*$ and $\dbar \colon \Wedge^{p,q}\frg^* \to \Wedge^{p,q+1}\frg^*$, defined by composing the differential $d$ with the respective projections. The almost complex structure $J$ is integrable in the above sense if and only if the differential $d$ on $\Wedge^\bullet \frg_\IC^*$ decomposes as $d = \partial + \dbar$. If $\omega^1, \dots, \omega^n$ is a complex basis of $\frg^{*1,0}$, then the integrability of $J$ implies that the differentials $d\omega^k$ are of the form
	\[
	d\omega^k = \sum_{i<j} A_{ij}^k \omega^{i} \wedge \omega^j +\sum_{i,j} B_{ij}^k \omega^{i} \wedge \overline \omega^j,
	\]
	where $A_{ij}^k, B_{ij}^k\in \IC$. We will often abbreviate $\omega^i \wedge \omega^j = \omega^{ij}$ and $\omega^i \wedge \overline\omega^j = \omega^{i\bar j}$. We refer to these equations as the complex structure equations for the pair $(\frg, J)$ with respect to the chosen basis.
	
	\subsection{Complex nilmanifolds}\label{sect: complex nilmanifolds}

	Let $\frg$ be a nilpotent Lie algebra and let $G$ be an associated simply connected nilpotent Lie group. A lattice $\Gamma \subset G$ is a discrete co-compact subgroup of $G$. By the results of \cite{malcev51} a lattice exists if and only if there is a rational subalgebra $\frg_\IQ$ of $\frg$ such that $\frg_\IQ \otimes \IR = \frg$. Such a subalgebra is called a rational structure for $\frg$ and its existence is equivalent to the existence of a basis of $\frg^*$ with rational structure constants.

	If $G$ does admit a lattice $\Gamma$, then $\Gamma$ acts on $G$ by multiplication from the left and the compact manifold $M = \Gamma \backslash G$ is called a \textit{nilmanifold}. Every complex structure $J$ on $\mathfrak{g}$ induces a complex structure on $G$, making $(G,J)$ into a complex manifold on which all left-translations are holomorphic. Since the lattice $\Gamma$ acts by left-multiplication, this left-invariant complex structure on $G$ descends to a complex structure on $M$. A complex structure on $M$ is said to be left-invariant if it is induced by a complex structure on $\frg$ in this way.
	
	\begin{defin}
		A \textit{complex nilmanifold} is a complex manifold $(M,J)$, where $M$ is a nilmanifold and $J$ is a left-invariant complex structure.
	\end{defin}

	\begin{rem}\label{rem: Kähler are tori}
		It was shown by Hasegawa in \cite{hasegawa1989minimal} that a nilmanifold is formal if and only if it is diffeomorphic to a torus. In particular, a complex nilmanifold $X$ does not carry a Kähler metric or satisfy the $\partial\dbar$-lemma unless $X$ is a complex torus.
	\end{rem}

	\subsection{Principal torus bundle series}\label{sect: torus bundles}
	In the following, we will recall some aspects of the real and complex geometry of nilmanifolds. Let $\frg$ be a real nilpotent Lie algebra. Following the notation of \cite{rollenske09}, the ascending central series on $\frg$ is the filtration defined by
	\[
	\kz^{i+1}\frg = \{ x \in \frg \,|\, [\frg,x] \subset \kz^i\frg\},
	\]
	where $\kz^0\frg =0$. The descending central series is the filtration defined by 
	\[
	\kc^{i+1}\frg = [\kc^i\frg, \frg]
	\]
	with $\kc^0\frg =\frg$, and the Lie algebra $\frg$ is said to be \textit{$\nu$-step nilpotent} if $\nu \in \IN$ is the smallest integer such that $\kz^\nu\frg = \frg$, or equivalently, $\kc^{\nu}\frg  =0$.

	We also recall that a subalgebra $\frh\subset \frg$ is called rational with respect to a rational structure $\frg_\IQ$ of $\frg$ if the intersection $\frh \cap \frg_\IQ$ is a rational structure for $\frh$. The subalgebras appearing in the ascending and descending central series of $\frg$ are rational with respect to every rational structure on $\frg$ \cite[Corollary~5.2.2, Theorem~5.2.3]{Cor-Green}. Therefore, both of these filtrations are compatible with every lattice $\Gamma \subset G$ in an associated simply connected Lie group, in the sense that they
	 induce a (possibly different) real iterated torus bundle structure on the nilmanifold $\Gamma \backslash G$ (see \cite[Section~1]{rollenske09} for details). 
	
	Unfortunately, this no longer works in the complex setting since both the ascending and descending central series are not necessarily preserved by a given complex structure on $\frg$. So in order to obtain a holomorphic torus bundle structure on a complex nilmanifold we need the following notion introduced in \cite{rollenske09}.

	\begin{defin}\label{def: stbs} 
		Let $(\frg,J)$ be a nilpotent Lie algebra with complex structure and a fixed rational structure. A $J$-invariant filtration
		\[0=\ks^0\frg \subset \ks^1\frg \subset \dots \subset \ks^\ell\frg=\frg\]
		of rational subalgebras is called a \textit{principal torus bundle series} of length $\ell$ if $\ks^i\frg$ is an ideal in $\ks^{i+1}\frg$ and $\ks^i\frg /\ks^{i-1}\frg$ is contained in the centre of $\frg/\ks^{i-1}\frg$ for all $i=1, \dots, \ell$.

		If the filtration $(\ks^i\frg)_{i=0, \dots, \ell}$ is invariant under every complex structure and the subspaces $\ks^i\frg$ are rational with respect to every rational structure on $\frg$, then it is called a stable principal torus bundle series (SPTBS).
	\end{defin}
	
	\begin{rem}
		The length $\ell$ of a principal torus bundle series may not coincide with the nilpotency index of $\frg$ \cite[Example~1.14]{rollenske09}. The same example also shows that not every nilpotent Lie algebra admits a SPTBS. But, as shown in \cite[Section~3]{rollenske09}, a SPTBS commonly exists on Lie algebras with small commutator, and for most Lie algebras up to real dimension six. Moreover, it was recently shown in \cite{andrada2025almost} that a SPTBS always exists for almost abelian Lie algebras, which we will discuss in Section \ref{sect: almost abelian nilmanifolds}.
	\end{rem}

	 A SPTBS of length $\ell$ on a nilpotent Lie algebra $\frg$ with associated Lie group $G$ induces for every complex structure $J$ on $\frg$ and for every lattice $\Gamma \subset G$ a tower of holomorphic principal torus bundles
	 \begin{equation}\label{eq: tower nilmanifold}
	 	\begin{tikzcd}
	 		X= X_1 \rar{\pi_1} & X_2 \rar{\pi_2} & {\cdots} \rar{\pi_{\ell-1}} & X_{\ell}
	 	\end{tikzcd} 
	 \end{equation}
	 on the complex nilmanifold $X = (\Gamma\backslash G, J)$. Concretely, for a given (stable) principal torus bundle series $(\ks^i\frg)_{i=1,\dots, \ell}$, the real Lie algebra underlying the complex nilmanifold $X_i$ is $\frg/\ks^i \frg$ and the fibres of the map $\pi_i \colon X_i \to X_{i+1}$ are complex tori $T_i$ with underlying Lie algebra $\ks^i\frg /\ks^{i-1}\frg$. The complex structures on these manifolds are given by the complex structures on $\frg/\ks^i\frg$ and $\ks^i\frg/\ks^{i-1}\frg$, respectively, induced by the left-invariant complex structure on $X$.

	\subsection{Abelian and complex parallelisable structures}
	The existence of a principal torus bundle series for a Lie algebra with complex structure $(\frg, J)$ implies that $J$ is a nilpotent complex structure in the sense of \cite{Cordero00NilpotentStructures}, but the converse does not hold in general \cite[Example~1.14]{rollenske09}. However, there are two classes of nilpotent complex structures where one always has a (not necessarily stable) principal torus bundle series. If the subalgebra $\frg^{1,0}$ is abelian, or equivalently, if $d(\Wedge^{1,0}\frg^*) \subset \Wedge^{1,1}\frg^*$, then $J$ is called an \textit{abelian} complex structure.
 
 	If $[\frg^{1,0}, \frg^{0,1}]=0$, or equivalently, if $d(\Wedge^{1,0}\frg^*) \subset \Wedge^{2,0}\frg^*$, then the pair $(\frg, J)$ is a complex Lie algebra and $J$ is called a \textit{complex parallelisable} structure. In this case, the Lie group $(G,J)$ is a complex Lie group and every complex nilmanifold $(\Gamma \backslash G,J)$ is complex parallelisable.

	The ascending central series is a principal torus bundle series for both abelian and complex parallelisable structures. While the descending central series is not necessarily preserved by an abelian complex structure, it also provides another principal torus bundle series for all complex parallelisable structures.
	
		\begin{exam}\label{exam: Iwasawa fibres}
		Consider the three-dimensional complex Heisenberg group
		\[
		H = \left\{
		\begin{pmatrix}
			1 & z_1 & z_3\\
			0 & 1 & z_2\\
			0 & 0 & 1
		\end{pmatrix}
		\middle | z_i \in \IC
		\right\}
		\]
		 with its lattice $\Lambda = \GL(3, \IZ[i]) \cap H$.  The \textit{Iwasawa manifold} is the three-dimensional complex parallelisable nilmanifold $\ki = \Lambda\backslash H$. The map $H \to \IC^2$ given by the projection on $(z_1, z_2)$ induces a principal bundle $\ki \to T_2$ over a complex 2-torus $T_2$, where the fibre $C$ is an elliptic curve with complex multiplication and $T_2$ is isogenous to the product $C\times C$  \cite[Example~9.1.3]{Winkelmann98}. It was shown in \cite[Proposition~3.8]{rollenske09} that the descending central series on the real Lie algebra underyling $\ki$ defines a SPTBS.
		
		It turns out that many of the phenomena related to our moduli problems can be observed on the real nilmanifold underlying $\ki$. So the Iwasawa manifold and its real Lie algebra will be a recurring example throughout the text.
	\end{exam} 
	
	\begin{rem}\label{rem: Winkelmann tori}
		In \cite[Chapter~9]{Winkelmann98} Winkelmann studied the complex tori $T_i$ appearing as the fibres of the maps $\pi_i\colon X_i \to X_{i+1}$ of an iterated bundle as in \eqref{eq: tower nilmanifold} for a given complex parallelisable nilmanifold $X$. If the bundle structure on $X$ is given by the descending central series, then the tori $T_1, \dots, T_{\nu-1}$ are isogenous to a product of simple tori with complex multiplication \cite[Theorem~9.1.1]{Winkelmann98}. Under an additional assumption on the Lie algebra this also holds for the base torus $T_\nu$ \cite[Proposition~9.1.2]{Winkelmann98}, as in Example \ref{exam: Iwasawa fibres}. However, there is generally no relation between the fibres and base of an iterated principal torus bundle given by a nilmanifold with an abelian complex structure.
	\end{rem}

	\begin{exam}\label{exam: Kodaira surface}
		A (primary)  \emph{Kodaira surface} is a non-trivial principal bundle of elliptic curves over an elliptic curve. Apart from complex tori, these are the only two-dimensional complex nilmanifolds, and all complex structures on a Kodaira surface are abelian. It was shown in \cite{Borcea84} that for every choice of elliptic curves $E_1$ and $E_2$ there is a Kodaira surface with fibre $E_1$ and base $E_2$.
	\end{exam}

	\subsection{Cohomology}\label{subsect: Lie algebra cohomology}
	Following \cite{rollenske09b}, we will briefly describe the cohomology of a nilmanifold in terms of left-invariant forms.
	
	Let $\frh$ be a Lie algebra of dimension $m$ and let $V$ be an $\frh$-module. Then, as explained in \cite[Section~7]{Weibel}, the cohomology $H^\bullet(\frh, V)$ of $\frh$ with values in $V$ can be computed via the Chevalley--Eilenberg complex
	\[
	\begin{tikzcd}
		0  \rar & V \rar{d} & \frh^* \otimes V \rar{d} &   {\cdots}  \rar{d} & \Wedge^{m-1}\frh^* \otimes V \rar{d} & \Wedge^{m}\frh^* \otimes V.
	\end{tikzcd}
	\]
	 For instance, if $\frh = \frg$ is a real nilpotent Lie algebra and we consider $V =\IR$ as the trivial $\frg$-module, then the differential in the associated Chevalley--Eilenberg complex $(\Wedge^\bullet \frg^*,d)$ is the differential defined in \eqref{eq: differential}.
	  Moreover, the space $\Wedge^\bullet \frg^*$ can be considered as the space of left-invariant forms on any nilmanifold $M$ with underlying Lie algebra $\frg$, and $d$ coincides with the exterior derivative on $M$ restricted to $\Wedge^\bullet \frg^*$. This yields an inclusion of $H^\bullet(\frg, \IR)$ into the de Rham cohomology of $M$, which was shown to always be an isomorphism by Nomizu \cite{nomizu54}.
	 
	There is an analogous description of the Dolbeault cohomology of a complex nilmanifold. Let $(\frg, J)$ be a nilpotent Lie algebra with complex structure and let $\frg^{0,1} \subset \frg_\IC$ be the associated $(-i)$-eigenspace. We will mostly be interested in the cohomology groups $H^q(\frg^{0,1}, V)$ for  $V= \Wedge^p\frg^{*1,0}$ or $V = \frg^{1,0}$. The $\frg^{0,1}$-module structure on $\Wedge^p\frg^{*1,0}$ is induced by the action 
	\begin{equation*}
		{\frg^{0,1}}\times {\frg^{*1,0}} \to {\frg^{*1,0}},\qquad  (\overline x, \omega) \mapsto-\overline x\lrcorner \dbar\omega.
	\end{equation*}
	The associated Chevalley--Eilenberg complex is the complex $(\Wedge^{p,\bullet}\frg^*, \dbar)$, and the groups 
	\[
	H^{p,q}(\frg,J) = H^q(\frg^{0,1}, \Wedge^p\frg^{*1,0})
	\]
	are called the Dolbeault cohomology groups of $(\frg,J)$.
	The $\frg^{0,1}$-module structure on $\frg^{1,0}$ is defined by 
	\[
	{\frg^{0,1}}\times {\frg^{1,0}} \to {\frg^{1,0}},\qquad  (\overline x, y) \mapsto [\overline x, y]^{1,0}
	\]
	and the differential in the Chevalley--Eilenberg complex computing the cohomology groups $H^q(\frg^{0,1}, \frg^{1,0})$ is also denoted by $\dbar$.

	As above we can identify $\Wedge^{p,q}\frg^*$ with the space of left-invariant $(p,q)$-forms on the complex nilmanifold $X = (M,J)$. Therefore, we obtain a map
	\[
	H^{p,q}(\frg, J) \to H_{\dbar}^{p,q}(X)
	\]
	to the Dolbeault cohomology of $X$, which was recently shown to always be an isomorphism \cite[Theorem~1.1]{Has26}. Similarly, one has an isomorphism between $H^q(\frg^{0,1}, \frg^{1,0})$ and $H^q(X, T_X) \cong H^q(X, \Theta_X)$, where $T_X$ is the holomorphic tangent bundle of $X$ and $\Theta_X$ is the sheaf of holomorphic vector fields~\cite[Corollary~2.5]{rollenske09b}.
	
	\subsection{Deformation theory}\label{subsubsect: Kuranishi theory}
	
	In this section, we review the deformation theory of complex nilmanifolds. If $X$ is a complex nilmanifold with underlying Lie algebra $\frg$, then the Kodaira--Spencer DGLA controlling the deformation theory of $X$ \cite[Section~8.3]{Manetti22}, contains the finite-dimensional DGLA 
	\[(\Wedge^\bullet \frg^{*0,1} \otimes \frg^{1,0},\, \dbar,\, [\cdot, \cdot]).\]
	Here, $(\Wedge^\bullet \frg^{*0,1} \otimes \frg^{1,0}, \dbar)$ is the Chevalley--Eilenberg complex for the $\frg^{0,1}$-module $\frg^{1,0}$ discussed in the previous section, and for $\overline\alpha \otimes x$, $\overline\beta \otimes y \in \frg^{*0,1} \otimes \frg^{1,0}$ the bracket
	\[
	[\overline\alpha \otimes x, \overline\beta \otimes y] = \overline \alpha \wedge L_x\overline \beta \otimes y  +  \overline \beta \wedge L_y\overline \alpha \otimes x + \overline\alpha \wedge \overline\beta \otimes [x,y],
	\]
	where $L$ is the Lie-derivative, is the restriction of the Schouten bracket on $X$ to $\frg^{*0,1} \otimes \frg^{1,0}$.
	
	 It was shown in \cite[Theorem~2.6]{rollenske09b} that the inclusion of $(\Wedge^\bullet \frg^{*0,1} \otimes \frg^{1,0},\, \dbar, \,[\cdot, \cdot])$ into the Kodaira--Spencer DGLA is a quasi-isomorphism if the Dolbeault cohomology of $X$ is computed by left-invariant forms. Together with the main result of \cite{Has26} mentioned in the previous section, this implies that sufficiently small deformations of any complex nilmanifold carry a left-invariant complex structure.
	
	The small deformations of a compact complex manifold are encoded by a complex space known as the Kuranishi space \cite{kuranishi62}. We will briefly recall Kuranishi's construction for a complex nilmanifold $X$ with underlying Lie algebra $\frg$ (cf. \cite{Rol11Kuranishi, paulsen2024verbal}). Since we already established that every sufficiently small deformation of the left-invariant complex structure on $X$ is again left-invariant, every such deformation can be described by a form $\phi \in \frg^{*0,1} \otimes \frg^{1,0}$ satisfying the Maurer--Cartan equation
	\begin{equation}\label{eq: MC}
		\dbar \phi +\tfrac12[\phi, \phi]=0.
	\end{equation}
	The first-order solutions of the Maurer--Cartan equation correspond to infinitesimal deformations of $X$, which are parametrised by $H^1(X,\Theta_X) \cong H^1(\frg^{0,1}, \frg^{1,0})$.  The complex manifold $X$ is said to have \textit{unobstructed} deformations if every first-order solution can be integrated to an actual solution of the Maurer--Cartan equation. More precisely, this means that there exist representatives $\eta_1, \dots, \eta_m$ of a basis of $H^1(X, \Theta_X)$ such that for every $\phi_1(t) = \sum_{i=1}^m t_i \eta_i$ with $t=(t_1,\dots,t_m) \in \IC^m$ there is a formal solution
	\begin{equation}\label{eq: series}
		\phi(t) = \phi_1(t) + \sum_{k \geq 2} \phi_k(t)
	\end{equation}
	to the system of equations
	\begin{equation}\label{eq:Phi_k}
		\dbar\phi_k(t) = -\frac{1}{2}\sum_{0<i<k}[\phi_i(t), \phi_{k-i}(t)], \qquad k\geq 1.
	\end{equation}
	By the results of \cite{kuranishi62} the existence of a formal solution to \eqref{eq:Phi_k} also ensures the existence of a convergent solution. 
	
	In order to define the Kuranishi space of $X$ we have to remove the ambiguity in the choice of a possible solution to \eqref{eq:Phi_k}. This is achieved by choosing a left-invariant hermitian metric on $X$. Let $\kh^q(X,\Theta_X)$ denote the space of harmonic $(0,q)$-forms with values in $\Theta_X$ defined for the chosen metric, and let $\kg$ denote the Green operator. If $\theta_1, \dots, \theta_m$ is a basis of $\kh^1(X, \Theta_X)$, then starting from $\phi_1(t)= \sum_{i=1}^m t_i\theta_i$ we can recursively define
	\begin{equation}\label{eq: Green}
		\phi_k(t) = -\dbar^*\kg\left(\frac{1}{2}\sum_{0<i<k}[\phi_i(t), \phi_{k-i}(t)]\right).
	\end{equation}
	The resulting formal power series $ \phi(t) = \sum_{k\geq 1} \phi_k(t)$ now satisfies the Maurer--Cartan equation \eqref{eq: MC} if and only if $H[\phi(t), \phi(t)] =0$, where $H$ is the projection to $\kh^2(X, \Theta_X)$. As in \cite{Rol11Kuranishi} we define the so-called obstruction map
	\[
	\obs \colon \kh^1(X, \Theta_X) \to \kh^2(X, \Theta_X)
	\]
	 by $\obs(\theta) = H[\phi(t), \phi(t)]$, where $\phi(t)$ is the power series defined by \eqref{eq: Green} with $\phi_1(t) = \theta$. In this notation, Kuranishi's theorem \cite{kuranishi62} states that for sufficiently small $\epsilon>0$, there is a semi-universal family of deformations of $X$ over the Kuranishi space 
		\[\Kur(X)=\{\theta\in \kh^1(X,\Theta_X)\mid \|\theta\|<\epsilon, \, \obs(\theta)=0\},\]
	known as the Kuranishi family of $X$.
	
	\begin{rem}
		Although this construction depends on the chosen hermitian metric, different choices result in isomorphic families. 
		Moreover, the complex space $\Kur(X)$ is smooth, in which case $\Kur(X)$ is just an open ball around $0 \in  \kh^1(X, \Theta_X)$, if and only if $X$ has unobstructed deformations. 
		
		In general, it is a difficult question whether or not a given complex manifold has unobstructed deformations, even for complex nilmanifolds. The only class of complex nilmanifolds where a precise characterisation is known is the class of complex parallelisable nilmanifolds \cite{paulsen2024verbal}. For this this class explicit solutions to the system of equations \eqref{eq:Phi_k} were computed for several examples in \cite{Rol11Kuranishi} and \cite{popovici2022deformations}.
		
	\end{rem}

	\section{Holomorphic maps between complex nilmanifolds}\label{sect: affine maps}
	
	The goal of this section is to prove the following structure theorem for holomorphic maps between complex nilmanifolds.
	
	\begin{thm}\label{thm: affine maps}
		Let $f \colon X \to Y$ be a holomorphic map between complex nilmanifolds $X = (\Gamma \backslash G, J)$ and $Y = (\Lambda\backslash H, I)$. Then there is a holomorphic lift $F \colon (G,J) \to (H,I)$ of $f$ such that $F= R_h \circ \Phi$, where $R_h$ is a right-translation on $H$ and $\Phi\colon G \to H$ is a group homomorphism with $\Phi(\Gamma) \subset \Lambda$.
	\end{thm}

	This theorem is well-known for holomorphic maps between complex tori (see for instance \cite[Theorem~1.5]{Debarre05Tori}). In \cite{Schumacher79Iwasawa} this was generalised to a class of complex parallelisable nilmanifolds and later proved
	by Winkelmann for holomorphic maps between arbitrary complex parallelisable nilmanifolds \cite[Theorem~5.2.5]{Winkelmann98}.

	\subsection{The universal cover of a complex nilmanifold}\label{subsect: universal cover}
	
	In order to prove Theorem~\ref{thm: affine maps} we first recall a recent result stating that the universal cover of an $n$-dimensional complex nilmanifold is always biholomorphic to $\IC^n$ \cite{Hasegawa26, Wehler26}.
	
	Let $J$ be a left-invariant complex structure on a simply connected nilpotent Lie group $G$, and let $G_\IC$ be its universal complexification. Since $G_\IC$ is also simply connected and nilpotent, the exponential map $\exp \colon \frg_\IC \to G_\IC$ is a biholomorphism, and thus $\exp(\frg^{1,0})$ and $\exp(\frg^{0,1})$ are closed subgroups of $G_\IC$. We can define a biholomorphism from $\IC^n$ to $G_\IC / \exp(\frg^{0,1})$ by
	\begin{equation}\label{eq: quotient is C^n}
		\begin{tikzcd}
			\IC^n \cong \frg^{1,0} \arrow[r, "\exp"] &\exp(\frg^{1,0}) \arrow[hookrightarrow]{r} &G_\IC \arrow[r, "\pi"] &G_\IC / \exp(\frg^{0,1})
		\end{tikzcd}
	\end{equation}
	and we denote its inverse by $\varphi\colon G_\IC/\exp(\frg^{0,1}) \to \IC^n$ (cf. \cite[Remark~2.2]{Hasegawa26}). Then, restricting $\varphi\circ\pi \colon G_\IC \to \IC^n$ to $G \subset G_\IC$ we obtain a holomorphic map
	\[P \colon (G,J) \to \IC^n.\]
	 This map was constructed by Snow in the more general setting of solvable Lie groups, where it was also shown that $P$ is a local biholomorphism \cite[Theorem~1, p.~194]{Snow86}.
	
	However, much more can be said in the nilpotent case. Indeed, it was shown in \cite{Wehler26} that the map $P \colon (G,J) \to \IC^n$ is a biholomorphism if $G$ is nilpotent. This result was also obtained independently by Hasegawa, Sillari, and Tomassini, and is contained in the joint article \cite{Hasegawa26}.

	\subsection{Proof of Theorem \ref{thm: affine maps}}
	
	We will first outline the strategy of the proof. Let $X = (\Gamma \backslash G, J)$ and $Y = (\Lambda\backslash H, I)$ be complex nilmanifolds of dimension $n$ and $m$, respectively. Let $f\colon X \to Y$ be a holomorphic map. The proof of Theorem \ref{thm: affine maps} is roughly divided into three steps.

	\textit{Step 1}: first, we will show that the proof of Theorem \ref{thm: affine maps} can be reduced to the case where $f$ maps the class of the neutral element $[e_G] \in X$ to the class $[e_H] \in Y$. Under the assumption $f([e_G]) = [e_H]$, the map $f$ has a unique holomorphic lift $F \colon (G,J) \to (H,I)$ such that $F(e_G) = e_H$ and in particular $F(\Gamma) \subset \Lambda$. Thus, to prove the theorem it remains to show that $F$ is a group homomorphism.

	\textit{Step 2}: next, we will pass from $(G,J)$ to $\IC^n$ and from $(H,I)$ to $\IC^m$ using the biholomorphism defined in the previous section. So the map $f\colon X\to Y$ lifts to a holomorphic map $\widetilde{F} \colon \IC^n \to \IC^m$ and by the first step we may assume $\widetilde F(0) =0$. We will prove that the lift $ \widetilde F\colon \IC^n \to \IC^m$ is a polynomial map.

	\textit{Step 3}: lastly, we will use the fact that $\widetilde{F}$ is a polynomial map to deduce that the lift $F \colon G \to H$ is a group homomorphism.

	\subsubsection{Step 1}
	
	Assume that Theorem \ref{thm: affine maps} holds for every holomorphic map sending the class of $e_G\in G$ in $X$ to the class of $e_H \in H$. Suppose $f\colon X \to Y$ maps the class $[e_G]$ to $[h]$ for some $h\in H$. Then the map $f' = R_{h^{-1}}\circ f$ defines a holomorphic map $X \to Y'= (\Lambda \backslash H, I')$, where $I'$ denotes the left-invariant complex structure $d{R_{h^{-1}}} \circ I \circ d R_h$ and we have $f'([e_G]) =[e_H]$. By our assumption the map $f'$ now lifts to a holomorphic group homomorphism $\Phi \colon (G,J) \to (H,{I'})$ and thus the map $R_h \circ \Phi\colon (G,J) \to (H,I)$ is a holomorphic lift of $f$. Hence, we have proved Theorem~\ref{thm: affine maps} under the assumption that it holds for maps sending $[e_G]$ to $[e_H]$. From now on we will assume that $f([e_G]) = [e_H]$.

	\subsubsection{Step 2}
	
	In this step, we will prove that the map $\widetilde{F}\colon \IC^n \to \IC^m$ is a polynomial map. Concretely, we will show that there exists a constant $C>0$ and $k \in \IN$ such that
	\[
	\Vert \widetilde{F}(z) \Vert \leq C(1+ \Vert z \Vert)^k
	\]
	for all $z \in \IC^n$. It then follows from a generalised Liouville theorem that $\widetilde F$ is a polynomial map of degree at most $k$ \cite[Section~6.3, p.~41]{Vladimirov66}.
	
	To this purpose, we endow $\IC^n$ with a multiplication ${\boldsymbol{\cdot}}$ such that the biholomorphism $P\colon (G,J) \to (\IC^n, {\boldsymbol{\cdot}})$ defined in Section \ref{subsect: universal cover} is a group homomorphism, that is, we set
	\begin{equation}\label{eq: multiplication on C^n}
		w {\boldsymbol{\cdot}} z = P (P^{-1}(w)P^{-1}(z))
	\end{equation}
	for $w,z \in \IC^n$. Similarly, we will denote by $*$ the multiplication making the biholomorphism $(H,I) \to (\IC^m, *)$ into a group homomorphism. Throughout this section, we will consider $\Gamma$ and $\Lambda$ as lattices in $(\IC^n, \boldsymbol{\cdot})$ and $(\IC^m, *)$, respectively.

	\begin{lem}\label{lem: F is equivariant}
		For every $\gamma \in \Gamma$ and $z \in \IC^n$ we have
		\[
		\widetilde F(\gamma {\boldsymbol{\cdot}} z) = \widetilde F(\gamma) * \widetilde F(z).
		\]
	\end{lem}
	\begin{proof}
		Let us fix an element $\gamma \in \Gamma$. Since $\widetilde F$ is a lift of $f$ we have a commutative diagram
		\[
		\begin{tikzcd}
			\IC^n \ar[r,"\widetilde F"]\ar[d, swap,"\pi"]&\IC^m\ar[d, "\pi'"]\\
			X\ar[r,"f"]& Y
		\end{tikzcd}
		\]
		and thus for every $z \in \IC^{n}$ there exists a $\lambda(z) \in \Lambda$ such that $\widetilde F(\gamma {\boldsymbol{\cdot}} z)= \lambda(z) * \widetilde F(z)$.  Consider the continuous map $ h_\gamma \colon \IC^n \to \IC^m$ defined by $h_\gamma(z) =  \widetilde F(\gamma {\boldsymbol{\cdot}} z) * (\widetilde F(z))^{-1}$. Then we have
		\[
		\pi'\circ h_\gamma(z) = \pi'(\lambda(z) * \widetilde F(z) * (\widetilde F(z))^{-1}) =  \pi'( \widetilde F(z) * (\widetilde F(z))^{-1}) = [0].
		\]
		Hence, the map $h_\gamma$ only takes values in the lattice $\Lambda$ and is therefore constant. Since we assumed $\widetilde F(0) =0$ this implies that $h_\gamma(z) = \widetilde F(\gamma)$ for every $z \in \IC^n$ which is equivalent to $\widetilde F(\gamma {\boldsymbol{\cdot}} z) = \widetilde F(\gamma) * \widetilde F(z)$.
	\end{proof}

	By \cite[Corollary~2.8, p.~45]{Vinberg00}, the homomorphism $\widetilde F_{|\Gamma} \colon \Gamma \to \Lambda$ extends uniquely to a homomorphism $\widetilde \Phi \colon (\IC^n, \boldsymbol{\cdot}) \to (\IC^m, *)$. The following lemma follows from the description of the biholomorphisms $P\colon (G,J) \to \IC^n$ and $P'\colon (H,I) \to \IC^m$ given in~\cite{Hasegawa26}.
	
	\begin{lem}\label{lem: maps are polynomial}
		The multiplication maps $\boldsymbol{\cdot} \colon \IC^n \times \IC^n \to \IC^n$, $* \colon \IC^m \times \IC^m \to \IC^m$, and the map $\widetilde \Phi$ are polynomial maps in the variables $z_i, \overline z_i$. 
	\end{lem}
	\begin{proof}
		It was shown in \cite{Hasegawa26} that the biholomorphisms $P$ and $P'$ are polynomial in exponential coordinates with polynomial inverse. Therefore, the corresponding multiplication maps defined by \eqref{eq: multiplication on C^n} are also polynomial. The map $\widetilde{\Phi}$ fits in the commutative diagram
		\[
		\begin{tikzcd}
			\frg \arrow[r, "d\Phi_{e_{G}}"] \arrow[d, "\exp", swap] & \frh \arrow[d,"\exp"] \\
			(G,J) \arrow[r, "\Phi"] \arrow[d, "P", swap] & (H,I) \arrow[d, "P'"] \\
			(\IC^n, \boldsymbol{\cdot}) \arrow[r, "\widetilde \Phi"]           & (\IC^m, *),        
		\end{tikzcd}
		\]
		where $\Phi$ is the unique extension of $\Phi_{|\Gamma}$. Again, since both $P$ and $P'$ are polynomial in exponential coordinates with polynomial inverse, the composition
		\[
		\widetilde \Phi = (P' \circ \exp) \circ d\Phi_{e_{G}} \circ (P \circ \exp)^{-1}
		\]
		is also a polynomial map.
	\end{proof}
	
	\begin{prop}\label{prop: polynomial growth}
		There exists a constant $C>0$ and $k \in \IN$ such that
		\[
		\Vert \widetilde{F}(z) \Vert \leq C(1+ \Vert z \Vert)^k
		\]
		for all $z \in \IC^n$. Hence, $\widetilde F$ is a polynomial map.
	\end{prop}
	\begin{proof}
		Since the lattice $\Gamma \subset (\IC^n, \boldsymbol{\cdot})$ is co-compact, there exists a compact domain $D \subset \IC^n$ such that $\Gamma \boldsymbol{\cdot} D = \IC^n$. Hence, for every $z \in \IC^n$, there is a $\gamma \in \Gamma$ and $z_0 \in D$ such that $z = \gamma \boldsymbol{\cdot} z_0$, and by Lemma \ref{lem: F is equivariant} we have
		\[
		\widetilde{F}(z) = \widetilde{F}(\gamma \boldsymbol{\cdot} z_0) = \widetilde F(\gamma) * \widetilde F(z_0).
		\]
		By Lemma \ref{lem: maps are polynomial}, the map $* \colon \IC^m \times \IC^m \to \IC^m$ is polynomial, and thus we can estimate 
		\begin{equation}\label{eq: first estimate}
			\Vert \widetilde{F}(z) \Vert = \Vert \widetilde F(\gamma) * \widetilde F(z_0) \Vert \leq c_1(1+ \Vert \widetilde{F}(\gamma)\Vert)^{d_1}(1+ \Vert \widetilde{F}(z_0)\Vert)^{d_1}
		\end{equation}
		for some $c_1 >0$ and $d_1 \in \IN$. Since $D$ is compact, the term $(1+ \Vert \widetilde{F}(z_0)\Vert)^{d_1}$ is bounded by some constant $c_2>0$ for every $z_0 \in D$. Moreover, we have $\widetilde F(\gamma) = \widetilde \Phi(\gamma)$ and $\widetilde \Phi$ is also a polynomial map, so we have
		\[
		\Vert \widetilde F(\gamma) \Vert =  \Vert \widetilde \Phi(\gamma) \Vert \leq c_3(1+ \Vert \gamma \Vert)^{d_3}
		\]
		for some $c_3 >0$ and $d_3 \in \IN$. Therefore, we can further estimate \eqref{eq: first estimate} by
		\begin{align}\label{eq: second estimate}
			\begin{split}
				c_1(1+ \Vert \widetilde{F}(\gamma)\Vert)^{d_1}(1+ \Vert \widetilde{F}(z_0)\Vert)^{d_1} &\leq c_1c_2(1+ \Vert \widetilde{\Phi}(\gamma)\Vert)^{d_1}\\
				&\leq c_1c_2(1+c_3)^{d_1}(1+ \Vert \gamma \Vert)^{d_1d_3}.
			\end{split}
		\end{align}
		So far, none of the constants depend on $z$, $z_0$ or $\gamma$. It remains to show that we can estimate $\Vert \gamma \Vert$ by the norm of $ z =  \gamma \boldsymbol{\cdot} z_0$. Again by Lemma \ref{lem: maps are polynomial}, the right-translation $R_{z_0^{-1}} \colon (\IC^n, \boldsymbol{\cdot}) \to (\IC^n, \boldsymbol{\cdot})$ is a polynomial map, so we have 
		\[
		\Vert \gamma \Vert = \Vert R_{z_0^{-1}}(z) \Vert \leq c(z_0)(1+ \Vert z \Vert)^{d(z_0)},
		\]
		where $c(z_0) >0$ and $d(z_0) \in \IN$ depend on $z_0$. However, the degree of the polynomials $R_{z_0^{-1}}$ is bounded by some $d_4 \geq d(z_0)$ and since $D$ is compact we can also find a constant $c_4>0$ with
		\begin{equation}\label{eq: third estimate}
			\Vert \gamma \Vert = \Vert R_{z_0^{-1}}(z) \Vert \leq c(z_0)(1+ \Vert z \Vert)^{d(z_0)} \leq c_4(1+ \Vert z \Vert)^{d_4}
		\end{equation}
		for all $\gamma \in \Gamma$ and $z_0 \in D$. Substituting \eqref{eq: third estimate} into \eqref{eq: second estimate} we obtain
		\[
		c_1(1+ \Vert \widetilde{F}(\gamma)\Vert)^{d_1}(1+ \Vert \widetilde{F}(z_0)\Vert)^{d_1} \leq c_1c_2(1+c_3)^{d_1}(1+c_4)^{d_1d_3}(1+ \Vert z\Vert)^{d_1d_3d_4}.
		\]
		Hence, for $C =c_1c_2(1+c_3)^{d_1}(1+c_4)^{d_1d_3}$ and $k =d_1d_3d_4$, the inequality \eqref{eq: first estimate} yields
		\[
		\Vert \widetilde{F}(z) \Vert \leq C(1+ \Vert z \Vert)^k,
		\]
		which concludes the proof.
	\end{proof}
	
	\subsubsection{Step 3}
	
	Finally, we will show that the map $\widetilde F$ coincides with the homomorphism $\widetilde \Phi \colon (\IC^n, \boldsymbol{\cdot}) \to (\IC^m, *)$. In this case, $F \colon G \to H$ is a composition of group homomorphisms and thus also a homomorphism. Recall that the maps $\widetilde{F}$ and $\widetilde \Phi$ already coincide on the lattice $\Gamma \subset \IC^n \cong \IR^{2n}$.
	
	\begin{lem}\label{lem: polynomials vanish on lattice}
		If a polynomial function $p \colon \IR^{2n} \to \IR$ vanishes on $\Gamma \subset \IR^{2n}$, then $p =0$.
	\end{lem}
	\begin{proof}
		Let $\frg$ be the Lie algebra of $(\IC^n, \boldsymbol{\cdot}) \cong G$ and let $\widetilde p \colon \frg \to \IR$ denote the composition
		\[
		\begin{tikzcd}
			\frg \arrow[r, "\exp"] &(G,J) \arrow[r, "P" ] &\IC^n \cong \IR^{2n} \arrow[r, "p"] & \IR.
		\end{tikzcd}
		\]
		The map $\widetilde p$ is again a polynomial map, which vanishes on the subset $\log(\Gamma)\subset \frg$. By \cite[Theorem~2.13, p.~48]{Vinberg00} the set $\log(\Gamma)$ contains an additive lattice $\Gamma' \subset \frg$. We claim that $\widetilde p =0$ because it vanishes on $\Gamma'$. Note that if $\widetilde p$ vanishes on infinitely many points on a given line $L \subset \frg$ through the origin, then $\widetilde p$ vanishes on $L$. Now, consider the set of lines $\kl = \{ L \subset \frg ß \, | \, \textup{ $L$ intersects $\Gamma'$ in infinitely many points} \}$. The map $\widetilde p$ vanishes on infinitely many points of every line in $\kl$ and thus $\widetilde p$  vanishes on $\bigcup_{L \in \kl} L$. Since $\Gamma' \subset \frg$ is an additive lattice, the union $\bigcup_{L \in \kl} L$ is a dense subset of $\frg$ and thus $\widetilde p =0$. As $P \circ \exp$ is a bijection this implies that $p =0$.
	\end{proof}
	
	By the results of the previous section, both $\widetilde F$ and $\widetilde \Phi$ can be considered as real polynomial maps on $\IR^{2n}$, so $\widetilde F-\widetilde \Phi$ is a polynomial vanishing on $\Gamma$, and Lemma \ref{lem: polynomials vanish on lattice} implies $\widetilde F = \widetilde \Phi$. This concludes the proof of Theorem \ref{thm: affine maps}.\qed
	
	\section{The space of left-invariant complex structures}\label{sect: Cg}

	Let  $\frg$ be a nilpotent Lie algebra of dimension $2n$ admitting complex structures. The goal of this section is to describe some properties of the space
	\[
	\kc(\frg) = \{ J \colon \frg \to \frg \, | \,  J^2= - \id, \,\textup{$J$ integrable}\}
	\]
	of complex structures on $\frg$. Recall that every almost complex structure $J$ on $\frg$ is uniquely determined by the $i$-eigenspace $\frg^{1,0}$ of $J_\IC \colon \frg_\IC \to \frg_\IC$. Hence, we can consider the set
	\[
	\ku_n = \{ V \in \Gr(n, \frg_\IC) \, | \, V \cap \overline V = 0 \}  \subset \Gr(n, \frg_\IC)
	\]
	as the space of almost complex structures on $\frg$. Note that $\ku_n$ defines an open subset of the Grassmannian $\Gr(n, \frg_\IC)$ and thus we can consider $\ku_n$ as a complex manifold of dimension $n^2$. Since a complex structure $J$ on $\frg$ is integrable if and only if the space $\frg^{1,0}$ is closed under the Lie bracket we will identify
	\[
	\kc(\frg) = \{ V \in \Gr(n, \frg_\IC) \, | \, V \cap \overline V =0, \; [V, V] \subset V\}.
	\]
	In this way, the space $\kc(\frg)$ can be viewed as an open subset of the variety
	\[
	\kv(\frg) = \{V \in \Gr(n, \frg_\IC) \, | \, [V,V] \subset V\}\subset \Gr(n,\frg_\IC),
	\]
	and as such has the structure of a complex space.
	
	To study some basic properties of this complex structure on $\kc(\frg)$, we start with a description of the tangent space at a point $J \in \kc(\frg)$: recall that the tangent space of the Grassmannian $\Gr(n, \frg_\IC)$ at a point $V$ can be identified with the space $\Hom(V, \frg_\IC/V)\cong V^* \otimes \frg_\IC/ V$ (see for instance \cite[Lemma~10.7]{Voisin_2002}). In particular, if $ V= \frg^{0,1}$ is the $(-i)$-eigenspace in $\frg_\IC$ associated to the complex structure $J$, then $\frg_\IC /\frg^{0,1} \cong \frg^{1,0}$ and thus we can describe the Zariski tangent space $T_J\kc(\frg)$ as a subspace of $\frg^{*0,1} \otimes \frg^{1,0}$ via the identification
	\[
	T_{\frg^{0,1}}\Gr(n,\frg_\IC) \cong \Hom(\frg^{0,1}, \frg^{1,0}) \cong \frg^{*0,1} \otimes \frg^{1,0}.
	\]
	Consider the map $\dbar \colon \frg^{*0,1} \otimes \frg^{1,0} \to \Wedge^2\frg^{*0,1} \otimes \frg^{1,0}$ associated to $J$ (see Section \ref{subsect: Lie algebra cohomology}). In \cite[Proposition~4.2]{salamon01} Salamon showed that if $J$ is a smooth point in $\kc(\frg)$, then the tangent space at $J$ is contained in the kernel of $\dbar$. More generally, one has
	
	\begin{lem}\label{lem: tangent space Cg}
		The Zariski tangent space of $\mathcal{C}(\mathfrak{g})$ at the point $J$ is $T_J\kc(\frg) =\ker \dbar$.
	\end{lem}
	\begin{proof}
		Let $\eta \in \mathfrak{g}^{*0,1}\otimes \mathfrak{g}^{1,0}$ be a tangent vector of $\Gr(n,\frg_\IC)$ at the point $\frg^{0,1}$. To compute the tangent space $T_J\mathcal{C}(\mathfrak{g})$ we work over the the dual numbers $\mathbb{C}[\varepsilon]/\varepsilon^2$. Consider an infinitesimal deformation $\mathfrak{g}_{\varepsilon}^{0,1}=\{z+\varepsilon\eta(z)\,|\, z\in \mathfrak{g}^{0,1}\}$ of $\mathfrak{g}^{0,1}$. Then the integrability condition $[\mathfrak{g}_{\varepsilon}^{0,1},\mathfrak{g}_{\varepsilon}^{0,1}]\subset \mathfrak{g}_{\varepsilon}^{0,1}$ is equivalent to the Maurer--Cartan equation
		\[
		\varepsilon\dbar\eta + \frac{\varepsilon^2}{2}[\eta,\eta] =0,
		\]
		which over $\mathbb{C}[\varepsilon]/\varepsilon^2$ is equivalent to $\dbar \eta =0$.
	\end{proof}

	Next, we want to deduce a criterion for a point in $\kc(\frg)$ to be reduced. To this purpose, it is often useful to consider $\kc(\frg)$ as a subspace of $\Gr(n, \frg_\IC^*)$ via the natural map $\Gr(n, \frg_\IC) \to \Gr(n, \frg_\IC^*)$ and describe it in terms of left-invariant differential forms.
	Recall that the Grassmannian $\Gr(n, \frg^*_\IC)$ can be embedded into $\IP\left(\Wedge^n\frg^*_\IC\right)$ via the Plücker embedding 
	\[
	\iota \colon \Gr(n,\frg_\IC^*) \hookrightarrow \IP\left(\Wedge^n\frg^*_\IC\right),
	\]
	which is given by mapping a vector space $V = \langle \omega^1, \dots, \omega^n \rangle$ to the class of the $n$-form $\omega^1 \wedge \dots \wedge \omega^n$ in $ \IP\left(\Wedge^n\frg_\IC^*\right)$. We will first describe the image of the restriction of $\iota$ to $\kc(\frg)$.

	\begin{lem}\label{lem: plücker image}
		A class $[\alpha] \in \IP\left(\Wedge^n\frg^*_\IC\right)$ is in the image of $\iota_{|\kc(\frg)}$ if and only if $\alpha$ is decomposable, $d$-closed, and $\alpha \wedge \overline \alpha \neq 0$.
	\end{lem}
	\begin{proof}
		Let $J$ be a complex structure on $\frg$. According to \cite[Theorem~1.3]{salamon01} there exists a basis $\omega^1, \dots, \omega^n$ of $\frg^{*1,0}$ such that $d\omega^{k+1} \in I(\omega^1, \dots , \omega^k)$, where $I(\omega^1, \dots , \omega^k)$ is the ideal in $\Wedge^2\frg_\IC^*$ generated by $\omega^1, \dots, \omega^k$. This implies that the form $\omega = \omega^1 \wedge \dots \wedge \omega^n$ is closed and since $ \frg^{*1,0} \cap \frg^{*0,1} =0$, the $2n$-form $\omega \wedge \overline{\omega}$ is non-zero.
		
		Conversely, we have to show that every closed decomposable form $\alpha = \alpha^1 \wedge \dots \wedge \alpha^n$ satisfying $\alpha \wedge \overline{\alpha} \neq 0$ determines a complex structure on $\frg$. Setting $V = \langle \alpha^1, \dots, \alpha^n \rangle$, the condition $\alpha \wedge \overline{\alpha} \neq 0$ implies that $V \oplus \overline V = \frg^*_\IC$ and thus $\alpha$ uniquely determines an almost complex structure on $\frg$. It remains to show that the almost complex structure given by $V$ is integrable, which is equivalent to the map
		\[
		\begin{tikzcd}
			V \arrow[r, "d"] &\Wedge^2\frg_\IC^* \arrow[r, "\pi^{0,2}"] &\Wedge^{2} \overline V
		\end{tikzcd}
		\] 
		being the zero-map. Let $z_1, \dots, z_n$ be the dual basis of $\alpha^1, \dots, \alpha^n$ and consider for $i \in \{1,\dots,n\}$ the map $c_i \colon \Wedge^{n+1} \frg_\IC^* \to \Wedge^2\frg^*_\IC$ which sends a form $\beta \in \Wedge^{n+1}\frg_\IC^*$ to $c_i(\beta) = (z_1 \wedge \dots \wedge \widehat{z_i} \wedge \dots z_n) \lrcorner \beta$. Then for every $\alpha^i$ we have
		\[
		\pi^{0,2} \circ d(\alpha^i) =  \pi^{0,2}\circ c_i (d\alpha) = 0
		\]
		since we assumed $\alpha$ to be closed. Hence, the space $V$ defines a complex structure on the Lie algebra $\frg$.
	\end{proof}
	
	\begin{rem}
		Considering $\kc(\frg)$ as a subspace of $\IP\left(\Wedge^n\frg^*_\IC\right)= \IP^{\binom{2n}{n}-1}$ via the Plücker embedding, the previous lemma shows that $\kc(\frg) = \ku_n \cap \IP \ker d$. This description is often useful to determine the connected components of $\kc(\frg)$. By the Fulton–Hansen connectedness theorem the variety $\kv(\frg) = \Gr(n, \frg_\IC^*) \cap \IP \ker d$ is connected if $n^2 + \dim \ker d > \binom{2n}{n}$ \cite[Corollary~1]{Fulton79}. Moreover, it follows from a Lefschetz-type theorem due to Sommese that the homotopy groups $\pi_j(\kv(\frg))$ coincide with  $\pi_j(\Gr(n, \frg_\IC^*))$ for $j\leq n^2 + \dim \ker d -\binom{2n}{n}-1$ \cite[Corollary~3.5]{Sommese82}. A comparison with \cite[Table~A.1]{salamon01} shows that $\kv(\frg)$ is connected and simply connected for nilpotent Lie algebras up to dimension six. 
	\end{rem}

	As explained above, up to scaling, the Plücker embedding associates to every complex structure $J$ on $\frg$ a form $\omega \in \Wedge^n\frg_\IC^*$. For every element $\eta \in \ker \dbar = T_J \kc(\frg)$ the form $\eta\lrcorner\omega$ is a $\dbar$-closed $(n-1,1)$-form. 
	
	\begin{prop}\label{prop: reduced points in Cg}
		If $J$ is a reduced point of $\kc(\frg)$, then $d (\eta \lrcorner \omega) =0$ for every $\eta \in \ker \dbar$, or equivalently, every $\dbar$-closed form in $\Wedge^{n-1,1}\frg^*$ is $d$-closed.
	\end{prop}
	\begin{proof}
		The differential of the Plücker embedding 
		\[
		d\iota_J\colon \frg^{*0,1} \otimes \frg^{1,0} \to \Hom\left(\langle \omega \rangle , \Wedge^n\frg_\IC^* /\langle\omega \rangle\right) = T_{\langle \omega \rangle}  \IP\left(\Wedge^n\frg_\IC^*\right)
		\]
		at the point corresponding to $J$ sends $\eta \in \frg^{*0,1} \otimes \frg^{1,0}$ to the homomorphism given by $\left(\omega \mapsto [\eta \lrcorner \omega]\right)$. Suppose $\kc(\frg)$ is reduced at $J$. Then $\kc(\frg)$ is reduced in an open neighbourhood $U$ of $J$ on which the map $\iota_{|U} \colon U \to \IP \ker d$ is an embedding of complex spaces. 
		By Lemma \ref{lem: plücker image} the image of the restriction of the Plücker embedding to $\kc(\frg)$ is contained in $\IP\ker d\subset \IP\left(\Wedge^n\frg^*_\IC\right)$. Hence, the image of $(d\iota_{|U})_J$ at the point $J$ is contained in $\Hom\left(\langle \omega \rangle , \ker d /\langle\omega \rangle\right)$.  Thus, if $\eta \in \ker \dbar = T_J\kc(\frg)$, then  $d(\eta \lrcorner \omega) =0$.
	\end{proof}

	\begin{rem}\label{rem: complex parallelisable reduced}
		The necessary condition given in Proposition \ref{prop: reduced points in Cg} is not sufficient for $J$ to be a reduced point. For instance, if $J$ is a complex parallelisable structure, that is, the pair $(\frg, J)$ is a complex Lie algebra, then the condition is always satisfied since in this case $\ker \dbar = \ker d \otimes \frg^{1,0}$. But the Kuranishi space of a complex parallelisable nilmanifold is often not reduced (see \cite[Table~1]{Rol11Kuranishi}).
	\end{rem}

	\begin{exam}\label{exam: non-reduced point}
		Consider the real Lie algebra $\frh$ underlying the Iwasawa manifold (see Example \ref{exam: Iwasawa fibres}). By \cite[Lemma~2.11]{ceballos2016invariant} there exists a complex structure $J$ on $\frh$ with complex structure equations 
		\[d\omega^1 = d\omega^2 =0, \; d\omega^3 = \omega^{1\bar 1} + \omega^{1\bar 2}.\]
		If $z_1, z_2, z_3 \in \frh^{1,0}$ is the dual basis of $\omega^1, \omega^2, \omega^3$, then
		\[
		\dbar(\overline \omega^3 \otimes z_2) = \dbar\overline \omega^3 \otimes z_2 - \overline \omega^3 \otimes \dbar z_2 = 0.
		\]
		Hence, $\eta = \overline \omega^3 \otimes z_2$ defines a tangent vector of $\kc(\frh)$ at $J$ but $d(\eta \lrcorner \omega^{123}) = d\omega^{13\bar 3} =\omega^{123\bar 1}$ is non-zero. Therefore, the complex space $\kc(\frh)$ is not reduced and in particular not smooth at $J$.
	\end{exam}

	Alternatively, we can describe the behaviour of a point $J \in \kc(\frg)$ in terms of the Kuranishi space of a nilmanifold with complex structure $J$.
	Let $X = (\Gamma \backslash G, J)$ be a complex nilmanifold with underlying Lie algebra $\frg$. As explained in Section \ref{subsubsect: Kuranishi theory}, after fixing a left-invariant hermitian metric $h$ on $X$ we obtain a distinguished representative of every class in $\Kur(X)$ and we can consider $\Kur(X)$ as a subspace around $0$ in the space $\kh^1(\frg^{0,1}, \frg^{1,0}) = \kh^1(X, \Theta_X)$. In particular, we have an embedding $\Kur(X) \hookrightarrow \kc(\frg)$ depending on the choice of the hermitian metric. Let $H^0(\frg^{0,1}, \frg^{1,0})^{\perp} \subset \frg^{1,0}$ be the orthogonal complement to the space of left-invariant holomorphic vector fields on $X$ with respect to $h$. Then, for every $z\in H^0(\frg^{0,1}, \frg^{1,0})^{\perp}$ and $I\in \Kur(X)$ we obtain a new left-invariant structure 
	\[
	R_{\exp(z+\bar{z})_*} I = dR_{\exp(z+\bar{z})}^{-1} \circ I \circ dR_{\exp(z+\overline{z})}.
	\]
	Hence, we have a map
	\[
	\psi_h \colon \Kur(X) \times H^0(\frg^{0,1}, \frg^{1,0})^{\perp} \to \kc(\frg), \qquad (I,z) \mapsto R_{\exp(z+\overline{z})_*} I
	\]
	depending on the choice of $h$.

	\begin{prop}\label{lem: Kur inside Cg}
		If every complex nilmanifold of type $(\frg, \Gamma)$ is unobstructed, then every connected component of $\kc(\frg)$ is a complex manifold.
	\end{prop}
	\begin{proof}
		We will show that the differential of $\psi_h$ at the point $(0,0) \in \Kur(X) \times H^0(\frg^{0,1}, \frg^{1,0})^\perp$ is an isomorphism for every unobstructed complex nilmanifold $X$. To this purpose we consider $\Kur(X)\subset \kc(\frg)$ as a subset of $ \IP\left(\Wedge^n\frg_\IC^*\right)$ via the Plücker embedding. Let $\omega \in \Wedge \frg_\IC^*$ be a left-invariant $(n,0)$-form on $X$ and let $\{X_t\}$ be a one-parameter family of deformations of $X$ given by $\phi(t) = \sum t^i \phi_i \in \frg^{0,1*} \otimes \frg^{1,0}$. Then the form
		\[
		\phi(t)(\omega) = \omega + t (\phi_1 \lrcorner \omega) + t^2( \phi_1 \lrcorner( \phi_1 \lrcorner \omega) + \phi_2 \lrcorner \omega) +  \textup{ higher order terms}
		\]
		is a left-invariant $(n,0)$-form on $X_t$ representing the class in $ \IP\left(\Wedge^n\frg_\IC^*\right)$ corresponding to $X_t$. Hence, for $z \in H^0(\frg^{0,1}, \frg^{1,0})^\perp$ the map $\gamma(t) = (\phi(t)(\omega), tz)$ defines a path in $\Kur(X) \times H^0(\frg^{0,1}, \frg^{1,0})^\perp$ with $\frac{d}{dt}_{|t=0}\gamma(t) = (\phi_1, z) \in \kh^1(X, \Theta_X)\times H^0(\frg^{0,1}, \frg^{1,0})^{\perp}$. On the image of $\gamma$ the map $\psi_h$ is given by mapping $\gamma(t)$ to the class of the $n$-form
		\begin{align*}
			\psi_h(\gamma(t)) = R_{\exp(t(z+ \overline z))}^*(\phi(t)(\omega)) = R_{\exp(t(z+ \overline z))}^*\omega +  t(\phi_1 \lrcorner \omega) + \textup{ higher order terms},
		\end{align*}
		which corresponds to a left-invariant complex structure in $\kc(\frg)$. To compute the derivative $\frac{d}{dt}_{|t=0}\psi_h(\gamma(t))$ we first observe 
		\[
		\frac{d}{dt}_{|t=0} R_{\exp(t(z+ \overline z))}^*\omega = L_{z+\overline z}\omega = d((z+\overline z) \lrcorner \omega) = d (z \lrcorner \omega) =  \dbar (z \lrcorner \omega) =\dbar z \lrcorner \omega.
		\]
		The second to last equality holds because $\partial\left(\Wedge^{n-1,0}\frg^*\right) =0$ and for the last equality we use the general formula $\dbar(z \lrcorner \omega) = z \lrcorner (\dbar \omega) + (\dbar z) \lrcorner \omega$ \cite[Lemma~2.4 (1)]{Xia22} and the fact that $\omega$ is closed. Thus, we have
		\begin{align*}
			\frac{d}{dt}_{|t=0}\psi_h(\gamma(t)) &= \frac{d}{dt}_{|t=0}(R_{\exp(t(z+ \overline z))}^*(\phi(t)(\omega))\\
			&=\frac{d}{dt}_{|t=0} R_{\exp(t(z+ \overline z))}^*\omega +  \phi_1 \lrcorner \omega\\
			&= (\phi_1 +\dbar z)\lrcorner \omega.
		\end{align*}
		Therefore, the differential 
		\begin{equation*}
			{d\psi_h}_{(0,0)} \colon \kh^1(X, \Theta_X)\times H^0(\frg^{0,1}, \frg^{1,0})^{\perp} \to \ker \dbar
		\end{equation*}
		of $\psi_h$ at $(0,0) \in \Kur(X) \times H^0(\frg^{0,1},\frg^{1,0})^\perp$ is given by $(\eta, z) \mapsto \eta +\dbar z$, which is an isomorphism. 
	\end{proof}

	As a consequence of \cite[Theorem~1.3]{salamon01}, it was shown in \cite[Theorem~2.7]{Barberis09} that every complex nilmanifold of dimension $n$ admits a (non-trivial) closed left-invariant $(n,0)$-form. In particular, the canonical bundle of a complex nilmanifold is trivial. Let us give a more geometric interpretation of Example \ref{exam: non-reduced point}.
	
	\begin{exam}\label{exam: non closed deformation}
		Consider again the Lie algebra with complex structure $(\frh, J)$ given in Example \ref{exam: non-reduced point}. By the criterion of Malcev \cite{malcev51}, the simply connected Lie group $H$ associated to $\frh$ admits a lattice $\Gamma$, so we can consider the complex nilmanifold $Y =(\Gamma \backslash H, J)$. Then the left-invariant form $\omega^{123} \in \Wedge^{3,0}\frh^*$ defined in Example \ref{exam: non-reduced point} yields a trivialising section of the canonical bundle of $Y$. We already saw that the class $[\eta] = [\overline \omega^3 \otimes z_2] \in H^1(\frh^{0,1}, \frh^{1,0}) = H^1(Y, \Theta_Y)$ defines an infinitesimal deformation of $Y$ such that the form $\eta \lrcorner \omega^{123}$ is not closed. Hence, there are infinitesimal deformations of $Y$ which no longer have trivial canonical bundle. This phenomenon will be discussed in more detail in Section \ref{sect: period maps}.
	\end{exam}

	\section{Teichmüller and moduli spaces}\label{chap: Teichmüller and moduli spaces}

	In this section, we formulate the moduli problems we aim to solve. We start by introducing the main objects of our study: the Teichmüller and moduli space of complex nilmanifolds of a given type. We also recall the notions of coarse and fine moduli spaces in our setting and discuss some existence criteria.

	\subsection{Definitions}\label{subsect: definitions}
	
	Let $M$ be a smooth compact connected manifold admitting complex structures. We denote by $\Diff(M)$ the group of diffeomorphisms of $M$ and by $\Diff^0(M)$ the connected component of the identity in $\Diff(M)$, that is, the group of diffeomorphisms of $M$ smoothly isotopic to the identity. There is an action of these groups on the space 
	\[
	\kc(M) = \{ J \in \End(TM) \, | \, J^2 = -\id, \; J \text{ integrable}\}
	\]
	of complex structures on $M$ given by
	\begin{equation}\label{eq: action of Diff}
		(f,J) \mapsto f_*J= df^{-1} \circ J \circ df. 
	\end{equation}
	The so-called \textit{Teichmüller space} of $M$ is defined as the quotient
	\[
	\kt(M) = \kc(M)/ \Diff^0(M)
	\]
	and the \textit{moduli space} of complex structures on $M$ is defined as
	\[
	\km(M) = \kc(M) / \Diff(M).
	\]
	Alternatively, one can define $\km(M)$ as the quotient of the Teichmüller space by the mapping class group $\Diff(M)/ \Diff^0(M)$. We consider both the Teichmüller and the moduli space as topological spaces equipped with the quotient topology of $\kc(M)$. 
	
	The connected components of the Teichmüller space are given by the deformation classes in the large. Concretely, two complex structures $J$ and $J'$ on $M$ correspond to points in the same connected component of $\kt(M)$ if and only if the complex manifold $(M,J)$ is a deformation of $(M,{J'})$ in the large \cite[Corollary~6]{catanese2011superficial}. In general, it is a difficult task to identify connected components of $\kt(M)$ and many examples where this is known to be possible belong to the class of nilmanifolds \cite{Cat04DefLargeI, rollenske09, andrada2025almost}. 
	
	From now on $M = \Gamma \backslash G$ will denote a nilmanifold with underlying Lie algebra $\frg$ and we will consider $\kc(\frg)  \subset \kc(M)$ as the space of left-invariant complex structures on $M$. We would like to restrict the notions of the Teichmüller and moduli space to the space $\kc(\frg)$. However, the action \eqref{eq: action of Diff} does not preserve $\kc(\frg)$, that is, for an arbitrary diffeomorphism $f\in \Diff(M)$ and a complex structure $J\in \kc(\frg)$ the complex structure $f_*J$ is not necessarily left-invariant. We define the Teichmüller space of complex nilmanifolds of type $(\frg, \Gamma)$ as 
	\[
	\kt(\frg) = \kc(\frg) / \sim,
	\]
	where the relation $\sim$ identifies two complex structures $J$ and $J'$ in $\kc(\frg)$ if there exists a diffeomorphism $f \in \Diff^0(M)$ such that $f_*J = J'$. Analogously, we define the moduli space of complex nilmanifolds of type $(\frg, \Gamma)$ as
	\[
	\km(\frg, \Gamma) = \kc(\frg)/ \cong,
	\]
	where $J\cong J'$ if there exists an $f\in \Diff(M)$ such that $f_*J = J'$. 
	
	As a consequence of Theorem \ref{thm: affine maps}, the spaces $\kt(\frg)$ and $\km(\frg, \Gamma)$ can also be described as quotients of finite-dimensional group actions: by associating to $g\in G$ the induced right-translation $R_g$ on $M= \Gamma \backslash G$ we can consider the group $G$ as a subgroup of $\Diff^0(M)$. Similarly, given an automorphism $\varphi \in \Aut(\Gamma)$ we can uniquely extend $\varphi$ to an automorphism of $G$ \cite[Corollary~2.8, p.~45]{Vinberg00} which then descends to a diffeomorphism of $M$. Hence, we obtain actions of both $G$ and $\Aut(\Gamma)$ on $\kc(M)$ defined as in \eqref{eq: action of Diff}.
	
	\begin{lem}
		The actions of $G$ and $\Aut(\Gamma)$ on $\kc(M)$ preserve the space $\kc(\frg)$.
	\end{lem}
	\begin{proof}
		Let $J \in \kc(\frg)$, viewed as a left-invariant complex structure on $G$. Then for all $h,g\in G$ the maps $L_h$ and $R_g$ commute. Therefore, we have ${L_h}_*({R_g}_*J) ={R_g}_*({L_h}_*J) = {R_g}_*J$ and thus ${R_g}_*J \in \kc(\frg)$. Let $\varphi \in \Aut(\Gamma)$ and let $\Phi$ be its unique extension to $G$. Then $L_h \circ \Phi = \Phi \circ L_{\Phi^{-1}(h)}$ for all $h \in G$ and thus $\Phi_*J$ is again left-invariant.
	\end{proof}
	
	Note that the action of $\Aut(\Gamma)$ on $\kc(\frg)$ descends to the quotient $\kc(\frg)/G$. The following corollary of Theorem \ref{thm: affine maps} gives a greatly simplified description of the Teichmüller and moduli spaces.
	
	\begin{cor}\label{cor: simple description}
		We have $\kt(\frg) = \kc(\frg)/G$ and $\km(\frg, \Gamma) = \kt(\frg)/\Aut(\Gamma)$.
	\end{cor}
	\begin{proof}
		According to Theorem \ref{thm: affine maps}, every holomorphic map $f \colon (M, J) \to (M,I)$ has a lift $F \colon (G,J) \to (G,I)$ of the form $F= R_g \circ \Phi$, where $R_g$ is a right-translation on $G$ and $\Phi \colon G \to G$ is a group homomorphism preserving the lattice $\Gamma \subset G$. Therefore, it only remains to show that $F= R_g$ if $f$ is isotopic to the identity on $M$.  Note that the restriction $\Phi|_{\Gamma}\colon \Gamma \rightarrow \Gamma$ coincides with the induced map $f_*$ on $\pi_1(M) \cong \Gamma$. Hence, if $f$ is isotopic to the identity, then $f_*= \Phi|_{\Gamma}$ has to be the identity, but since $\Phi$ is already determined by its restriction to $\Gamma$ this implies that $\Phi$ is the identity.
	\end{proof}

	\begin{rem}
		In \cite[Theorem~1]{Hasegawa05Solv} Hasegawa showed that every complex structure on a solvmanifold of dimension four or less is isomorphic to a left-invariant one. We already saw that the only complex nilmanifolds up to dimension two are complex tori and Kodaira surfaces (see Example \ref{exam: Kodaira surface}). The real nilmanifold underlying a Kodaira surface is also called a \textit{Kodaira--Thurston manifold}. One consequence of Hasegawa's result is that the Teichmüller space of all complex structures on such a manifold has a finite number of connected components.
		
		In higher dimensions the situation becomes much more complicated. For instance, it was shown in \cite[Corollary~7.7]{Cat04DefLargeI} that the Teichmüller space $\kt(T^6)$ of the real six-dimensional torus $T^6$ already has countably many connected components of unbounded dimension. Therefore, we will restrict our considerations to the spaces $\kt(\frg)$ and $\km(\frg, \Gamma)$ which, as mentioned above, will often make up a union of connected components of $\kt(M)$ and $\km(M)$, respectively.
	\end{rem}

	\begin{rem}\label{rem: orientation}
		Instead of considering the space $\kc(M)$ of all complex structures on $M$ one could fix an orientation of $M$ and consider the the space $\kc^+(M)$ of complex structures compatible with this orientation. In this case, one defines the Teichmüller space $\kt^+(M)$ as above and for the moduli space $\km^+(M)$ one restricts the action to the group $\Diff^+(M)$ of orientation-preserving diffeomorphisms. The difference between these approaches depends on the existence of an orientation-reversing diffeomorphism, as explained in \cite[Remark~2.1]{Meersseman19}: if $M$ admits no orientation-reversing diffeomorphism, then $\kt(M)$ and $\km(M)$ are disjoint unions of the oriented Teichmüller and moduli spaces of both orientations. If $M$ does admit an orientation-reversing diffeomorphism, then $\km^+(M) = \km(M)$, but $\kt(M)$ has twice as many connected components as $\kt^+(M)$. 
		We will usually not impose an orientation on $M$.
	\end{rem}

	\begin{rem}
		Leaving aside the lattice $\Gamma$ for a moment and considering $\kc(\frg)$ as the space of complex structures on the Lie algebra $\frg$, we have an action of the automorphism group $\Aut(\frg)$ on $\kc(\frg)$ given by $(\Phi, J) \mapsto \Phi^{-1} \circ J \circ \Phi$. So the quotient space $\kc(\frg)/\Aut(\frg)$ parametrises equivalence classes of complex structures on $\frg$ up to Lie algebra automorphism. Note that there is a map from $\kt(\frg)$ to $\kc(\frg)/\Aut(\frg)$ since we can consider $\kt(\frg)$ as the quotient of $\kc(\frg)$ by the inner automorphisms of $\frg$. By taking the differential of an automorphism of $G$ induced by an automorphism of $\Gamma$, we can also consider $\Aut(\Gamma)$ as a subgroup of $\Aut(\frg)$. Hence, we have the quotient maps
		\[	
		\begin{tikzcd}
			\kt(\frg) \arrow[r] \arrow[rd] & \km(\frg, \Gamma) \arrow[d ]\\
			& \kc(\frg)/\Aut(\frg).
		\end{tikzcd}
		\]
		While there is not much known about the spaces $\kt(\frg)$ and $\km(\frg, \Gamma)$, the space $\kc(\frg)/\Aut(\frg)$ was already described for all Lie algebras up to dimension six in \cite{ceballos2016invariant} and for some Lie algebras of dimension eight in \cite{Latorre23}. However, the equivalence relation given by the $\Aut(\frg)$-action is usually too coarse to study $\kt(\frg)$ or $\km(\frg, \Gamma)$ via its map to $\kc(\frg)/\Aut(\frg)$. For instance, if $M$ is a Kodaira--Thurston manifold with its underlying Lie algebra $\frg$, then $\kc(\frg)/\Aut(\frg)$ is just a point while both $\kt(\frg)$ and $\km(\frg, \Gamma)$ admit the structure of a complex two-dimensional manifold (cf. \cite{Borcea84} or Section \ref{subsect: Kodaira manifolds}).
	\end{rem}

	\subsection{Fine and coarse moduli spaces}\label{sect: coarse and fine moduli spaces}
	Let $M =\Gamma \backslash G$ be a real nilmanifold of dimension $2n$ with underlying Lie algebra $\frg$.
	
	The topological spaces $ \kt(\frg)$ and $\km(\frg, \Gamma)$ parametrise left-invariant complex structures on $M$ up to different notions of equivalence. To give a precise meaning to the question whether there exists a natural complex structure on $\kt(\frg)$ or $\km(\frg, \Gamma)$ we first recall the notions of (topological) coarse and fine moduli spaces.
	
	To formally define a moduli problem for a class $\kc$ of objects in some category, in our case the category of (reduced) complex spaces, one needs the notion of a family of these objects and a notion of equivalence of such families. The corresponding \textit{moduli functor} is then defined as the contravariant functor $\kg$ from the category of complex spaces to the category of sets, which associates to a complex space $B$ the set of equivalence classes of families over $B$.
	
	A  complex space $S$ is said to be a \textit{fine moduli space} for a moduli functor $\kg$ if there exists a universal family $\kx \to S$ in $\kg(S)$, that is, for every family $\ky \to T$ in $\kg(T)$ there exist unique morphisms $f \colon T \to S$ and $F \colon \ky \to \kx$ such that
	\[
	\begin{tikzcd}
		\ky\arrow[r,"F"]\ar[d]&\kx\arrow[d]\\
		T\arrow[r, "f"]&S
	\end{tikzcd}
	\]
	is a cartesian diagram. In this case, the complex space $S$ represents the functor $\kg$. If a fine moduli space exists it is unique up to unique isomorphism. However, due to the existence of non-trivial automorphisms a fine moduli space will often not exist. This phenomenon already occurs for moduli of elliptic curves, as described in \cite[Example~8.4.1]{PeriodMappings17}. Thus, one is usually concerned with the more general notion of a coarse moduli space, we recall here the definition given in \cite[Chapter~7.1]{lange2012complex}.
	
	\begin{defin}\label{def: coarse moduli}
		A \textit{coarse moduli space} for a moduli functor $\kg$ is a complex space $S$ together with a natural transformation
		\[
		\Psi: \kg \to \Mor(\,\cdot\,, S)
		\]
		satisfying the following properties:
		\begin{enumerate}
			\item for a point $p$ the map $\Psi(p) \colon \kg(p) \to \Mor(p,S)$ is bijective, that is, the isomorphism classes of objects in $\kc$ are in bijection with points in $S$;
			\item for any complex space $T$ and every natural transformation $\Psi' \colon \kg \to \Mor(\, \cdot \, ,T)$, there exists a unique morphism of complex spaces $\psi\colon S\to T$ such that the associated diagram 
			\begin{equation}\label{eq: diagram coarse moduli}
				\begin{tikzcd}
					\kg \arrow[r, "\Psi"] \arrow[rd,swap,"\Psi'"] & \Mor(\, \cdot \,, S) \arrow[d,"\psi \circ" ]\\
					& \Mor(\,, \cdot \, T)
				\end{tikzcd}
			\end{equation}
			commutes. 
		\end{enumerate}
	\end{defin}
	
	Similarly, a \textit{topological coarse moduli space} for $\kg$ is a topological space $T$ together with a natural transformation $\Phi \colon  \kg \to C(\, \cdot \,, T)$ to the set of continuous maps to $T$ satisfying properties $(i)$ and $(ii)$ in the previous definition with all morphisms being just continuous maps. It immediately follows from property $(ii)$ that a (topological) coarse moduli space is unique up to unique isomorphism if it exists. 
	
	We will mostly be concerned with the following two moduli functors associated to a class $\kc \subset \kc(\frg)$ of left-invariant complex structures on the nilmanifold $M$.
	\begin{itemize}
		\item The functor $\kf$ associates to a complex space $B$ the set of equivalence classes of holomorphic families of complex nilmanifolds in $\kc$. Two families $\pi\colon \kx \to B$ and $\pi'\colon \kx' \to B$ are equivalent if there exists a biholomorphism $F \colon \kx \to \kx'$ such that the diagram
		\begin{equation}\label{eq: diagram equivalent families}
			\begin{tikzcd}
				\kx \arrow[rr,"F"] \arrow[dr, swap, "\pi"] & &\kx' \arrow[dl, "\pi'"]\\
				& B
			\end{tikzcd}
		\end{equation}
		commutes. To any morphism of complex spaces $f\colon B' \to B$ we define the map $\kf(f) \colon \kf(B) \to \kf(B')$ by sending a family $\kx \to B$ to the pull-back family $f^*\kx \to B'$.
		\item The functor $\kf^0$ associates to a complex space $B$ the set of equivalence classes of pairs
		\[
		(\pi \colon \kx \to B,\, [\phi\colon M \to \kx]),
		\]
		where $\pi$ is a holomorphic family of complex nilmanifolds in $\kc$ and $[\phi\colon M \to \kx]$ is a $\Diff^0(M)$-framing of the family $\pi$. For a morphism $f \colon B' \to B$ the map $\kf^0(f) \colon \kf^0(B) \to \kf^0(B')$ is given by sending the pair $(\kx \to B,[\phi\colon M \to \kx])$ to $(f^*\kx \to B,[M \xrightarrow{\phi} \kx \to f^*\kx])$. 
	\end{itemize}

	Following \cite{Meersseman19}, we briefly recall the notion of a $\Diff^0(M)$-framing, also known as a Teichmüller structure (see \cite[Section~3]{Arbarello09}), for a given family. A $\Diff^0(M)$-framing for a family $\pi\colon \kx \to B$ over a connected base $B$ is given by the isotopy class of a map $M \to \kx$, which is a diffeomorphism onto a fibre of $\pi$. Here, an isotopy is a smooth map $H \colon M \times [0,1] \to \kx$ such that for every $t\in [0,1]$ the map $H_t = H(\, \cdot\,,t)$ defines a diffeomorphism of $M$ to a fibre of $\pi$. If $H$ is such an isotopy and both $H_0$ and $H_1$ map to the same fibre of $\pi$, then the diffeomorphism $H_1^{-1} \circ H_0$ of $M$ lies in $\Diff^0(M)$.
	Now, suppose $H$ and $H'$ are isotopies with $H_0 = H_0' = \phi$ and such that $H_1$ and $H_1'$ are diffeomorphisms to the same fibre of $\pi$. Then, $H_1$ and $H_1'$ differ by the diffeomorphism $H_1^{-1} \circ H_1' = (H_1^{-1} \circ H_0) \circ (H_0^{-1} \circ H_1') \in \Diff^0(M)$.
	Hence,  a $\Diff^0(M)$-framing gives an identification of the fibres of $\pi$ with $M$, which is defined up to a diffeomorphism in the identity component $\Diff^0(M)$. We say that two families $(\pi\colon \kx \to B, [\phi\colon M \to \kx])$ and $(\pi'\colon \kx'\to B, [\phi'\colon M \to \kx'])$ with a $\Diff^0(M)$-framing are equivalent if there exists a biholomorphic map $F\colon \kx \to \kx'$ such that the above diagram \eqref{eq: diagram equivalent families} commutes and the framing $[F \circ \phi\colon M \to \kx']$ coincides with $[\phi'\colon M \to \kx']$.

	\begin{rem}\label{rem: framing Kuranishi family}
		After choosing a hermitian metric on a complex nilmanifold $X_0=(M,J_0)$ the Kuranishi family of $X_0$ comes with a $\Diff^0(M)$-framing: on the product $M \times \Kur(X_0)$ there exists the structure of a complex space $\kk$ such that the projection  $\pi\colon\kk \to \Kur(X_0)$ is holomorphic and $\pi^{-1}(t) = (M,J_t)$, where $J_t$ is the left-invariant complex structure corresponding to the power series solution \eqref{eq: Green} with respect to the chosen metric. 
	\end{rem}

	\subsection{Existence}\label{subsect: existence criteria}
	
	Our first goal will be to show that the spaces $\km(\frg, \Gamma)$ and $\kt(\frg)$ always form a topological coarse moduli space for the functors $\kf$ and $\kf^0$, respectively, when they are defined for the class $\kc = \kc(\frg)$ of all complex nilmanifolds of type $(\frg, \Gamma)$. This is well-known for the moduli space of complex tori, i.e., when $\frg$ is abelian (see \cite[Theorem~3.2, p.~216]{lange2012complex}). By the universal property of coarse moduli spaces, this allows us to give some non-existence criteria for a complex moduli space depending on the topology of $\km(\frg, \Gamma)$ and $\kt(\frg)$.
	First, we have to construct suitable natural transformations $ \Psi \colon \kf \to C(\, \cdot \,, \km(\frg, \Gamma))$ and  $\Psi^0 \colon \kf^0 \to C(\, \cdot \,, \kt(\frg))$. We will begin with the construction of $\Psi$.

	Let $\pi \colon \kx \to B$ be a family of complex nilmanifolds. Then we can define a map  $\tau_\pi\colon B \to \km(\frg, \Gamma)$ by mapping $b\in B$ to the isomorphism class representing the complex nilmanifold $\pi^{-1}(b)$. Note that this does not give a well-defined map to $\kc(\frg)$ or $\kt(\frg)$ since the identification of the fibre $\pi^{-1}(b)$ with a pair $(M,J_b)$, where $J_b$ is a left-invariant complex structure on $M$, is only defined up to diffeomorphism.
	\begin{lem}\label{lem: tau is continuous}
		The map $\tau_\pi$ is continuous.
	\end{lem}
	\begin{proof}
		Let $b\in B$ and denote $X_b = \pi^{-1}(b)$. Then there exists a small neigbourhood $U$ of $b$ in $B$ and a holomorphic map $f\colon U \to \Kur(X_b)$ such that the restriction of the family $\pi$ to $U$ is equivalent to the pull-back of the Kuranishi family of $X_b$ along $f$. After choosing a hermitian metric on $X_b$ the Kuranishi space $\Kur(X_b)$ can be embedded into $\kc(\frg)$. By composing this embedding with $f$ we obtain a holomorphic map $\sigma \colon U \to \kc(\frg)$, which depends both on the map $f$ and the choice of hermitian metric on $X_b$. However, after composing $\sigma$ with the projection $\kc(\frg) \to \km(\frg, \Gamma)$ we obtain a  continuous map $U \to \km(\frg, \Gamma)$ independent of these choices which coincides with the restriction of $\tau_\pi$ to $U$.
	\end{proof}
	We can now define the natural transformation
	\[
	\Psi \colon \kf \to C(\, \cdot \,, \km(\frg,\Gamma))
	\]
	by sending a family $\pi$ to the continuous map $\tau_\pi$.
	
	For the construction of $ \Psi^0 \colon \kf^0 \to C(\, \cdot \,, \kt(\frg))$ we consider a family of complex nilmanifolds  $(\pi\colon\mathcal{X} \rightarrow B, [\phi\colon M \to \kx])$ with a $\Diff^0(M)$-framing. For each $b\in B$ the framing yields an identification of the fibre $\pi^{-1}(b)$ with a complex nilmanifold $(M,J_b)$, defined up to an element of $\Diff^0(M)$. Hence, sending $b\in B$ to the class of $J_b$ in $\kt(\frg)$ defines a map
	\begin{equation}\label{eq: tau_phi}
		\tau_{\pi,[\phi]} \colon B \to \kt(\frg),
	\end{equation}
	which depends on the framing $[\phi\colon M \to \kx]$ of the family. As in Lemma \ref{lem: tau is continuous} one shows that the map $\tau_{\pi,[\phi]} \colon B \to \kt(\frg)$ is continuous and we can define the natural transformation
	\[
	\Psi^0\colon \kf^0 \to C(\, \cdot\,, \kt(\frg)) 
	\]
	by associating to a pair $(\pi\colon\mathcal{X} \rightarrow B, [\phi\colon M \to \kx])$ the continuous map $\tau_{\pi,[\phi]}$.
	
	\begin{rem}\label{rem: family over Cg}
		In \cite[Section~2.2]{rollenske09b} Rollenske constructs a family over the complex space $\kc(\frg)$ such that the fibre over $J \in \kc(\frg)$ is the complex nilmanifold $(M,J)$: consider the pull-back $\ku(\frg) = (G \times \ku_n) \times_{\ku_n} \kc(\frg)$ with respect to the projection $(G \times \ku_n) \to \ku_n$ and the inclusion $\kc(\frg) \hookrightarrow \ku_n$. The lattice $\Gamma$ acts on $\ku(\frg)$ from the left and by taking the quotient we obtain a family $\pi \colon \ku(\frg, \Gamma) \to \kc(\frg)$ with $\pi^{-1}(J) = (M,J)$. In particular, the family $\pi$ comes with a framing $[\phi]$ and the induced maps $\tau_{\pi} \colon \kc(\frg) \to \km(\frg, \Gamma)$ and $\tau_{\pi, [\phi]} \colon \kc(\frg) \to \kt(\frg)$ coincide with the projections $\kc(\frg) \to \km(\frg, \Gamma)$ and $\kc(\frg) \to \kt(\frg)$, respectively.
	\end{rem}
	
	The following generalises the existence of a topological coarse moduli space for complex tori \cite[Theorem~3.2, p.~216]{lange2012complex} to complex nilmanifolds.
	
	\begin{prop}\label{prop: topological coarse moduli}
		The topological spaces $\km(\frg, \Gamma)$ and $\kt(\frg)$ together with the natural transformations $\Psi$ and $\Psi^0$ are topological coarse moduli spaces for the functors $\kf$ and $\kf^0$, respectively.
	\end{prop}
	\begin{proof}
		We begin with the functor $\kf$. For a point $p$ the map
		\[ \Psi(p)\colon \mathcal{F}(p) \to C(p, \km(\frg,\Gamma)) = \km(\frg, \Gamma) \]
		is bijective since the equivalence of families over a point identifies biholomorphic complex nilmanifolds. Let $T$ be a topological space and let $\Psi'\colon \kf \to C(\, \cdot \,, T)$ be a natural transformation. To prove the second item of Definition \ref{def: coarse moduli} we have to show that there exists a unique continuous map $\psi \colon \km(\frg, \Gamma) \to T$ such that the diagram \eqref{eq: diagram coarse moduli} commutes.
		Consider the maps
		\[
		\begin{tikzcd}
			\kf(p) \arrow[r, "\Psi(p)"] \arrow[rd,swap,"\Psi'(p)"] & \km(\frg, \Gamma)\\
			& T.
		\end{tikzcd}
		\]
		It suffices to show that the map $\psi = \Psi'(p) \circ \Psi(p)^{-1}$  is continuous. Consider the family $\pi\colon \ku(\frg, \Gamma) \to \kc(\frg)$ from Remark \ref{rem: family over Cg} with the associated map $\tau_\pi\colon \kc(\frg) \to \km(\frg, \Gamma)$. Evaluating the map $\Psi'(\kc(\frg))\colon \kf(\kc(\frg)) \to C(\kc(\frg), T)$ on the equivalence class of the family $\pi$ we obtain a continuous map $\tau_\pi'\colon \kc(\frg) \to T$ such that $\psi \circ \tau_\pi = \tau_\pi'$. Since $\tau_\pi$ is the quotient map this implies that $\psi$ has to be continuous. This proves the claim for the functor $\kf$ with the natural transformation $\Psi$.
		
		For the functor $\kf^0$ and a point $p$ the set $\kf^0(p)$ is the set of equivalence classes of pairs $(X, [\phi])$, where $X$ is a complex manifold biholomorphic to a complex nilmanifold of type $(\frg, \Gamma)$ and $[\phi]$ is the isotopy class of a diffeomorphism $\phi \colon M \to X$. Two pairs $(X, [\phi])$ and $(X', [\phi'])$ are equivalent if there exists a biholomorphism $f \colon X \to X'$ such that $[f \circ \phi] = [\phi']$. Clearly, the associated map $\Psi^0(p) \colon \kf^0(p) \to \kt(\frg)$ is surjective. Suppose $\Psi^0(p)(X, [\phi]) = \Psi^0(p)(X', [\phi'])$ and let $X =(N,J)$ and $X'=(N',J')$. We can assume that $\phi^*J$ and $(\phi')^*J'$ are left-invariant complex structures on $M$, so the assumption $\Psi^0(p)(X, [\phi]) = \Psi^0(p)(X', [\phi'])$ implies that $\phi^*J$ and $(\phi')^*J'$ correspond to the same point in $\kt(\frg)$ and thus there is a diffeomorphism $g\in \Diff^0(M)$ defining a biholomorphism from $(M,{\phi^*J})$ to $(M,{(\phi')^*J'})$. Hence, the map $f= \phi' \circ g \circ \phi^{-1}$ is a biholomorphism from $X$ to $X'$ and we have
		$ [f \circ \phi] = [\phi' \circ g] = [\phi']$. Therefore, the pairs $(X, [\phi])$ and $(X', [\phi'])$ are equivalent and the map $\Psi^0(p) \colon \kf^0(p) \to \kt(\frg)$ is a bijection. To prove the second condition in Definition \ref{def: coarse moduli} one uses exactly the same argument as for $\kf$ and the fact that the map $\tau_{\pi, [\phi]}\colon \kc(\frg) \to \kt(\frg)$ is just the quotient map.
	\end{proof}

	Suppose there are complex structures on $\km(\frg, \Gamma)$ and $\kt(\frg)$ such that the maps $\tau_\pi$ and $\tau_{\pi,[\phi]}$ defined above are holomorphic for every family $\pi$ and every framing $[\phi]$. Then $\Psi$ and $\Psi^0$ are natural transformations from $\kf$ to $\Mor(\, \cdot\,, \km(\frg, \Gamma))$ and from $\kf^0$ to $\Mor(\, \cdot\,, \kt(\frg))$, respectively. If the complex structure on $\km(\frg, \Gamma)$ makes it into a coarse moduli space with respect to the functor $\kf$ and the natural transformation $\Psi$ we will simply say that $\km(\frg, \Gamma)$ exists as a complex space. Similary, we say that $\kt(\frg)$ exists as a complex space if it exists as a coarse moduli space for $\kf^0$ with the natural transformation $\Psi^0$. 
	
	The following is an immediate consequence of Proposition \ref{prop: topological coarse moduli} together with the universal property of coarse moduli spaces (cf. \cite[Corollary~1.2, p.211]{lange2012complex}).
	
	\begin{cor}\label{cor: non hausdorff moduli}
		If $\km(\frg,\Gamma)$, respectively $\kt(\frg)$, is not Hausdorff, then $\km(\frg,\Gamma)$, respectively $\kt(\frg)$, does not exist as a complex space.
	\end{cor}

	We will see in Section \ref{subsect: complex parallelisable structures} that $\kt(\frg)$ is not a Hausdorff space if $\frg$ is a non-abelian Lie algebra admitting a complex parallelisable structure (see Proposition \ref{prop: Tg non hausdorff complex par}). This is roughly due to the fact that the dimension of the automorphism group varies in the Kuranishi family of a complex parallelisable nilmanifold when it is not a complex torus \cite[Theorem~5.1]{rollenske11}. Generally, this jumping phenomenon of the dimension of the automorphism group is the major obstruction for the existence of a complex coarse moduli space. As in \cite{Meersseman19}, we define the function
	\[
	h^0 \colon \kc(\frg) \to \IN_0, \qquad  J \mapsto \dim H^0(\frg^{0,1}, \frg^{1,0}).
	\]
	Recall that for a complex nilmanifold $X =(M,J)$ we have 
	\[
	h^0(J) = \dim H^0(\frg^{0,1}, \frg^{1,0}) = \dim H^0(X, \Theta_X) = \dim \Aut(X).
	\]
	Hence, if we consider $\kc(\frg)$ as the space of left-invariant complex structures on $M$, then the function $h^0$ assigns to $J\in \kc(\frg)$ the dimension of the automorphism group of $X$.  It was proved by Grauert in \cite{Grauert60} that the function $h^0$ is upper-semi continuous in holomorphic families, so in our setting we have $h^0(J') \leq h^0(J)$ for all complex structures $J'$ in a sufficiently small neighbourhood around $J$ in $\kc(\frg)$.
	
	The main strategy to obtain a complex structure on $\kt(\frg)$ is to glue Kuranishi spaces of the complex manifolds parametrised by $\kt(\frg)$ to a global complex structure on $\kt(\frg)$. This requires the assumption that the Kuranishi families are universal, which is equivalent to the function $h^0$ being constant, provided that the Kuranishi space is reduced  \cite[Theorem~4.2]{wavrik1969obstructions}.

	More precisely, for a complex nilmanifold $X= (M,J)$ with a fixed left-invariant hermitian metric we have an embedding $\Kur(X) \hookrightarrow \kc(\frg)$. So by composing this embedding with the projection $\kc(\frg) \to \kt(\frg)$ we obtain a continuous map
	\[
	\pi_X \colon \Kur(X) \to \kt(\frg).
	\]
	This map coincides with the map $\tau_{\pi, [\phi]}\colon \Kur(X) \to \kt(\frg)$ defined in \eqref{eq: tau_phi}, where $\pi \colon \kk \to \Kur(X)$ is the Kuranishi family and $[\phi]$ is the $\Diff^0(M)$-framing described in Remark \ref{rem: framing Kuranishi family}.
	The map $\pi_X$ is surjective onto the germ $\kt(\frg)_J$ of $\kt(\frg)$ at the point corresponding to the complex structure $J$ since the Kuranishi family is complete. 
	As explained in \cite[Remark~14]{catanese2011superficial}, the map $\pi_X \colon \Kur(X) \to\kt(\frg)_J$ is a homeomorphism if and only if it is injective, in which case $\kt(\frg)$ is at least locally analytic around the class of $J$. The following is a precise formulation of \cite[Remark~14]{catanese2011superficial} in our setting.       
	
	\begin{prop}\label{prop: Tg coarse moduli}
		Suppose the following holds for every complex nilmanifold $X$ of type $(\frg, \Gamma)$.
		\begin{enumerate}
			\item The map $\pi_X \colon \Kur(X) \to \kt(\frg)$ is injective.
			\item The Kuranishi family of $X$ is universal.
		\end{enumerate}
		If $\kt(\frg)$ is a Hausdorff space, then $\kt(\frg)$ exists as a complex space.
	\end{prop}
	\begin{proof}
		Let $X$ and $X'$ be complex nilmanifolds of type $(\frg, \Gamma)$. Since $\pi_X$ and $\pi_{X'}$ are injective we can consider $\Kur(X)$  and $\Kur(X')$ as subspaces of $\kt(\frg)$. Taking the intersection $\Kur(X) \cap \Kur(X')$ in $\kt(\frg)$ we obtain an open subset in both $\Kur(X)$ and $\Kur(X')$ where the complex structures coincide because the Kuranishi families of $X$ and $X'$ are universal. Therefore, the charts given by the Kuranishi spaces define a complex structure on $\kt(\frg)$ and since we assumed $\kt(\frg)$ to be Hausdorff this makes $\kt(\frg)$ into a complex space.

		With respect to this complex structure on $\kt(\frg)$ the maps $\tau_{\pi,[\phi]} \colon B \to \kt(\frg)$  defined for a family $(\pi \colon \kx \to B, [\phi \colon M \to \kx])$ with a $\Diff^0(M)$-framing are holomorphic and thus $\Psi^0$ defines a natural transformation from $\kf^0$ to $\Mor(\, \cdot \,, \kt(\frg))$.
		
		Since we already proved that $\kt(\frg)$ is a topological coarse moduli space it suffices to prove the universal property for holomorphic maps. Let $S$ be a complex space and let $\Psi' \colon \kf^0 \to \Mor(\cdot,\, S)$ be a natural transformation. We have to show that for a point $p$ the induced map $\psi^0 = \Psi'(p) \circ \Psi^0(p)^{-1}\colon \kt(\frg) \to S$ is holomorphic. Consider the Kuranishi family $\pi\colon \kk \to \Kur(X)$ for a given complex nilmanifold $X$. Then, by evaluating the maps $\Psi^0(\Kur(X))\colon \kf^0(\Kur(X)) \to \Mor(\Kur(X),\,\kt(\frg))$ and $\Psi'(\Kur(X))\colon \kf^0(\Kur(X)) \to \Mor(\Kur(X),\,S)$ on the class of the Kuranishi family $\pi$, we obtain a commutative diagram
		\[
		\begin{tikzcd}
			\Kur(X) \arrow[r, "\pi_X"] \arrow[rd,swap,"\varphi"] & \kt(\frg) \arrow[d,"\psi^0"] \\
			&  S,
		\end{tikzcd}
		\]
		where $\varphi$ is holomorphic and $\pi_X$ is a local biholomorphism by construction. Since $\kt(\frg)$ is constructed by glueing Kuranishi spaces this implies that $\psi^0$ is holomorphic.
	\end{proof}

	Of course, the assumptions of Proposition \ref{prop: Tg coarse moduli} are very strong and in general difficult to verify. Nevertheless, we will see in Section \ref{sect: period maps} that there are several criteria on the Lie algebra underlying $M$ implying these assumptions.

	\begin{rem}
		The map $\pi_Y \colon \Kur(Y) \to \kt(Y)$ can be defined for every compact complex manifold $Y$ and has been studied for several classes of manifolds. See for instance \cite{catanese2011superficial} for Calabi Yau manifolds and surfaces of general type, or \cite[Section~28.3]{Ornea24} for Hopf manifolds. The case in which $\pi_Y$ is a local homeomorphism is often referred to by the slogan Kuranishi = Teichmüller.
	\end{rem}
	
	The following proposition shows that there is a partial converse to the existence criterion given in Proposition \ref{prop: Tg coarse moduli}.
	
	\begin{prop}\label{prop: fibre dimension criterion}
		If $\kt(\frg)$ exists as a complex space, then the function $h^0$ is constant on the connected components of $\kc(\frg)$.
	\end{prop}
	\begin{proof}
		
		If $\kt(\frg)$ exists as a complex space, then there is a complex structure on $\kt(\frg)$ such that the quotient map $\pi \colon \kc(\frg) \to \kt(\frg)$ is holomorphic. In particular, the graph
		\[
		\textup{Graph}(\pi) = \{(J,\pi(J))\, | \, J \in \kc(\frg)\}
		\]
		of $\pi$ is an analytic subspace of $\kc(\frg) \times \kt(\frg)$ \cite[Theorem~12]{Remmert57}. Consider the holomorphic map $\id \times \pi \colon \kc(\frg) \times \kc(\frg) \to \kc(\frg) \times \kt(\frg)$. Then the preimage of $\textup{Graph}(\pi)$ under the map $\id \times \pi$ is the subset
		\[
		R = \{(J, g \cdot J) \, | \, J \in \kc(\frg), \, g \in G\}
		\]
		of $\kc(\frg)\times \kc(\frg)$ generated by the $G$-action on $\kc(\frg)$ and since $\textup{Graph}(\pi)$ is analytic so is $R$. Restricting the projection $\kc(\frg) \times \kc(\frg) \to \kc(\frg)$ onto the first factor to $R$, we obtain a holomorphic map $p \colon R \to \kc(\frg)$ such that the fibre $p^{-1}(J) =(J, G\cdot J)$ has the dimension of the $G$-orbit $G\cdot J$. Since $p$ is a holomorphic map the fibre dimension of $p$, which is just the orbit dimension, is an upper-semi continuous function on $\kc(\frg)$ \cite[Theorem~15]{Remmert57}. Because the dimension of the stabiliser of the $G$-action with respect to a point $J$ is equal to $h^0(J)$ the dimension of the stabiliser is also upper-semi continuous \cite{Grauert60}. But $ h^0(J) + \dim G\cdot J$ always equals half the real dimension of $G$ and since both $h^0$ and the orbit dimension are upper-semi continuous functions this implies that they are both constant on the connected components of $\kc(\frg)$.
	\end{proof}
	
	If $\kt(\frg)$ exists as a complex space, then the existence of a complex structure on $\km(\frg, \Gamma)$ depends on the properness of the $\Aut(\Gamma)$-action on $\kt(\frg)$. Indeed, suppose $Y$ is a complex space and $\Lambda$ is a group acting properly discontinuously on $Y$, that is, for any two compact subsets $K_1, K_2 \subset Y$ the set $\{\lambda \in \Lambda\, | \, \lambda K_1 \cap K_2 \neq \emptyset\}$ is finite. Then the quotient $Y/\Lambda$ also admits the structure of a complex space such that the map $Y \to Y / \Lambda$ is holomorphic \cite{Cartan57}.
	Moreover, if $Y$ is reduced or normal so is $Y/\Lambda$ and if $Y$ is a complex manifold and the action is also free, then $Y/\Lambda$ is a complex manifold.

	\subsection{First examples}\label{sect: first examples}
	
	There are two classes of complex nilmanifolds for which the existence of a moduli space has already been studied, namely, complex tori and so-called Kodaira manifolds, which are the higher-dimensional analogues of Kodaira surfaces. The results about complex tori go back to Kodaira and Spencer \cite{Kodaira58}, while the study of moduli for Kodaira surfaces was carried out by Borcea in \cite{Borcea84}, and later extended to Kodaira manifolds in \cite{grantcharov2004deformations}. We will briefly recall their results.
	
	\subsubsection{Complex tori}\label{subsect: complex tori}
	
	Let $\frg \cong \IR^{2n}$ be the abelian Lie algebra of dimension $2n$. Since every almost complex structure on $\frg$ is integrable, the space $\kc(\frg)$ of complex structures on $\frg$ can be viewed as the open subset $\ku_n =\{ V \in \Gr(n, \frg_\IC) \, | \, V \cap \overline V = 0 \}$ of $\Gr(n, \frg_\IC)$. For a fixed lattice $\Gamma \cong \IZ^{2n}$ the space $\ku_n$ parametrises all complex tori of dimension $n$.
	
	Alternatively, we can describe this space by varying the lattice in the complex Lie group $\IC^n$: after fixing a basis of $\frg_\IC \cong \IC^{2n}$ an element $V \in \ku_n$ can be identified with the class of a complex $n \times 2n$ matrix $A$. The complex torus corresponding to $A$ is $X =\IC^n/A\IZ^{2n}$ and the matrix $A$ is called a \textit{period matrix} for $X$. Two matrices define the same complex torus if they correspond to the same element of $\Gr(n, \IC^{2n})$, that is, they differ by an element of $\GL(n, \IC)$. Hence, the space of complex tori of dimension $n$ can also be described as
	\[
	\ku_n = \Bigg\{ [A] \in \Gr(n,\IC^{2n}) \, \Bigg| \, \det \begin{pmatrix}
		A \\
		\overline A
	\end{pmatrix}
	\neq 0 \Bigg\}.
	\]
	The space $\ku_n$ has two connected components $\ku_n^+$ and $\ku_n^-$ given by the sign of $\det \begin{pmatrix}
		A \\
		\overline A
	\end{pmatrix}$. As already mentioned in Section \ref{sect: affine maps}, a holomorphic map between complex tori is affine. Since translations by the abelian Lie group $G = \IR^{2n}$ act trivially on $\ku_n$, the Teichmüller space of complex tori of dimension $n$ is just $\ku_n$. For the lattice $\Gamma \cong \IZ^{2n}$ we have $\Aut(\Gamma) \cong \GL(2n, \IZ)$, so the moduli space of complex tori is the quotient $\ku_n/\GL(2n, \IZ)$, where the action of an element $M \in \GL(2n, \IZ)$  on a class $[A] \in \ku_n$ is given by $[AM]$.
	
	It was first observed in \cite{Kodaira58} that unless $n = 1$ the action of $\GL(2n, \IZ)$ on $\ku_n$ is not properly discontinuous, and in fact $\ku_n / \GL(2n, \IZ)$ is not a Hausdorff space. In particular, there exists no coarse moduli space of complex tori of dimension $n$ for $n \geq 2$ (cf. \cite[Corollary~3.3, p.~216]{lange2012complex}).

	\subsubsection{Kodaira manifolds}\label{subsect: Kodaira manifolds}
	
	A Kodaira manifold is a complex nilmanifold whose underlying Lie algebra is isomorphic to $\frh_{2k+1} \oplus \IR$, where $\frh_{2k+1}$ is the $(2k+1)$-dimensional Heisenberg algebra. The case $k=1$ corresponds to the class of Kodaira surfaces studied in \cite{Borcea84}. 
	
	Let $\frg = \frh_{2k+1} \oplus \IR$. Then there exists a basis $e_1,\dots,e_k,f_1,\dots,f_k,c,z$ of $\frg$ such that the only non-zero brackets are $[e_i,f_i] =c$ for $i=1, \dots, k$. For every $\sigma = (\sigma_2, \dots, \sigma_k)$ with $\sigma_i \in \{1,-1\}$ we can define an almost complex structure $J_\sigma$ on $\frg$ by
	\begin{align*}
		J_\sigma e_1 &= f_1,\\
		J_\sigma e_i &= \sigma_i f_i, \quad i = 2,\dots,k,\\
		J_\sigma c &= z.
	\end{align*}
	It was proved in \cite[Proposition~3.6]{rollenske09} that $J_\sigma$ is integrable, and up to Lie algebra automorphism, every complex structure on $\frg$ is of this form.
	
	Therefore, every complex structure on $\frg$ preserves the centre $\kz\frg = \langle c,z\rangle$ and the filtration $0 \subset \kz\frg \subset \frg$ defines a SPTBS on $\frg$. Hence, every Kodaira manifold with underlying Lie algebra $\frg$ has the structure of a principal elliptic curve bundle over a complex torus of dimension $k$.
	
	We begin by describing the Teichmüller space $\kt(\frg)$.
	Let us define
	\[
	\kh_k =\Bigg\{ [A] \in \ku_k \, \Bigg| \, A \begin{pmatrix}
		0 & -I_k\\
		I_k & 0
	\end{pmatrix} A^T = 0 \Bigg\},
	\]
	considered as a subspace of the space of complex structures on the abelian Lie algebra $\langle e_1,\dots,e_k,f_1,\dots,f_k \rangle$. The connected components of the complex space $\kh_k$ are the subsets
	\[
	\kh_{k,\ell} = \Bigg\{ [A] \in \ku_k \, \Bigg| \, A \begin{pmatrix}
		0 & -I_k\\
		I_k & 0
	\end{pmatrix} A^T = 0, \; \text{index}\Bigg( i\overline A \begin{pmatrix}
		0 & -I_k\\
		I_k & 0
	\end{pmatrix} A^T \Bigg) =\ell \Bigg\},
	\]
	defined for $\ell=0,\dots, k$. Here, the index is the number of negative eigenvalues. We will see in Section \ref{subsect: 2-step nilpotent Lie algebras} that the Lie algebra $\frg = \frh_{2k+1}\oplus \IR$ always satisfies the assumptions of Proposition \ref{prop: Tg coarse moduli}.  Together with the results of \cite[Sections~4.1-4.4]{grantcharov2004deformations} we obtain the following description of $\kt(\frg)$.
	
	\begin{prop}
		The Teichmüller space $\kt(\frg)$ exists as a complex manifold biholomorphic to $\kh_k \times \ku_1$. In particular, we have $\dim \kt(\frg) = \frac{1}{2}k(k+1) +1$ and the connected components of $\kt(\frg)$ are of the form $\kh_{k,\ell} \times \ku_1^\pm$.
	\end{prop}
	
	The factor $\kh_k$ parametrises complex tori of dimension $k$ appearing as the base of a Kodaira manifold and $\ku_1$ parametrises the elliptic curves in the fibre.

	Turning to the moduli space, by \cite[Theorem~1]{Borcea84} the lattices $\Gamma$ in the Lie group $G$ associated to $\frg =\frh_3 \oplus \IR$ are classified up to isomorphism by the single torsion coefficient $m$ of $H_1(\Gamma \backslash G, \IZ)$. A representative for each isomorphism class of these lattices is given by the lattice $\Gamma_m$ generated by  $\exp(e_1), \exp(f_1)$, $\exp(z)$, and $\exp(\frac{1}{m}c)$. So it suffices to treat the case $\Gamma = \Gamma_m$, and it was shown in \cite{Borcea84} that $\km(\frg,\Gamma_m)$ is biholomorphic to the product of $\IC$ with a punctured disc. 	
	Unfortunately, as for complex tori, a moduli space does not exist for Kodaira manifolds of higher dimension. Concretely, it was shown in \cite{grantcharov2004deformations} that the quotient $\km (\frh_{2k+1}, \Gamma)$ is non-Hausdorff for $k\geq 2$ and for every choice of lattice $\Gamma$.

	\begin{rem}
		Suppose we fix an orientation for a given Kodaira manifold $M =\Gamma \backslash G$ with underlying Lie algebra $ \frg =\frh_{2k+1} \oplus \IR$. Then $G = H_{2k+1}  \times \IR$, where $H_{2k+1}$ is a simply connected Lie group associated to $\frh_{2k+1}$. According to \cite[Theorem~5]{Ghorbel09} there exists a lattice $\Gamma' \subset H_{2k+1}$ such that $\Gamma \cong \Gamma' \times \IZ$. The automorphism $\id_{\Gamma'} \times -\id_\IZ$ of $\Gamma' \times \IZ$ induces an orientation-reversing diffeomorphism  of the Kodaira manifold $M$. Hence, the oriented Teichmüller space $\kt^+(\frg)$ has only half the connected components of $\kt(\frg)$ while the oriented moduli space coincides with the unoriented one (cf. Remark \ref{rem: orientation}).
	\end{rem}

	\section{Period maps}\label{sect: period maps}

	In this section we introduce two period maps. These will be our main tools for the study of the Teichmüller space.
	
	We first explain how the classical period map defined for families of Calabi--Yau manifolds can also be defined as a holomorphic map on a suitable reduction of the space of left-invariant complex structures on a given nilmanifold or holomorphic families thereof.
	We then investigate how the behaviour of this map corresponds to the structure of the Teichmüller space.
	
	Next, we introduce our second period map which depends on the choice of a stable principal torus bundle series on a given Lie algebra. This map is roughly defined by mapping a given complex nilmanifold with the structure of an iterated principal torus bundle to the product of complex tori appearing as the fibres of this bundle. We then study some basic properties of this map and explain how the two period maps are related.

	\subsection{Period maps associated to a nilmanifold}\label{sect: period map nilmanifold}

	As already mentioned in the introduction, the Tian--Todorov theorem (see \cite{Tian86}) states that a Calabi--Yau manifold $Y$ is unobstructed, that is, the Kuranishi space $\Kur(Y)$ of $Y$ is smooth. Since small deformations of $Y$ are still Calabi--Yau manifolds one can define a map $\kp_Y \colon \Kur(Y) \to \IP H_{dR}^n(Y, \IC)$ by sending $t \in \Kur(Y)$ to the class of a trivialising section of the canonical bundle of the fibre $Y_t$ in the Kuranishi family. This map is known as the local period map of $Y$ and it was shown by Griffiths in \cite{Griffiths68} that $\kp_Y$ is in fact a holomorphic map. Moreover, he computed the differential of $\kp_Y$ at $0 \in \Kur(Y)$ and it follows from the canonical bundle of $Y$ being trivial together with Serre duality that $\kp_Y$ is an embedding. As a consequence, it was proved by Catanese that the map $\pi_Y \colon \Kur(Y) \to \kt(Y)$ is always injective (cf. Section \ref{subsect: existence criteria}). More generally, he proved the following criterion \cite[Proposition~15]{catanese2011superficial}. 
	
	\begin{prop}\label{prop: catanese criterion} If there exists an injective continuous map $f\colon \Kur(X) \to S$ to a Hausdorff space $S$ factoring through $\pi_X$, then the map $\pi_X\colon \Kur(X) \rightarrow \kt(X)$ is a local homeomorphism.
	\end{prop}

	In light of Propostion \ref{prop: Tg coarse moduli}, we would like to apply this criterion to complex nilmanifolds. Let $X = (M, J)$ be a complex nilmanifold of dimension $n$ and of type $(\frg, \Gamma)$. We have already remarked that $X$ will never admit a Kähler structure unless it is a complex torus \cite{hasegawa1989minimal}. However, as mentioned in Section \ref{sect: Cg}, the canonical bundle of a complex nilmanifold can always be trivialised by a closed left-invariant form \cite[Theorem~2.7]{Barberis09}. Since this left-invariant form is unique up to scaling by a constant this allows us, as in the Kähler case, to associate a complex structure in $\Kur(X)$ or $\kc(\frg)$ to the class of a left-invariant closed $(n,0)$-form in $\IP H^n(\frg , \IC) = \IP H^n_{dR}(M, \IC)$. Unfortunately, unlike in the Kähler case, the Kuranishi space of $X$ and thus $\kc(\frg)$ may not be smooth or even reduced (see Example \ref{exam: non-reduced point}). As explained in Example \ref{exam: non closed deformation}, this corresponds to the fact that an infinitesimal deformation of $X$ may no longer have trivial canonical bundle. So for now we only have continuous maps
	\[
	\begin{tikzcd}
		{|\kc(\frg)|} \arrow[r, "\widetilde \kp"]  & \IP H^n(\frg, \IC) \\
		{|\Kur(X)|} \arrow[ru,swap, "\kp_X"] &  
	\end{tikzcd}
	\]
	on the topological spaces $|\kc(\frg)|$ and $|\Kur(X)| $ underlying $\kc(\frg)$ and $\Kur(X)$, respectively. Both $\widetilde \kp$ and $\kp_X$ are given by associating to a complex structure $I$ the one-dimensional vector space $H^{n,0}(\frg, I) \subset H^n(\frg, \IC)$ considered as a point in $\IP H^n(\frg, \IC)$. 
	
	\begin{prop}
		The maps $\widetilde \kp$ and $\kp_X$ are holomorphic on the reduction of $\kc(\frg)$ and $\Kur(X)$, respectively. Moreover, the differentials
		\[
		\begin{tikzcd}
			\ker \dbar\supset T_J \kc(\frg)_{\textup{red}}  \arrow[r,"d\widetilde\kp_J"] & \Hom\left(H^{n,0}(\frg,J) , H^n(\frg,\IC) /H^{n,0}(\frg,J)\right)  \\
			H^1(X, \Theta_X) \supset T_0 \Kur(X)_{\textup{red}} \arrow[ru,swap, "(d\kp_X)_0"]          &  
		\end{tikzcd}
		\]
		are given by $d\widetilde\kp_J(\eta) = [\eta \lrcorner \omega] = (d\kp_X)_0([\eta])$, where $\eta \in T_J \kc(\frg)_{\textup{red}}$ and $\omega \in \Wedge^n\frg^*_\IC $ generates the space $H^{n,0}(\frg, J)\subset \IP H^n(\frg, \IC)$. 
	\end{prop}
	\begin{proof}
		Recall that the Plücker embedding $\iota\colon \Gr(n, \frg_\IC^*) \to  \IP\left(\Wedge^n\frg_\IC^*\right)$ is a holomorphic embedding and let $J \in \kc(\frg) \subset \Gr(n, \frg_\IC^*)$. Then every form $\omega$ representing the class $\iota(J) \in \IP\left(\Wedge^n\frg^*_\IC\right)$ is closed by Lemma \ref{lem: plücker image} and we have $\langle \omega \rangle = H^{n,0}(\frg, J)$. Hence, $\iota(J)$ is contained in the subspace $\IP\ker d \subset  \IP\left(\Wedge^n\frg_\IC^*\right)$ and by restricting $\iota$ to the reduction of $\kc(\frg)$ we obtain an embedding $\kc(\frg)_{\textup{red}} \hookrightarrow \IP\ker d$. Moreover, the form $\omega$ cannot be exact since $\partial(\Wedge^{n-1,0}\frg^*) =0$. In particular, the image of $\kc(\frg)_{\textup{red}} \hookrightarrow \IP\ker d$ is contained in the domain of the projection $\IP \ker d \dashrightarrow \IP H^n(\frg, \IC)$. Therefore, the diagram
		\[
		\begin{tikzcd}
			\kc(\frg)_{\textup{red}} \arrow[r, "\iota"] \arrow[rd, swap, "\widetilde \kp"] & \IP \ker d \arrow[d, dashed] \\
			& \IP H^n(\frg, \IC)          
		\end{tikzcd}
		\]
		commutes and the map $\widetilde \kp$ is holomorphic. To compute the differential of $\widetilde \kp$, we recall that the differential of the Plücker embedding 
		\[
		d\iota\colon \frg^{*0,1} \otimes \frg^{1,0} \to \Hom\left(\langle \omega \rangle , \Wedge^n\frg_\IC^* /\langle\omega \rangle\right) = T_{\langle \omega \rangle}  \IP\left(\Wedge^n\frg_\IC^*\right)
		\]
		at the point $J$ takes an element $\eta \in \frg^{*0,1} \otimes \frg^{1,0}$ to the homomorphism mapping $\omega$ to the class of $\eta \lrcorner \omega$ in $\Wedge^n\frg_\IC^* /\langle\omega \rangle$. Thus, the differential of $\widetilde\kp$ takes $\eta$ to the homomorphism sending $[\omega] \in H^n(\frg, \IC)$ to $[\eta \lrcorner \omega] \in H^n(\frg,\IC)/ \langle [\omega] \rangle$. After choosing an hermitian metric on $X$ we can consider $\Kur(X)_{\textup{red}}$ as a subspace of $\kc(\frg)_\textup{red}$ such that $\kp_X$ is the restriction of $\widetilde\kp$ to $\Kur(X)_\textup{red}$. Hence, $\kp_X$ is holomorphic on $\Kur(X)_{\textup{red}}$ and its differential at $0$ is as above.
	\end{proof}
	
	\begin{rem}\label{rem: maximal domain period map}
		Note that the period map $\widetilde \kp$ can also be defined as a holomorphic map on the complex space $\kc(\frg) \cap \IP \ker d$, which is possibly larger than the reduction of $\kc(\frg)$. Similarly, we could define $\kp_X$ on the complex space $\Kur(X) \cap \IP \ker d$ after choosing an embedding $\Kur(X) \hookrightarrow \kc(\frg)$. However, for our purposes there is no advantage in making this distinction, so we will always work with the reduction of $\kc(\frg)$ and $\Kur(X)$.
	\end{rem}
	
	We will refer to $\kp_X \colon \Kur(X)_{\textup{red}} \to \IP H^n(\frg, \IC)$ as the \textit{local period map} associated to $X$ or $J$.
	
	\begin{rem}
		One could also define the local period map via the universal property of the Grassmannian: if the Kuranishi space $\Kur(X)$ is smooth, then the subspaces $H^{n,0}(\frg, J_t) \subset H^n(\frg, \IC)$ determine a holomorphic subbundle of the trivial bundle of rank $b_n(\frg)$ on $\Kur(X)$. By the universal property of the Grassmannian every such bundle is obtained as the pullback of the universal bundle on $\Gr(1, H^n(\frg, \IC)) = \IP H^n(\frg, \IC)$ along a uniquely determined map $\Kur(X) \to \IP H^n(\frg, \IC)$. This map coincides with the local period map $\kp_X$.
	\end{rem}
	
	\begin{rem}
		It was shown in \cite{xia2022variation} that under certain cohomological conditions, the map $\kp_Y$ can be defined on the base of a smooth family of deformations of a (not necessarily Kähler) manifold $Y$ with trivial canonical bundle. In our setting, these conditions translate to the equality
		\[
		\Wedge^{n,1}\frg^* \cap d \Big( \Wedge^n \frg_\IC^* \Big) = d\left( \Wedge^{n,0}\frg^*\right).
		\]
		Since $d\left( \Wedge^{n,0}\frg^*\right)=0$ this is equivalent to every $\dbar$-closed form in $\Wedge^{n-1,1}\frg^*$ being $d$-closed, which is a necessary condition for the Kuranishi space of a complex nilmanifold with underlying Lie algebra $\frg$ to be reduced (see Proposition \ref{prop: reduced points in Cg}).
	\end{rem}
	
	To define a period map on the Teichmüller space we recall that the identity component $\Diff^0(M)$ of the diffeomorphism group of $M = \Gamma \backslash G$ acts trivially on $H^n_{dR}(M,\IC) \cong H^n(\frg, \IC)$. Again, we view the Lie group $G$ as a subgroup of $\Diff^0(M)$ by associating to $g \in G$ the map $R_g \colon M \to M$ induced by the right-translation on $G$. Therefore, the map $\widetilde{P} \colon \kc(\frg) \to \IP H^n(\frg, \IC)$ descends to a map
	\[
	\begin{tikzcd}
		\kp \colon \kt(\frg) \to \IP H^n(\frg, \IC).
	\end{tikzcd}
	\]
	We will refer to $\kp$ as the \textit{global period map} associated to $\frg$. In general, this map is just continuous since $\kt(\frg)$ may not admit a complex structure. The image of $\kp$ is contained in the so-called \textit{period domain} 
	\[
	\kd = \{[\alpha] \in \IP H^n(\frg, \IC) \, | \, \alpha \wedge \overline \alpha \neq 0,\, [\alpha \wedge \alpha] =0 \},
	\]
	which is an open subset of a projective variety in $ \IP H^n(\frg, \IC)$. It follows directly from \cite[Theorem~B]{Borcea86} that $\kp$ is surjective onto $\kd$ if $\frg= \IR^4$ or $\frg = \frh_3 \oplus \IR$. However, the period map will generally not be surjective for nilmanifolds of higher dimension. In fact, by Lemma \ref{lem: plücker image} a class in $\kd$ is in the image of $\kp$ only if it has a decomposable representative. 
	
	For our purposes we will be mostly interested in the question when the local or global period maps are injective. For instance, the map $\kp$ makes the diagram
	\[
	\begin{tikzcd}
		\Kur(X) \arrow[r, "\pi_X"] \arrow[rd,swap,"\kp_X"] & \kt(\frg) \arrow[d,"\kp" ]\\
		& \kd
	\end{tikzcd}
	\]
	commute for every complex nilmanifold $X$ of type $(\frg, \Gamma)$. So by the criterion given in Proposition \ref{prop: catanese criterion} the map $\pi_X$ is a local homeomorphism if $\kp_X$ is injective. The following lemma will be used frequently throughout the text.
	
	\begin{lem}\label{lem: local period map}
		If every $d$-exact form in $\Wedge^{n-1,1}\frg^*$ is $\dbar$-exact, then the local period map $\kp_X$ is an embedding. 	
		
		In particular, the map $\kp_X$ is an embedding if the Frölicher spectral sequence of $X$ degenerates at the first page.
	\end{lem}
	\begin{proof}
		Suppose $\omega$ generates the space $\Wedge^{n,0}\frg^*$. Then the contraction map
		\[
		\lrcorner\omega \colon \Wedge^{0,k}\frg^* \otimes \frg^{1,0} \to \Wedge^{n-1,k}\frg^*
		\]
		is an isomorphism, which descends to an isomorphism between $H^k(X, \Theta_X)$ and $H^{n-1,k}_{\dbar}(X)$. Let $[\eta] \in T_J\Kur(X)_{\textup{red}} \subset H^1(X, \Theta_X)$ such that $(d\kp_X)_0([\eta]) =0$. Then there exists a form $\beta = \beta_1 +\beta_2 \in \Wedge^{n-1,0}\frg^* \oplus \Wedge^{n-2,1}\frg^*$ such that $d\beta = \dbar \beta_1 + \partial \beta_2 = \eta \lrcorner \omega$, and by our assumption $\partial \beta_2 =0$. Since the contraction map above is an isomorphism there exists an element $z \in \frg^{1,0}$ such that $z\lrcorner \omega = \beta_1$ and thus $0 = \eta\lrcorner\omega -\dbar(z\lrcorner \omega) = (\eta+\dbar z)\lrcorner \omega$. Therefore $\eta = -\dbar z$, and thus the class $[\eta] \in H^1(X, \Theta_X)$ is zero.
	\end{proof}

	In the following sections we will see several criteria for the global period map to be injective. We will often use the following criterion to show that the Teichmüller space exists as a complex space.
	
	\begin{prop}\label{prop: Tg exists}
		Suppose $\kc(\frg)$ is reduced and the fibres of $\widetilde{\kp}\colon \kc(\frg) \to \kd$ are connected. If the local period map $\kp_X$ is an embedding for every complex nilmanifold of type $(\frg, \Gamma)$, then $\kt(\frg)$ exists as a complex space and the global period map $\kp$ is an embedding.
	\end{prop}
	\begin{proof}
		According to Proposition \ref{prop: Tg coarse moduli} we first have to show that for every $X$ the map $\pi_X\colon \Kur(X) \to \kt(\frg)$ is injective and that its Kuranishi family is universal. The former is implied by Proposition \ref{prop: catanese criterion} since we assumed $\kp_X$ to be an embedding.
		
		Let $J$ be a complex structure on $\frg$. If $J$ is a reduced point in $\kc(\frg)$, then the local period map $\kp_X$ with respect to $X = (M,J)$ is an embedding if and only if every $d$-exact form in $\Wedge^{n-1,1}\frg^*$ is already $\dbar$-exact. Under this assumption on forms in $\Wedge^{n-1,1}\frg^*$ it was shown in \cite[Theorem~3.6]{Rao18} that the function $h^0$ given by $h^0(I) = \dim \Aut(M,I)$ is constant on the Kuranishi space of $X$. Hence, the Kuranishi family of $X$ is universal by \cite[Theorem~4.2]{wavrik1969obstructions}.
		
		Since the map $\pi_X$ is a local homeomorphism and the local period map $\kp_X$ is injective at every point, it follows that the fibres of the global period map $\kp$ are discrete, and because we assumed the fibres of $\widetilde \kp$ to be connected, this implies that $\kp$ has connected and discrete fibres. Thus, the global period map $\kp$ is injective. In particular, the Teichmüller space $\kt(\frg)$ is Hausdorff and it follows from Proposition \ref{prop: Tg coarse moduli} that $\kt(\frg)$ exists as a complex space such that $\kp$ is an embedding.
	\end{proof}

	\begin{rem}
		Note that for every $J \in \kc(\frg)$ the fibre dimension of $\widetilde\kp$ at $J$ is equal to $\dim G\cdot J = n-h^0(J)$ if $\kp$ is injective, and since the fibre dimension of $\widetilde \kp$ is upper-semi continuous on $\kc(\frg)$, both $n-h^0$ and $h^0$ are upper-semi continuous. Thus, the function $h^0$ is constant on the connected components of $\kc(\frg)$ if $\kp$ is injective.
	\end{rem}

	\subsection{Period maps associated to a stable torus bundle series}\label{sect: period map SPTBS}

	Throughout this section we fix a nilpotent Lie algebra $\frg$ of dimension $2n$ with nilpotency index $\nu$, together with a SPTBS
	\[
	0 = \ks^0\frg \subset \ks^1\frg \subset \ks^2\frg \subset \dots \subset\ks^{\nu-1}\frg \subset \ks^\nu\frg = \frg.
	\]
	Let $\gotht^i =  \ks^i\frg / \ks^{i-1}\frg$ and set $2k_i = \dim \gotht^i$. Then, as explained in Section \ref{sect: torus bundles}, for every lattice $\Gamma \subset G$ in the associated simply connected Lie group, the nilmanifold $M= \Gamma \backslash G$ has the structure of an iterated real principal torus bundle
	\begin{equation*}
		\begin{tikzcd}
			M= M_1 \rar{\pi_1} & M_2 \rar{\pi _2} & {\cdots} \rar{\pi_{\nu-1}} & M_\nu= T^{2k_\nu}
		\end{tikzcd} 
	\end{equation*}
	such that the fibre of $\pi_i$ is a $2k_i$-dimensional torus $T^{2k_i}$ with Lie algebra $\gotht^i$. Moreover, for every $i=1,\dots, \nu$, a complex structure $J$ on $\frg$ induces complex structures $J_i$ on $\frg/\ks^{i-1}\frg$ and $J^i$ on $\gotht^i$. So we get a holomorphic iterated principal torus bundle 
	\begin{equation}\label{eq: tower}
		\begin{tikzcd}
			X= X_1 \rar{\pi_1} & X_2 \rar{\pi _2} & {\cdots} \rar{\pi_{\nu-1}} & X_\nu = T_\nu,
		\end{tikzcd} 
	\end{equation}
	where $X_i= (M_i,J_i)$ and with fibres $T_i =(T^{2k_i}, J^i)$. Considering $\ku_{k_i} \subset \Gr(k_i,\gotht^i_\IC)$ as the space of left-invariant complex structures on $T^{2k_i}$, we can associate to every $J \in \kc(\frg)$ a tuple of complex structures $(J^1,\dots, J^\nu)\in \ku_{k_1} \times\dots \times \ku_{k_\nu}$.  Similar to $\widetilde \kp$, sending a complex structure $J$ to the tuple $(J^1, \dots, J^\nu)$ defined with respect to $\ks = (\ks^i\frg)_{i=1,\dots, \nu}$ is initially only well-defined as a continuous map
	\[
	\widetilde\kp^\ks\colon |\kc(\frg)| \to \ku =\ku_{k_1} \times\dots \times \ku_{k_\nu}
	\]
	because $X$ may admit infinitesimal deformations that no longer have a torus bundle structure.
	As for the period maps defined in the previous section we have
	\begin{prop}
		The map $\widetilde\kp^\ks$ is holomorphic on the reduction of $\kc(\frg)$.
	\end{prop}
	\begin{proof}
		Consider the set
		\[
		\kv =\left\{V \in \Gr(n,\frg_\IC) \, \middle| \, \dim_\IC(V \cap \ks^i\frg_\IC) = \sum_{j=1}^i k_j \; \textup{ for all } i =1,\dots, \nu\right\}.
		\]
		The set $\kv$ describes a Schubert cell in $\Gr(n, \frg_\IC)$, so $\kv$ is an open subset of a normal variety in $\Gr(n, \frg_\IC)$ \cite[Theorem~3]{Ramanan85}. Let us denote by $p_i \colon \frg_\IC \to \frg_\IC/ \ks^i\frg_\IC$ the projection. Then we can define a holomorphic map
		\[
		\kv \to \Gr(k_1, \gotht^1_\IC) \times \Gr(k_2, \gotht^2_\IC) \times \dots \times \Gr(k_\nu, \gotht^\nu_\IC)
		\]
		by
		\[
		V \mapsto \left(V \cap \ks^1\frg_\IC,\; p_1(V \cap \ks^2\frg_\IC), \dots,p_{\nu-2}(V \cap \ks^{\nu-1}\frg_\IC),\; p_{\nu-1}(V)\right).
		\]
		Since $\ks$ defines a SPTBS on $\frg$, the set $\kc(\frg) \subset \Gr(n, \frg_\IC)$ is contained in $\kv$ and for a complex structure $J$ the space $p_i(\frg^{1,0}\cap \ks^{i-1}\frg_\IC) \subset \gotht^i_\IC$ defines the complex structure $J^i$ on $\gotht^i$.
		Hence, this map coincides with $\widetilde \kp^\ks$ on $|\kc(\frg)|$ and thus it is holomorphic on the reduction of $\kc(\frg)$.
	\end{proof}
	
	\begin{rem}
		As for the period map $\widetilde \kp$ (see Remark \ref{rem: maximal domain period map}) we could define $\widetilde \kp^\ks$ on a possibly larger complex space, but for our purposes it suffices to consider the reduction of $\kc(\frg)$.
		
		Also, note that one may define maps similar to $\widetilde \kp$ and $\widetilde\kp^\ks$ for every suitable filtration of $\frg$ compatible with every complex structure. The maps $\widetilde\kp$ and $\widetilde\kp^\ks$ correspond to the two extremal cases of the trivial filtration and a stable principle torus bundle series.
	\end{rem}
	
	Since $\gotht^i$ is contained in the centre of $\frg/\ks^{i-1}\frg$, the right-action of $G$ on $\kc(\frg)$ induces the trivial action on $\ku_{k_i}$ for every $i$. In other words, if there exists a holomorphic right-translation between two complex nilmanifolds $X=(M,J)$ and $X'=(M,J')$, then the fibre tori $T_i$ and $T_i'$ with respect to the SPTBS $\ks$ are the same. Therefore, the map $\widetilde\kp^\ks$ descends to a map 
	\[
	\kp^\ks \colon \kt(\frg) \to \ku.
	\]
	In the same way as above, we can define a local version $\kp^\ks_X\colon \Kur(X)_{\textup{red}} \to \ku$ of $\kp^\ks$ which satisfies $\kp_X^\ks = \kp^\ks \circ \pi_X$.

	\subsection{Relation between the period maps}\label{subsect: relation between period maps}
	
	Before we apply the period maps $\kp$ and $\kp^\ks$ to the study of our moduli spaces we briefly explain the relation between these maps. For $i=1,\dots, \nu,$ we denote by
	\[
	Z_i = \im\left(\widetilde\kp_i \colon \kc(\frg/\ks^{i-1}\frg) \to \IP H^{n_i}(\frg/\ks^{i-1}\frg, \IC) \right)
	\]
	the image of the period maps on $\kc(\frg/\ks^{i-1}\frg)$, where $2n_i = \dim \frg/\ks^{i-1}\frg$.
	In particular, $Z_\nu$ can be identified with $\ku_{k_\nu}$ since the Lie algebra $\frg/\ks^{\nu-1}\frg$ is abelian. Our goal is to construct maps $\tau_i\colon Z_i \to Z_{i+1}$ making the diagram
	\[
	\begin{tikzcd}
		\kc(\frg) \arrow[r] \arrow[d, "\widetilde\kp"] & \kc(\frg/\ks^1\frg) \arrow[r] \arrow[d, "\widetilde\kp_2"] &     \dots \arrow[r] &\kc(\frg/\ks^{\nu-2}\frg) \arrow[r] \arrow[d, "\widetilde\kp_{\nu-1}"] &\kc(\frg/\ks^{\nu-1}\frg) \arrow[d, "\widetilde\kp_\nu"] \\
		Z_1\arrow[r, "\tau_1"]          & Z_2\arrow[r,"\tau_2"]                  & \dots \arrow[r,"\tau_{\nu-2}"]   &Z_{\nu-1} \arrow[r, "\tau_{\nu-1}"]  & Z_\nu      
	\end{tikzcd}
	\]
	commute. We will construct these maps as follows: let $\omega_i \in \Wedge^{n_i} (\frg/\ks^{i-1}\frg)_\IC^*$ be a representative of a class in $Z_i$ corresponding to a complex structure $J_i$ on $\frg/\ks^{i-1}\frg$ and consider the associated map 
	\[
	\varphi_{\omega_i} \colon (\frg/\ks^{i-1}\frg)_\IC \to \Wedge^{n_i-1} (\frg/\ks^{i-1}\frg)_\IC^*,
	\]
	which sends $x\in (\frg/\ks^{i-1}\frg)_\IC$ to the form $x \lrcorner \omega_i$. The kernel of $\varphi_{\omega_i}$ is exactly the space $\frg_i^{0,1}$ of $(0,1)$-vectors with respect to $J_i$.  Let $x_1, \dots, x_{k_i}$ be a basis of the vector space  $ \overline{\frg_i^{0,1}} \cap \gotht^i_\IC =\frg^{1,0}_i \cap \gotht^i_\IC$ which determines the induced complex structure on $\gotht^i$. Then the form $\omega_{i+1} = (x_1\wedge \dots \wedge x_{k_i}) \lrcorner \omega_i$ vanishes on $\gotht_\IC^i \subset (\frg/\ks^{i-1}\frg)_\IC$, so we can consider it as an element of $\Wedge^{n_{i+1}} (\frg/\ks^{i}\frg)_\IC^*$ and $\omega_{i+1}$ corresponds to the complex structure $J_{i+1}$ on $\frg/\ks^i\frg$ induced by $J_i$. Hence, the class $[\omega_{i+1}]$ defines an element of $Z_{i+1}$ and we have 
	\begin{lem}
		The class $[\omega_{i+1}] \in Z_{i+1}$ does not depend on the representative of the class $[\omega_i] \in Z_i$ nor on the chosen basis. 
	\end{lem}
	\begin{proof}
		By the definition of a SPTBS the elements $x_1,\dots,x_{k_i}$ are contained in the centre of $\frg_\IC/\ks^{i-1}\frg_\IC$ and thus we have
		\[ 
		(x_1\wedge \dots \wedge  x_{k_i}) \lrcorner (\omega_i +d \beta) = \omega_{i+1} + (-1)^{k_i} d ((x_1 \wedge \dots \wedge x_{k_i})\lrcorner \beta)
		\]
		for every $\beta \in \Wedge^{n_i-1}\frg_\IC^*$.
	\end{proof}
	
	Hence, we can define $\tau_i \colon Z_i \to Z_{i+1}$ by $\tau_i([\omega_i]) =[\omega_{i+1}]$ and the above diagram commutes. We immediately obtain
	
	\begin{prop}\label{prop: forms on base}
		Let $J$ and $J'$ be complex structures on $\frg$ with their induced complex structures $J_i$ and $J'_i$ on $\frg/\ks^{i-1}\frg$. If $\widetilde \kp(J) = \widetilde\kp (J')$, then $\widetilde\kp(J_i) = \widetilde\kp (J'_i)$ for every $i=1, \dots, \nu$. In particular, the induced complex structures on the base torus $M_\nu$ coincide.
	\end{prop}

	If one restricts $\widetilde \kp$ to the (possibly empty) subset of abelian complex structures on $\frg$ and $\ks$ is the ascending central series, then $\widetilde\kp(J) = \widetilde\kp(J')$  implies $\widetilde\kp^\ks(J) = \widetilde\kp^\ks(J')$, that is, all the complex tori given by the bundle structure of $(M,J)$ and $(M,J')$ coincide. However, as the following example will show, this is generally not the case. 
	
	\begin{exam}\label{exam: iwasawa x curve}
		Let $X = \ki \times E$ be the product of the Iwasawa manifold $\ki$ with an elliptic curve $E$. Then there is basis $\omega^1, \dots, \omega^4$ of left-invariant $(1,0)$-forms on $X$ with structure equations
		\[
		d\omega^1 = d\omega^2 = d\omega^3 =0, \; d\omega^4 = \omega^{12},
		\]
		and the ascending central series on the underlying real Lie algebra provides $X$ with the structure of a principal bundle over a complex 2-torus $T_2$ whose fibre is a complex 2-torus $T_1$.
		If $z_1,\dots, z_4$ denotes the dual basis, then the class of $\overline \omega^3 \otimes z_4$ defines an element of $H^1(X, \Theta_X)$. For small $t \in \IC$ the forms $t\overline \omega^3 \otimes z_4$ define a one-parameter family $\{J_t\}$ of deformations of the complex structure $J= J_0$ on $X$, deforming only the fibre torus $T_1$. So the complex nilmanifolds $X_t$ with complex structure $J_t$ are still 2-torus bundles over $T_2$ but the fibre tori are different from $T_1$. Nevertheless, we still have
		\[
		\widetilde{\kp}(J_t) = [\omega^{123} \wedge (\omega^4+t\overline \omega^3)] = [\omega^{1234} + td\omega^{34\bar 3}] = [\omega^{1234}] = \widetilde\kp(J).
		\]
	\end{exam}

	\chapter*{Part II: Applications}
	
	In the following sections we will study the behaviour of the previously introduced period maps and the Teichmüller spaces for several classes of nilmanifolds.
	
	\section{Torus bundles over tori}\label{subsect: 2-step nilpotent Lie algebras}
	
	In this section, we are concerned with the Teichmüller spaces for principal torus bundles over a torus. We first recall some results on deformations of holomorphic principal torus bundles.
	
	\subsection{Deformations of principal torus bundles}\label{subsect: twist deformations}
	Let $\pi \colon X \to Y$ be a principal torus bundle with fibre $T$, where both $X$ and $Y$ are complex nilmanifolds. Let us denote by $\frg$, $\gotht\subset \frg$, and $\gothb = \frg/ \gotht$ the underlying real Lie algebras of $X$, $T$, and $Y$, respectively. The aim of this section is to study the space $H^1(X, \Theta_X)$ of infinitesimal deformations of $X$ and to discuss how small deformations of $X$ relate to deformations of $T$ and $Y$.
	
	In \cite{Hoefer} Höfer studied the more general situation of small deformations of holomorphic principal torus bundles over compact complex manifolds. We start by recalling his results for the bundle $\pi\colon X \to Y$. Following the notation of \cite{Hoefer}, the $d_2$-differential of the Borel spectral sequence for $\pi$ induces homomorphisms
	\[
	\gamma \colon H^{1,0}_{\dbar}(T) \to H_{\dbar}^{1,1}(Y), \qquad  \varepsilon \colon H^{0,1}_{\dbar}(T) \to H_{\dbar}^{0,2}(Y).
	\]
	In our setting these maps can be described as follows: consider the decomposition of the complexifications $\gothb_\IC = \gothb^{1,0} \oplus \gothb^{0,1}$ and  $\gotht_\IC = \gotht^{1,0} \oplus \gotht^{0,1}$ given by the left-invariant complex structures on $Y$ and $T$. Then, as explained in Section \ref{subsect: Lie algebra cohomology}, we have $H^{1,0}_{\dbar}(T) = \gotht^{*1,0}$, $H^{0,1}_{\dbar}(T) = \gotht^{*0,1}$ and the cohomology groups $H^{p,q}_{\dbar}(Y)$ are computed by the complex $(\Wedge^{p,\bullet}\gothb^*, \dbar)$. After choosing a basis of $\frg^{1,0}$ compatible with the filtration $0 \subset \gotht^{1,0} \subset \frg^{1,0}$ and taking its dual we can identify $\gotht^{*1,0}$ with a subspace of $\frg^{*1,0}$. Restricting $\dbar$ on $\frg^{*1,0}$ to $\gotht^{*1,0}$ we get a map $\dbar \colon \gotht^{*1,0} \to  \ker\left( \dbar \colon \Wedge^{1,1} \gothb^* \to \Wedge^{1,2} \gothb^* \right)$ by identifying $\gothb^*$ with $\Ann\gotht \subset \frg^*$. This map does not depend on the choice of basis and we have the commutative diagram 
	\[
	\begin{tikzcd}
		{H^{1,0}_{\dbar}(T)=\gotht^{*1,0}} \arrow[r, "\dbar"] \arrow[rd, "\gamma", swap] & {\ker\left( \dbar \colon \Wedge^{1,1} \gothb^* \to \Wedge^{1,2} \gothb^* \right)} \arrow[d] \\
		& {H^{1,1}_{\dbar}(Y).}          
	\end{tikzcd}
	\]
	Similarly, the map $\varepsilon$ is given by the restriction of $\dbar$ to $\gotht^{*0,1}$. The maps $\gamma$, $\varepsilon$ and their duals induce the following maps (for more details, see \cite[Section~14]{Hoefer}):
	\begin{align*}
		&\varepsilon \otimes \id \colon H^{0,1}_{\dbar}(T) \otimes \gotht^{1,0} \to H^{0,2}_{\dbar}(Y) \otimes \gotht^{1,0},\\[0,5em]
		& \gamma_1 \colon H^0(Y, \Theta_Y) \to H^{0,1}_{\dbar}(Y) \otimes \gotht^{1,0},\\[0,5em]
		& \gamma_2 \colon H^1(Y, \Theta_Y) \to H^{0,2}_{\dbar}(Y) \otimes \gotht^{1,0},\\[0,5em]
		& \gamma_3 \colon H^{0,1}_{\dbar}(T) \otimes H^0(Y, \Theta_Y) \to H^2(Y, \Theta_Y),
	\end{align*}
	and 
	\[
	\gamma_4 = \gamma_1 \otimes \id \colon H^0(Y, \Theta_Y) \otimes H^{0,1}_{\dbar}(T) \to  H^{0,1}_{\dbar}(Y) \otimes \gotht^{1,0} \otimes  H^{0,1}_{\dbar}(T) .
	\]
	As in \cite{Hoefer} we denote $A_F= \ker (\varepsilon \otimes \id)$, $A_T= \coker \gamma_1$,  $A_P = \ker \gamma_2$ and $A_D = \ker\gamma_3 \cap \ker \gamma_4$. We collect the results of \cite[Section~14]{Hoefer} in the following proposition.
	
	\begin{prop}\label{prop: H1 decomposition}
		The space $H^1(X, \Theta_{X})$ admits a decomposition
		\[
		H^1(X, \Theta_{X}) \cong A_F \oplus A_T \oplus A_P \oplus A_D,
		\]
		where
		\begin{enumerate}
			\item the space $A_F$ parametrises deformations of the fibre torus $T$,
			\item the space $A_T$ parametrises deformations of $X$ preserving both $Y$ and $T$,
			\item the space $A_P$ parametrises deformations of the base $Y$ while preserving the bundle structure,
			\item the space $A_D$ parametrises deformations destroying the bundle structure.
		\end{enumerate}
	\end{prop}
	
	The decomposition given in Proposition \ref{prop: H1 decomposition} was also considered in \cite{Tatsuo79} in a more general setup, where it was shown that the infinitesimal deformations in $A_T$ are always unobstructed. In our setting this is easily verified since the space $\gotht^{1,0}$ is contained in the centre of $\frg_\IC$ and thus the Schouten bracket of two elements in $A_T$ vanishes. By the same argument the infinitesimal deformations in $A_F$ are also unobstructed. Since the deformations parametrised by $A_T$ do not change the fibre or the base of $\pi \colon X \to Y$ but only the bundle structure, they are called the \textit{twist deformations} of $\pi$.

	\subsection{Period maps and the Teichmüller space}
	Due to the existence of twist deformations, the period map $\kp^\ks$ associated to a SPTBS $\ks$ will almost never be injective. Our first goal is to give a precise characterisation of $2$-step nilpotent Lie algebras for which $\kp^\ks$ is injective and to deduce some criteria for the injectivity of $\kp$. This allows us construct the Teichmüller space as a complex manifold for certain 2-step nilpotent Lie algebras.
	
	Let us fix a $2$-step nilpotent Lie algebra $\frg$ of dimension $2n$ and let $M$ be a nilmanifold with underlying Lie algebra $\frg$. Suppose we have a complex structure $J$ on $\frg$ admitting a (not necessarily stable) principal torus bundle series $0 \subset \gotht \subset \frg$ of length 2. Hence, the complex nilmanifold $X=(M,J)$ has the structure of a holomorphic principal torus bundle $\pi \colon X \to Y$ over a complex torus $Y$ with fibre $T$. Consider the decomposition
	\[
	H^1(X, \Theta_X) = A_F \oplus A_T \oplus A_P \oplus A_D
	\]
	given in Proposition \ref{prop: H1 decomposition}. We are particularly interested in the case where the fibre $T$ is an elliptic curve.
	
	\begin{prop}\label{prop: Kuranishi smooth h^0=1}
		The space of twist deformations $A_T$ of the bundle $\pi\colon X \to Y$ is trivial if and only if $\dim \Aut(X) =1$. In this case, the fibre $T$ is an elliptic curve and $X$ has the following properties.
		\begin{enumerate}
			\item We have $\dim H^1(X, \Theta_X) = \dim (A_F\oplus A_P) = \frac{1}{2}n(n-1) +1$.
			\item The Kuranishi family of $X$ is smooth and universal.
			\item The local period map $\kp_X$ is an embedding.
		\end{enumerate}
	\end{prop}
	\begin{proof}
		Since $J$ preserves the filtration $0\subset \gotht \subset \frg$ where $\dim \gotht =2k$, there exists a basis $\omega^1, \dots, \omega^{n-k}, \omega^{n-k+1}, \dots, \omega^n$ of $\frg^{*1,0}$ compatible with this filtration. So we have
		\[
		\gotht_\IC^* \cong \gotht^{*1,0} \oplus \gotht^{*0,1} = \langle \omega^{n-k+1}, \dots, \omega^{n} \rangle \oplus \langle \overline \omega^{n-k+1}, \dots, \overline \omega^{n} \rangle
		\]
		and we denote
		\[
		\mathfrak{b}^*_\IC = \mathfrak b^{*1,0} \oplus \mathfrak b^{*0,1} = \langle \omega^1, \dots, \omega^{n-k} \rangle \oplus \langle \overline \omega^1, \dots, \overline \omega^{n-k} \rangle.
		\]
		In particular, we get $d\omega^i =0$ for $i \leq n-k$ and $d\omega^i \in \Wedge^2\mathfrak{b}_\IC^*$ for $i> n-k$. Hence, for the basis $z_1,\dots, z_n$ of $\frg^{1,0}$ dual to $\omega^1, \dots, \omega^n$ we have $\dbar z_i =0$ for $i > n-k$ and  $\dbar z_i \in \mathfrak b^{*0,1} \otimes \gotht^{1,0}$ for $i \leq n-k$. Recall
		from Section \ref{subsect: twist deformations} that 
		\[
		A_T = \coker\left(\dbar\colon \mathfrak{b}^{1,0} \to  \mathfrak b^{*0,1} \otimes \gotht^{1,0}\right).
		\]
		Therefore, $A_T=0$ is equivalent to the map $\dbar\colon \mathfrak{b}^{1,0} \to  \mathfrak b^{*0,1} \otimes \gotht^{1,0}$ being surjective. Hence,  $\gotht^{1,0}$ has to be one-dimensional and the map is an isomorphism. Since $\dim \Aut(X) = \dim \ker \left( \dbar\colon \frg^{1,0} \to \frg^{*0,1} \otimes \frg^{1,0}\right)$ this is equivalent to $\dim \Aut(X) =1$. 
		
		From now on we will assume that $A_T=0$, or equivalently $\dim \Aut(X)=1$. Hence, we have $k=1$ and thus $\gotht^{*1,0} = \langle \omega^n \rangle$ and $\gotht^{1,0} = \langle z_n \rangle$ because $\dim \gotht^{1,0} \leq \dim \Aut(X)$.
		
		For the first item, it follows from the assumption $A_T=0$ that the dimension of $H^1(X, \Theta_X)$ equals the dimension of $A_F \oplus A_P \oplus A_D$, where
		\[
		A_F \oplus A_P = \ker\left( \dbar \colon (\mathfrak{b}^{*0,1} \otimes \mathfrak{b}^{1,0}) \oplus( \gotht^{*0,1} \otimes \gotht^{1,0}) \to \Wedge^2 \mathfrak{b}^{*0,1} \otimes \gotht^{1,0}\right) 
		\]
		and 
		\[
		A_D = \ker \left( \dbar \colon \gotht^{*0,1} \otimes \gothb^{1,0} \to (\Wedge^2 \gothb^{*0,1} \otimes \gothb^{1,0}) \oplus (\gotht^{*0,1} \otimes  \mathfrak b^{*0,1} \otimes \gotht^{1,0}) \right).
		\]
		We first observe that $A_D=0$ since the map $\id \otimes \dbar \colon \gotht^{*0,1} \otimes \gothb^{1,0} \to \gotht^{*0,1} \otimes  \mathfrak b^{*0,1} \otimes \gotht^{1,0}$ is an isomorphism if $A_T =0$.
		Moreover, the map $\dbar \colon \mathfrak{b}^{*0,1} \otimes \mathfrak{b}^{1,0} \to \Wedge^2 \mathfrak{b}^{*0,1} \otimes \gotht^{1,0}$ is surjective if $A_T =0$ since $\dbar$ vanishes on $\gothb^{*0,1}$ and $\dbar\colon \mathfrak{b}^{1,0} \to  \mathfrak b^{*0,1} \otimes \gotht^{1,0} $ is an isomorphism. Therefore, we have
		\[
		\dim H^1(X, \Theta_X) = \dim (A_F \oplus A_P) = (n-1)^2 +1 - \binom{n-1}{2} = \frac{1}{2} n(n-1) +1.
		\]	
		For the second item it suffices to show that $\Kur(X)$ is smooth since the function $h^0$ is constant on $\Kur(X)$ by the assumption $h^0(J) = \dim \Aut(X) =1$. It then follows from Wavrik's criterion that the Kuranishi family is universal \cite{wavrik1969obstructions}. We will again use the fact that the map $\dbar \colon \mathfrak{b}^{*0,1} \otimes \mathfrak{b}^{1,0} \to \Wedge^2 \mathfrak{b}^{*0,1} \otimes \gotht^{1,0}$ is surjective. Let $\phi$ be a representative of a class in $H^1(X, \Theta_X) = A_F \oplus A_P$. Then we can assume that $\phi = \phi_F + \phi_P$, where $\phi_F \in \gotht^{*0,1} \otimes \gotht^{1,0}$ and  $\phi_P \in  \mathfrak{b}^{*0,1} \otimes \mathfrak{b}^{1,0}$. Computing the Schouten bracket $[\phi, \phi]$ yields
		\begin{align*}
			[\phi, \phi] &= [\phi_F, \phi_F] + [\phi_P, \phi_P]  +2[\phi_F, \phi_P]= [\phi_P, \phi_P] + 2[\phi_F, \phi_P] \in \Wedge^2\mathfrak{b}^{*0,1} \otimes \gotht^{1,0}.
		\end{align*}
		Therefore, we have $[\phi, \phi] \in \im \left(\dbar \colon \mathfrak{b}^{*0,1} \otimes \mathfrak{b}^{1,0} \to \Wedge^2 \mathfrak{b}^{*0,1} \otimes \gotht^{1,0} \right)$. Since the Schouten bracket of two elements in $\mathfrak{b}^{*0,1} \otimes \mathfrak{b}^{1,0}$ is also always contained in $\Wedge^2\mathfrak{b}^{*0,1} \otimes \gotht^{1,0} \subset \im \dbar$ this allows us to construct for every $\phi$ a solution to the system of equations \eqref{eq:Phi_k} given by the Maurer--Cartan equation.
		
		Finally, it remains to show that the local period map $\kp_X \colon \Kur(X) \to \kd$ is an embedding. By Lemma \ref{lem: local period map} we can equivalently show that every $d$-exact form in $\Wedge^{n-1,1}\frg^*$ is $\dbar$-exact. Concretely, we will prove that every $d$-exact form in $\Wedge^{n-1,1}\frg^*$ is contained in the subspace $\Wedge^{n-1,1}\gothb^* \subset \Wedge^{n-1,1}\frg^*$. Since $\dim \Aut(X) =1$, the map $\dbar \colon \Wedge^{n-1,0}\frg^* \to \Wedge^{n-1,1}\mathfrak \gothb^*$ has rank $n-1$, and thus it is surjective. Hence, every form in $\Wedge^{n-1,1}\gothb^*$ is $\dbar$-exact.
		
		Let $\alpha \in \Wedge^{n-1,1} \frg^*$ and let $\beta \in \Wedge^{n-1}\frg_\IC^*$ such that $d\beta = \alpha$. The vector space $\Wedge^{n-1}\frg_\IC^*$ can be decomposed into
		\[
		 \Wedge^{n-1} \gothb_\IC^* \oplus  \left(\Wedge^{n-2} \gothb_\IC^* \otimes \gotht^{*1,0} \right) \oplus  \left(\Wedge^{n-2} \gothb_\IC^* \otimes \gotht^{*0,1} \right) \oplus  \left(\Wedge^{n-3} \gothb_\IC^* \otimes \gotht^{*1,0} \otimes \gotht^{*0,1} \right),
		\]
		so $\beta$ can be written as $\beta = \beta_0 + \beta_{1,0}+ \beta_{0,1} + \beta_{1,1}$. Since every form in $\gothb_\IC^*$ is closed, the form $\beta_0$ is closed, and both $d\beta_{1,0}$ and $d\beta_{0,1}$ are contained in $\Wedge^{n-1,1}\gothb^*$. So it remains to treat $d\beta_{1,1}$. We can write $\beta_{1,1} = \gamma \wedge \omega^n \wedge \overline \omega^n$ for some $\gamma \in \Wedge^{n-3} \gothb_\IC^*$, and by assumption $d\beta_{1,1} \in \Wedge^{n-1,1} \frg^*$, so we have
		\[
		d \beta_{1,1} = (-1)^{n-3} \gamma \wedge ( \partial \omega^n \wedge \overline \omega^n - \omega^n \wedge \partial \overline\omega^n).
		\]
		Let $q$ be the largest index such that the $(p,q)$-part $\gamma^{p,q}$ of $\gamma$ is non-zero with $p+q = n-3$. Then, $\gamma^{p,q} \wedge \omega^n \wedge \partial\overline \omega^n$ has to be zero, which is equivalent to $\gamma^{p,q} \wedge \partial\overline\omega^n =0$. Writing $\partial \overline\omega^n = \sum_{i=1}^{n-1}   \beta^i \wedge \overline \omega^i$ with $\beta^i \in \gothb^{*1,0}$, we get $\dbar z_i = \overline \beta^i \otimes z_n$ for the dual $z_i \in \gothb^{1,0}$ of $\omega^i$. Again, since $\dim \Aut(X) =1$, the map $\dbar \colon \gothb^{1,0} \to \gothb^{*0,1} \otimes \langle z_n \rangle$ is an isomorphism and thus the forms $ \beta^1, \dots, \beta^{n-1}$ are a basis of $\gothb^{*1,0}$. Therefore, the form $\gamma^{p,q} \wedge \partial\overline \omega^n$ is zero if and only if $\gamma^{p,q}=0$. Since $q$ was chosen as the largest index with $\gamma^{p,q} \neq 0$, this implies that $\gamma =0$. Hence, $d\beta_{1,1}=0$, and thus $\alpha = d(\beta_{1,0}+ d\beta_{0,1}) \in \Wedge^{n-1,1}\gothb^*$ is $\dbar$-exact. This concludes the proof. 
	\end{proof}

	Suppose the principal torus bundle series $ \ks =(0 \subset \gotht \subset \frg)$ is stable, and consider the associated period map $\kp^\ks \colon \kt(\frg) \to \ku$. We saw in Example \ref{exam: iwasawa x curve} that there is usually no relation between the period maps $\kp$ and $\kp^\ks$. However, we have
	
	\begin{lem}\label{lem: period maps for dim t =2}
		Let $J, J' \in \kc(\frg)$ such that $\widetilde \kp(J)  = \widetilde{\kp}(J')$. If $\dim_\IR \gotht =2$, then $\widetilde \kp^\ks(J) = \widetilde \kp^\ks(J')$.
	\end{lem}
	\begin{proof}
		As in the proof of Proposition \ref{prop: Kuranishi smooth h^0=1} we have a basis $\omega^1, \dots, \omega^n$ of $(1,0)$-forms with respect to $J$ such that the only non-closed form is $\omega^n$ since $\dim \gotht =2$. Let $\omega = \omega^1 \wedge \dots \wedge\omega^n$ and let $\omega' \in \Wedge^n \frg_\IC^*$ be a representative of the class $\widetilde{\kp}(J') \in \kd$. If $\widetilde{\kp}(J) = \widetilde{\kp}(J')$, then we can assume that $\omega' = \omega + d \beta$ for some $\beta \in \Wedge^{n-1}\frg_\IC^*$. Moreover, it follows from Proposition \ref{prop: forms on base}, that $J$ and $J'$ induce the same complex structure on $\frg/\gotht$. So we may assume that $\omega' = \omega^1 \wedge \dots \wedge \omega^{n-1} \wedge (\omega^n + \overline \alpha)$ with $\alpha = \sum_{i=1}^n \lambda_i \overline \omega^i$. To show that $\widetilde \kp^\ks(J) = \widetilde \kp^\ks(J')$ it suffices to show that $\lambda_n =0$. Seeking a contradiction, suppose that $\lambda_n \neq 0$. Then the form $\lambda_n \omega^1 \wedge \dots \wedge \omega^{n-1} \wedge \overline \omega^n$ appears as a summand of the form $d\beta$. Since $d\omega^i = 0$ for $i<n$ this implies that $\beta$ has a summand of the form $\gamma \wedge \omega^n \wedge \overline \omega^n$, where $\gamma \in \Wedge^{n-3,0}\frg^*$ with respect to $J$. From here, one derives a contradiction using the same argument as for the injectivity of the local period map in the proof of Proposition \ref{prop: Kuranishi smooth h^0=1}.
	\end{proof}
	
	\begin{lem}\label{lem: fibre is orbit}
		For every complex structure $J$ on $\frg$ we have $(\widetilde\kp^\ks)^{-1}(\widetilde\kp^\ks(J)) = G\cdot J$ if and only if $h^0(J) =1$. In particular, we have $(\widetilde\kp)^{-1}(\widetilde\kp(J)) = G\cdot J$ if $h^0(J)=1$.
	\end{lem}
	\begin{proof}
		Since $\frg$ is 2-step nilpotent we have $(\widetilde\kp^\ks)^{-1}(\widetilde\kp^\ks(J)) = G\cdot J$ if and only if the space of twist deformations $A_T$ associated the torus bundle given by $J$ and $\ks$ is trivial. By Proposition \ref{prop: Kuranishi smooth h^0=1} this is equivalent to $h^0(J)=1$. If $h^0(J) =1$, then $\dim \gotht =2$ and by Lemma \ref{lem: period maps for dim t =2} we have $G \cdot J \subset (\widetilde\kp)^{-1}(\widetilde\kp(J)) \subset (\widetilde\kp^\ks)^{-1}(\widetilde\kp^\ks(J)) = G\cdot J$ and thus $(\widetilde\kp)^{-1}(\widetilde\kp(J)) = G\cdot J$ 
	\end{proof}
	
	We can now prove the main result of this section.
	
	\begin{thm}\label{thm: fibre map not injective}
		Let $\frg$ be a $2$-step nilpotent Lie algebra of dimension $2n$ with a SPTBS $\ks$ of length $2$. If $\dim \Aut(X) =1$ for every complex nilmanifold $X$ of type $(\frg, \Gamma)$, then the following holds.
		\begin{enumerate}
			\item The period maps $\kp^\ks \colon \kt(\frg) \to \ku$ and $\kp \colon \kt(\frg) \to \kd$ are both injective. Moreover, the map $\kp^\ks$ is injective only if $\dim \Aut(X)=1$ for every complex nilmanifold $X$ of type $(\frg, \Gamma)$.
			\item The space $\kc(\frg)$ is a complex manifold of dimension $\frac{1}{2}n(n+1)$ and the Teichmüller space $\kt(\frg)$ exists as a complex manifold of dimension $\frac{1}{2}n(n-1) +1$.
		\end{enumerate}
	\end{thm}
	\begin{proof}
		The first item follows directly from Lemma \ref{lem: fibre is orbit}. By Proposition \ref{prop: Kuranishi smooth h^0=1} the Kuranishi space of every complex nilmanifold $X$ of type $(\frg,\Gamma)$ is smooth and of dimension $\frac{1}{2}n(n-1)+1$. Hence, it follows from Proposition \ref{lem: Kur inside Cg} that $\kc(\frg)$ is a complex manifold of dimension
		\[
		\dim \Kur(X) + (n- \dim \Aut(X)) = \frac{1}{2}n(n-1) +1 + (n-1) = \frac{1}{2}n(n+1).
		\] 
		Since $\kp$ is injective the Teichmüller space $\kt(\frg)$ is Hausdorff, and because $\kp_X$ is also an embedding for every $X$ by Proposition \ref{prop: Kuranishi smooth h^0=1}, it follows from Propositions \ref{prop: Tg coarse moduli} and \ref{prop: Tg exists} that $\kt(\frg)$ exists as a complex space compatible with the Kuranishi charts. Again, since $\Kur(X)$ is smooth and of dimension $\frac{1}{2}n(n-1)+1$ for every $X$, this implies that $\kt(\frg)$ is a complex manifold of the same dimension.
	\end{proof}
	
	\begin{rem}
		It follows directly from the description given in Section \ref{subsect: Kodaira manifolds} that every Kodaira manifold $X$ satisfies $\dim \Aut(X)=1$. Hence, Theorem \ref{thm: fibre map not injective} describes the Teichmüller space for the Lie algebras $\frh_{2k+1} \oplus \IR$. In Section \ref{sect: nilmanifolds of dimension six} we will see that starting in dimension six there are several Lie algebras not of this form satisfying the assumptions of Theorem \ref{thm: fibre map not injective}. We will also see in Section \ref{sect: almost abelian  nilmanifolds} that the condition $\dim \Aut(X) =1$ is not necessary for the period map $\kp$ to be injective.
	\end{rem}

	In light of the previous theorem, it is natural to consider the space 
	\[
	\kc_1(\frg) = \{J \in \kc(\frg) \, | \, h^0(J) =1\}.
	\]
	Since $h^0$ is always at least 1 this defines an open subset of $\kc(\frg)$, which is dense on the connected components of $\kc(\frg)$ intersecting $\kc_1(\frg)$. Note that $\kc_1(\frg)$ can only be non-empty if $\dim \gotht =2$. If $X=(M,J)$ is a complex nilmanifold of type $(\frg, \Gamma)$ with complex structure $J \in \kc_1(\frg)$, then every complex structure parametrised by $\Kur(X)$ is also contained in $\kc_1(\frg)$. Hence, the image of the map $\pi_X \colon \Kur(X) \to \kt(\frg)$ is contained in the open subset 
	\[
	\kt_1(\frg) = \{[J] \in \kt(\frg) \, | \, h^0(J) =1\} \subset \kt(\frg).
	\]
	Similar to $\kt(\frg)$, the space $\kt_1(\frg)$ can be constructed as a topological coarse moduli space for framed complex nilmanifolds of type $(\frg, \Gamma)$ with one-dimensional automorphism group, and with the same argument as in Theorem \ref{thm: fibre map not injective} one shows 
	\begin{prop}\label{prop: T1 as complex manifold}
		If $\frg$ admits a stable principal torus bundle series of length 2, then $\kt_1(\frg)$ is a complex manifold and $\kp$ embeds $\kt_1(\frg)$ into $\IP H^n(\frg, \IC)$.
	\end{prop}
	The space $\kt_1(\frg)$ will play an important role in the subsequent sections.
	
		\begin{exam}\label{exam: Iwasawa}
		As in the Examples \ref{exam: non-reduced point} and \ref{exam: non closed deformation} we consider the real Lie algebra $\frh$ underlying the Iwasawa manifold. It follows from the description of $\kc(\frh)$ given in \cite{ketsetzis2004complex} that every connected component of $\kc(\frh)$ admits a complex structure with one-dimensional automorphism group, so by the previous proposition the global period map $\kp$ is at least generically injective on the Teichmüller space of the Iwasawa manifold. As we will see in the next section, this is the best one can hope for if the Lie algebra admits complex parallelisable structures.
	\end{exam}

	\section{Nilmanifolds with complex parallelisable structures}\label{subsect: complex parallelisable structures}

	Let $\frg$ be a non-abelian nilpotent Lie algebra. We denote the space of complex parallelisable structures on $\frg$ by $\kc^{\para}(\frg)$. Recall that a complex structure $J$ on $\frg$ is complex parallelisable if and only if $[\frg^{1,0}, \frg^{0,1}] =0$. Therefore, the space $\kc^{\para}(\frg)$ is a closed subset of $\kc(\frg)$ and we can identify 
	\[
	\kc^{\para}(\frg)=\{V \in \kc(\frg)\, | \, [V, \overline V] =0 \}.
	\]
	It follows from the results of \cite[Chapter~5]{Winkelmann98} or Theorem \ref{thm: affine maps}, that if two complex parallelisable nilmanifolds of type $(\frg, \Gamma)$ are biholomorphic via a biholomorphism smoothly isotopic to the identity, then these manifolds are the same. So the complex space $\kc^{\para}(\frg)$ is already the Teichmüller space for complex parallelisable structures on $\frg$. In particular, $\kc^{\para}(\frg)$ can be embedded as a subspace of $\kt(\frg)$, but this embedding is usually not well-behaved.
	
	\begin{prop}\label{prop: Tg non hausdorff complex par}
		The space $\kt(\frg)$ is not locally Hausdorff at any point representing a complex parallelisable structure.
	\end{prop}
	\begin{proof}
		Let $J$ be a complex parallelisable structure on $\frg$. To show that $\kt(\frg)$ is not locally Hausdorff at the point $J\in \kt(\frg)$ we start by constructing a one-parameter family $\{ J_t\}$ of complex parallelisable deformations of $J$. In particular, every complex structure $J_t$ defines a distinct point in $\kt(\frg)$. Then we show that for every $J_t$ there exist a one-parameter family $\{J^s_t\}$ of deformations of $J_t$ such that for $s\neq 0$ the classes of $J^s_t$ and $J_{t'}^s$ in $\kt(\frg)$ coincide for all $t$ and $t'$. Hence, the complex structures in the family $\{J_t\}$ are inseparable in $\kt(\frg)$ and thus $\kt(\frg)$ is not locally Hausdorff at $J_0 = J \in \kt(\frg)$. Since $J$ is complex parallelisable, it preserves the descending central series
		\begin{equation*}
			0 = \kc^\nu\frg \subset \kc^{\nu-1}\frg \subset \dots \subset \kc^1\frg \subset \frg.
		\end{equation*}
		So we can choose a basis $\omega^1,\dots,\omega^n$ of $\frg^{*1,0}$ with dual basis $z_1,\dots, z_n$ which is compatible with this filtration. Suppose the vectors  $z_m, \dots, z_n$ are a basis of $\kc^{\nu-1}\frg^{1,0}$. For all $k \geq m$ the differentials $d\omega^k$ have the form
		\[
		d\omega^k = \sum_{i<j<k} A_{ij}^k \omega^{ij}.
		\]
		Then, because $\frg^{1,0}$ is $\nu$-step nilpotent, there exist indices $i$ and $j$ such that $z_j\in\kc^{\nu-2}\mathfrak g$, $\omega^i \in \Ann (\kc^1\frg^{1,0})$, and $A_{ij}^k$ is non-zero for some indices $k$. Notice that for any such $k$ the element $z_k$ is in the centre of $\frg_\IC$ and we have $d\overline\omega^i =0$. In particular, the element $\overline \omega^i \otimes \left(  -\sum_{k}A_{ij}^k z_k \right)$ is $\dbar$-closed and its Schouten bracket vanishes. Hence, for small $t\in \IC$, we can consider the complex structures $J_t$ on $\frg$ corresponding to the direction
		\begin{equation*}
			t\overline\omega ^i \otimes \left(  -\sum_{k} A_{ij}^k z_k \right) \in H^1(\frg^{0,1},\frg^{1,0}).
		\end{equation*}
		The space of $(n,0)$-forms with respect to the complex structure $J_t$ is generated by the form
		\[
		\omega_t = \omega^1 \wedge \dots \wedge \omega^{m-1} \wedge \left(\omega^m - tA_{ij}^m \overline \omega^i\right) \wedge \dots \wedge \left(\omega^n - tA_{ij}^n \overline \omega^i\right)
		\]
		Since the vector $\sum_{k} A_{ij}^k z_k$ is in the centre of $\frg_\IC$ the complex structures $J_t$ remain complex parallelisable \cite[Theorem~5.1]{Rol11Kuranishi}. To contruct the deformations $J_t^s$ we consider the element
		\begin{equation*}
			\phi_t(s) = s\overline\omega^i\otimes \left(z_i + \overline t\sum_k \overline A_{ij}^k\overline z_k\right) \in \frg_t^{*0,1} \otimes \frg_t^{1,0}.
		\end{equation*}
		Since $\overline \omega^i$ is closed we get $\dbar_t\phi_t(s) =0$, where $\dbar_t$ is taken with respect to $J_t$, and because $\sum_k A_{ij}^k\overline z_k$ is in the centre we have $[\phi_t(s),\phi_t(s)]=0$. Therefore, we can set $J^s_t$ to be the small deformations of $J_t$ corresponding to $\phi_t(s)$. Notice that $z_i$ is not contained in the centre and thus the complex structures $J_t^s$ are no longer complex parallelisable for $s\neq 0$ \cite[Theorem~5.1]{Rol11Kuranishi}.
		We claim that the complex structures $J_0^s$ and $J_t^s$ are isomorphic for all $s$ and $t$. The form $\omega_t + \phi_t(s) \lrcorner \omega_t$ generates the space of $(n,0)$-forms with respect to the complex structure $J_t^s$. It suffices to show that for all $s\neq 0$, the form $\omega_0 + \phi_0(s) \lrcorner \omega_0$ differs from $\omega_t + \phi_t(s) \lrcorner \omega_t$ by a suitable right-translation. For $t\in \IR$, let $g_{s,t} = \text{exp}\left(\frac{ t}{s} z_j+ \frac{ t}{\overline s} \overline z_j\right)$, where $j$ is the same index as in the construction of $J_t$ and $\phi_t(s)$. Then we compute
		\begin{align*}
			R_{g_{s,t}}^*( \omega_0  + \phi_0(s) \lrcorner \omega_0) &= \omega_0 + R_{g_{s,t}}^*(\omega^1 \wedge \dots \wedge\omega^{i-1}\wedge s\overline\omega^i \wedge \omega^{i+1} \wedge \dots \wedge \omega^n)\\
			&= \omega_0 + \omega^1 \wedge \dots \wedge s\overline\omega^i \wedge \dots \wedge \omega^{m-1} \wedge R_{g_{s,t}}^*(\omega^m \wedge \dots \wedge \omega^n)\\
			&= \omega_0 + \omega^1 \wedge \dots \wedge s\overline\omega^i \wedge \dots \wedge \omega^{m-1} \\
			& \qquad \wedge \left(\omega^{m}+ \frac{ t}{s} \sum_{k<j} A_{kj}^m \omega^k-\frac{ t}{s} \sum_{k>j}  A_{jk}^{m}\omega^k\right)\\
			& \qquad \wedge \dots \wedge \left(\omega^{n} + \frac{t}{s} \sum_{k<j} A_{kj}^n\omega^k-\frac{ t}{s} \sum_{k>j}  A_{jk}^{n}\omega^k\right)\\
			&= \omega^1 \wedge \dots \wedge (\omega^i + s\overline\omega^i) \wedge \dots \wedge \omega^{m-1}\\
			&\qquad \wedge \left(\omega^{m} + \frac{ t}{s} A_{ij}^{m}\omega^i\right) \wedge \dots \wedge \left(\omega^{n} +  \frac{ t}{s} A_{ij}^{n}\omega^i\right)\\
			&=  \omega^1 \wedge \dots \wedge (\omega^i + s\overline\omega^i) \wedge \dots \wedge \omega^{m-1}\\
			&\qquad \wedge \left(\omega^{m} - t A_{ij}^{m}\overline \omega^i\right) \wedge \dots \wedge \left(\omega^{n} - t A_{ij}^{n}\overline\omega^i\right)\\
			&= \omega_t+ \phi_t(s) \lrcorner \omega_t
		\end{align*}
		Therefore, the classes $[J_t^s]$ and $[J_{t'}^s]$ coincide for all $t,t' \in \IR$ and we can conclude that $\kt(\frg)$ is not locally Hausdorff at $J$.
	\end{proof}
	
	\begin{rem}\label{rem: analytically inseparable}
		Let $X_t$ and $X_{t'}$ be complex nilmanifolds endowed with the complex structures $J_t$ and $J_{t'}$, respectively, constructed in the proof above. The deformation families $\{J_t^s\}$ and $\{J_{t'}^s\}$ define holomorphic families $\kx \to \Delta$ and $\kx' \to \Delta$ over a small disc $\Delta \subset \IC$ centred at $0$ such that $\kx$ and $\kx'$ are isomorphic over $\Delta\backslash\{0\}$. In particular, the moduli space $\km(\frg, \Gamma)$ is also non-Hausdorff for any lattice $\Gamma$.
		
		In \cite[Definition~12.1]{meersseman2024geography}, two manifolds admitting such families are called analytically non-separated, which is generally stronger than $J_t$ and $J_{t'}$ just being inseparable points in $\kt(\frg)$.
	\end{rem}

	Proposition \ref{prop: Tg non hausdorff complex par} shows in particular that the the period maps on $\kt(\frg)$ cannot be injective if $\frg$ is non-abelian and admits complex parallelisable structures. In fact, since complex tori with complex multiplication are discrete in their parameter space, it follows from the results of \cite[Chapter~9]{Winkelmann98} that the period map associated to the descending central series is often constant on the connected components of $\kc^{\para}(\frg)$.
	
	If $\frg$ admits complex parallelisable structures, then both the centre and the commutator of $\frg$ have to be even-dimensional. In the simplest case $\dim \kz\frg = \dim \kc^1\frg= 2$ it was shown in \cite[Proposition~3.8]{rollenske09} that the filtration $0 \subset \kc^1\frg \subset \frg$ defines a SPTBS on $\frg$. In this case, we have some control over the period maps outside of $\kc^{\para}(\frg)$.

	\begin{prop}\label{prop: generically injective complex par}
		Suppose $\dim \kz\frg = \dim\kc^1\frg =2$ and let $\ks = (0 \subset \kc^1\frg \subset \frg)$. Then the set $\kt_1(\frg)$ is dense in the union of connected components of $\kt(\frg)$ containing complex parallelisable structures.
		
		In particular, the period maps $\kp$ and $\kp^\ks$ are generically injective on the union of connected components of $\kt(\frg)$ containing complex parallelisable structures.
	\end{prop}
	\begin{proof}
		By Lemma \ref{lem: fibre is orbit} it suffices to show that every connected component of $\kc(\frg)$ containing a complex parallelisable structure also contains a complex structure $I$ with $h^0(I)=1$. Let $J$ be a complex parallelisable structure on $\frg$.  Then the complex Lie algebra $\frg^{1,0}$ is nilpotent with $\dim_\IC \kc^1\frg^{1,0} = 1 = \dim_\IC \kz\frg^{1,0}$. Hence, $\frg^{1,0}$ is isomorphic to a Heisenberg Lie algebra, so $\frg^{*1,0}$ has a basis $\omega^1, \dots, \omega^{2k+1}$ with structure equations
		\[
		d\omega^1 = \dots = d\omega^{2k} =0, \qquad d\omega^{2k+1} = \sum_{j=1}^k \omega^{2j-1} \wedge \omega^{2j}.
		\]
		We will construct small deformations $J_t$ of $J$ with $h^0(J_t) =1$. Consider the element 
		\[
		\phi(t) = t\left(\sum_{j=1}^n \overline\omega^{2j-1} \otimes z_{2j} + \overline \omega^{2j} \otimes z_{2j-1} \right) +t^2 \overline\omega^{2k+1} \otimes z_{2k+1} \in \frg^{*0,1} \otimes \frg^{1,0},
		\]
		where $z_1,\dots, z_{2k+1}$ denotes the dual basis to $\omega^1, \dots, \omega^{2k+1}$. Then one computes that $\phi(t)$ satisfies the Maurer--Cartan equation, and we set $J_t$ as the complex structure corresponding to $\phi(t)$. According to \cite[Theorem~1.2]{Xia_2021}, the number $h^0(J_t)$ is equal to the dimension of the kernel of the map $\frg^{1,0} \to \frg^{*0,1} \otimes \frg^{1,0}$ given by sending $z \in \frg^{1,0}$ to $\dbar z + [z,\phi(t)]$.  Recall that the bracket is defined as
		\[
		[x, \overline \alpha \otimes y] =  x\lrcorner d\overline\alpha \otimes y- \overline\alpha \otimes [x,y]
		\]
		for $x,y \in \frg^{1,0}$ and $\overline \alpha \in \frg^{*0,1}$. Since $J$ is complex parallelisable $\dbar z=0$ and for $z = \sum_{i=1}^{2k+1} \lambda_i z_i$ we compute
		\[
		[z,\phi(t)] = \left(\sum_{j=1}^k \lambda_{2j-1} \overline \omega^{2j-1} - \lambda_{2j} \overline \omega^{2j}\right) \otimes z_{2k+1}.
		\]
		Therefore, the kernel of this map is generated by $z_{2k+1}$ and thus $h^0(J_t) =1$.
	\end{proof}

	\section{Lie algebras with small commutator}\label{subsect: Lie algebas with small commutator}

	In this section we will study the period maps and the Teichmüller space for nilmanifolds whose underlying Lie algebra has at most two-dimensional commutator. Let $\frg$ be a nilpotent Lie algebra with $\dim \kc^1\frg \leq 2$ admitting complex structures. In \cite[Propositions~3.6 and 3.8]{rollenske09}, Rollenske provided the following classification.
	      
	\begin{itemize}
		\item If $\dim \kc^1\frg =1$, then every complex structure on $\frg$ is abelian. In particular, the ascending central series is a SPTBS on $\frg$.
		\item If $ \dim \kc^1 \frg =2$ and $\frg$ is 2-step nilpotent, then we have the following cases.
		\begin{enumerate}
			\item If $\dim \kz\frg$ is even, then $\frg$ does not admit a SPTBS in general, unless $\dim \kz \frg =2$ in which case the ascending and descending central series coincide and define a SPTBS on $\frg$.
			\item If $\dim \kz\frg$ is odd, then the descending central series is a SPTBS on $\frg$.
		\end{enumerate}
		\item If $\frg$ is $3$-step nilpotent, then every complex structure on $\frg$ is abelian, so the ascending central series is a SPTBS on $\frg$.
	\end{itemize}

	The goal of this section is to prove the following theorem.
	
	\begin{thm}\label{thm: Teichmüller 2-dim commutator}
		let $\frg$ be a nilpotent Lie algebra admitting complex structures. If $\dim \kc^1\frg \leq 2$ and either $\dim \kz\frg =2$ or $\dim \kz\frg$ is odd, then we have the following cases.
		\begin{enumerate}
			\item If $\dim \kc^1\frg=1$, then $\kt(\frg)$ exists a complex manifold and the period map $\kp$ is an embedding.
			\item If $\frg$ is $2$-step nilpotent and $\dim \kz\frg =2$, then $\kt(\frg)$ does not exist as a complex space in general, but if $\kt_1(\frg)$ is non-empty, then $\kt_1(\frg)$ exists as a complex manifold and $\kp$ is an embedding on $\kt_1(\frg)$. 
			\item If $\frg$ is $2$-step nilpotent and $\dim \kz\frg$ is odd, then $\kt(\frg)$ does not exist as a complex space in general.
			\item If $\frg$ is $3$-step nilpotent, then $\kt(\frg)$ exists a complex manifold and the period map $\kp$ is an embedding.
		\end{enumerate}
	\end{thm}
	
	The second item of Theorem \ref{thm: Teichmüller 2-dim commutator} follows directly from Propositions \ref{prop: T1 as complex manifold} and \ref{prop: generically injective complex par} proved in the previous sections. Note that the results in Section \ref{subsect: complex parallelisable structures} do not cover the case where $\frg$ has odd-dimensional centre, since such Lie algebras do not admit complex parallelisable structures. In Section \ref{sect: almost abelian nilmanifolds} we will see many examples of such Lie algebras where $\kt(\frg)$ exists as a complex manifold and the period map $\kp$ is an embedding, but this does not hold in general as the next example shows.
	
	\begin{exam}\label{exam: counterexample dim 10}
		Consider the complex structure $J$ given by the structure equations
		\[
		d \omega^1 = d\omega^2 = d\omega^3 = d\omega^4 =0, \; d\omega^5 = (\omega^1 + \overline \omega^1) \wedge \omega^2 + \omega^{34}.
		\]
		Let $\frh$ be the underlying real Lie algebra and let $z_1,\dots,z_5$ denote the basis dual to $\omega^1,\dots,\omega^5$. Then we have $\kc^1\frh_\IC= \langle z_5, \overline z_5 \rangle$ and $\kz\frh_\IC = \langle z_1 - \overline z_1, z_5, \overline z_5 \rangle$.
		According to Proposition \ref{prop: fibre dimension criterion}, $\kt(\frh)$ can only exist as a complex space if the function $h^0$ is constant on the connected components of $\kc(\frh)$. So it suffices to construct a deformation $J_t$ of $J$ such that $h^0(J_t)< h^0(J)$. First, we compute
		\[
		\dbar z_1 = 0, \; \dbar z_2 = -\overline \omega^1 \otimes z_5, \; \dbar z_3 = 0, \; \dbar z_4 =0, \; \dbar z_5 =0.
		\]
		Hence, we have $h^0(J) =4$. Consider for small $t \in \IC$ the form 
		\[
		\phi(t) = t\overline \omega^3 \otimes z_3.
		\]
		Since $\dbar \phi(t) =0$ and $[\phi(t), \phi(t)] =0$, the form $\phi(t)$ defines a deformation $J_t$ of $J$. We will again use the fact that $h^0(J_t)$ equals the dimension of the kernel of the map $\frh^{1,0} \to \frh^{*0,1} \otimes \frh^{1,0}$ given by sending $z \in \frh^{1,0}$ to $\dbar z + [z,\phi(t)]$ \cite[Theorem~1.2]{Xia_2021}. So we compute
		\begin{align*}
			\dbar z_1 + [z_1, \phi(t)] &=0 ,\\
			\dbar z_2 + [z_2, \phi(t)] &= -\overline \omega^1 \otimes z_5 ,\\
			\dbar z_3 + [z_3, \phi(t)] &= 0,\\
			\dbar z_4 + [z_4, \phi(t)] &= -\overline \omega^3 \otimes [z_4, z_3] = \overline \omega^3 \otimes z_5,\\
			\dbar z_5 + [z_5, \phi(t)] &= 0.
		\end{align*}
		Hence, the kernel of this map is three-dimensional and thus $h^0(J_t)=3< h^0(J)$.
	\end{exam}

	Example \ref{exam: counterexample dim 10} proves the third item of Theorem \ref{thm: Teichmüller 2-dim commutator}. To complete the proof of Theorem \ref{thm: Teichmüller 2-dim commutator} it remains to cover the cases, where $\frg$ is either $3$-step nilpotent or has one-dimensional commutator. The latter are Lie algebras of the form $\frh_{2k+1}\oplus \IR^m$. The main ingredient to show that in either of these cases the Teichmüller space exists as a complex manifold is the following lemma.
	
	\begin{lem}\label{lem: Kuranishi smooth abelian}
	Let $X$ be a complex nilmanifold with underlying Lie algebra $\frg$, and let $\omega^1, \dots, \omega^n$ be a basis of $\frg^{*1,0}$ with $d\omega^k = \sum_{i,j \leq n} B_{ij}^k \omega^i \wedge \overline \omega^j$. Suppose there exist 1-forms $\xi^1, \dots, \xi^n$ such that the following holds.
	\begin{enumerate}
		\item The forms $ \omega^1, \dots, \omega^n, \xi^1, \dots, \xi^n$ are a basis of $\frg_\IC^*$.
		\item We have $d\omega^k = \sum_{i,j \leq n} B_{ij}^k \omega^i \wedge \xi^j$ for all $k =1, \dots,n$.
	\end{enumerate}
	Then the Frölicher spectral sequence of $X$ degenerates at the first page. In particular, the Kuranishi family of $X$ is smooth and universal, and the local period map $\kp_X$ is an embedding.
	\end{lem}
	\begin{proof}
		Consider the isomorphism $\varphi \colon \bigoplus_{1\leq p,q\leq n} \Wedge^{p,q}\frg^* \to \bigoplus_{1 \leq k \leq n} \Wedge^k\frg_\IC^*$ given by $\varphi(\omega^k) = \omega^k$ and $\varphi(\overline\omega^k) = \xi^k$ for $k=1,\dots,n$. Then we have
		\[
		\varphi (\dbar \omega^k) = \sum_{i,j\leq n} B_{ij}^k \varphi(\omega^i) \wedge \varphi(\overline \omega^j) = \sum_{i,j\leq n} B_{ij}^k \omega^i \wedge \xi^j = d\omega^k = d\varphi(\omega^k)
		\]
		and 
		\[
		\varphi(\dbar\overline\omega^k) = 0 = d\xi^k = d \varphi(\overline \omega^k).
		\]
		Hence, $\varphi$ is an isomorphism preserving the total degrees satisfying $\varphi \circ \dbar = d \circ \varphi$, and thus induces an isomorphism $\bigoplus_{p+q=k} H_{\dbar}^{p,q}(X) \cong H^k(\frg, \IC)$ for every $k=1, \dots,n$. Therefore, the Frölicher spectral sequence of $X$ degenerates at the first page.
		\end{proof}
	
	To a apply the previous lemma in our special cases we have to understand the possible structure equations of complex structures on the two classes of Lie algebras at hand.
	
	\begin{lem}\label{lem: structure equations abelian}
		Let $\frg$ be a nilpotent Lie algebra of dimension $2n$ with $\dim \kc^1\frg \leq 2$ and let $J$ be a complex structure on $\frg$.
		\begin{enumerate}
			\item If $\dim \kc^1\frg =1$, then there exists a basis $\omega^1, \dots, \omega^n$ of $\frg^{*1,0}$ with structure equations
			\[
			d\omega^1 = \dots = d\omega^{n-1}, \; d\omega^{n} = \sum_{i=1}^k \sigma_i \omega^i \wedge \overline \omega^i
			\]
			for some $k < n$ and $\sigma_i \in \{1,-1\}$. 
			\item If $\frg$ is $3$-step nilpotent, then there exists a basis $\omega^1, \dots,\omega^m, \dots, \omega^n$ of $\frg^{*1,0}$ with structure equations
			\[
			d\omega^m = \omega^{1\bar 1}, \; \quad d \omega^n = \sum_{i,j < m} B_{ij} \omega^{i\bar j} + \lambda \overline B_{ij} \omega^{j\bar i} + (B_{1m} \omega^{1\overline m} + \lambda \overline B_{1m}\omega^{m\bar 1}).
			\]
			for some $\lambda \in \IC\backslash\{0\}$ and $d\omega^k=0$ for $k \neq m,n$.
		\end{enumerate}
	\end{lem}	
	\begin{proof}
		The first item follows directly from the classification of complex structures on Lie algebras with one-dimensional commutator proved in \cite[Propostion~3.6]{rollenske09}.
		
		If $\frg$ is 3-step nilpotent, then every complex structure on $\frg$ preserves the ascending central series and since $\dim \kc^1\frg =2$, the Lie algebra $\frg/\kz\frg$ of dimension $2m$ has one-dimensional commutator. Therefore, there exists a basis $\omega^1, \dots, \omega^m , \dots, \omega^n $ of $\frg^{*1,0}$ such that the only non-zero differentials are $ d\omega^{m} = \sum_{i=1}^k \sigma_i \omega^i \wedge \overline \omega^i$ for some $k<m$ and $\sigma_i \in \{1,-1\}$.
		Moreover, since $\dim \kc^1\frg_\IC = \rk \left(d \colon \frg_\IC^* \to \Wedge^2 \frg_\IC^*\right)=2$ we can assume that the only other non-closed form is $\omega^n$ with 
		\[
		d\omega^n = \sum_{i,j<n} B_{ij} \omega^i \wedge \overline \omega^j.
		\]
		Furthermore, we have $d\omega^n = \lambda d\overline \omega^n$ for some $\lambda \in \IC$ because $\dim \langle d\omega^m, d\omega^n , d\overline \omega^n \rangle = \dim \kc^1\frg_\IC =2$ and the Lie algebra $\frg_\IC/\kz\frg_\IC$ has one-dimensional commutator. Hence, we get $B_{ji} = -\overline \lambda B_{ij}$, and since $\frg$ is 3-step nilpotent there exists an index $j$ such that the coefficient $B_{mj}$ is non-zero. Therefore, we have
		\[
		0 = d^2 \omega^n = d\omega^m \wedge \left(\sum_{j} B_{mj} \overline \omega^j + \bar \lambda \overline B_{mj} \omega^j \right).
		\]
		This implies that there is an index $i$ such that $d\omega^m = \omega^{i \bar i}$ and $B_{mj} =0$ for $i \neq j$, so we may assume $i=1$.
	\end{proof}
	
	The following proposition completes the proof of Theorem \ref{thm: Teichmüller 2-dim commutator}.

	\begin{prop}\label{prop: 3-step 2-dim commutator}
		Let $\frg$ be a nilpotent Lie algebra with $\dim \kc^1\frg \leq 2$ admitting complex structures. If $\dim \kc^1\frg =1$ or $\frg$ is $3$-step nilpotent, then $\kt(\frg)$ exists as a complex manifold and the period map $\kp$ is an embedding. 
	\end{prop}
	\begin{proof}
		Let $J$ be a complex structure on $\frg$. Let $\omega = \omega^1 \wedge \dots \wedge \omega^n$, where $\omega^1, \dots, \omega^n$ is a basis of $\frg^{*1,0}$ with structure equations as in Lemma \ref{lem: structure equations abelian}. We will prove the claim by verifying the assumptions of the criterion given in Proposition \ref{prop: Tg exists} for the Lie algebra $\frg$. Using Lemma \ref{lem: structure equations abelian} we will first show that both cases satisfy the assumption of Lemma \ref{lem: Kuranishi smooth abelian}, so it remains to show that the period map $\widetilde \kp$ has connected fibres. We will treat the two cases separately.
		\begin{prooflist}
			\item[Case 1] $\dim \kc^1\frg =1$. Applying Lemma \ref{lem: Kuranishi smooth abelian} for $\xi^k = \overline \omega^k$, we get that $\kc(\frg)$ is smooth, and $\kp_X$ is an embedding for every complex nilmanifold $X$ with underlying Lie algebra $\frg$. So in order to apply Proposition \ref{prop: Tg exists} it remains to show that the fibres of the period map $\widetilde{\kp}$ are connected. Since every complex structure on $\frg$ is abelian, the ascending central series $\kz$ defines a SPTBS on $\frg$ and it follows from Proposition~\ref{prop: forms on base} that $\widetilde\kp(\widetilde\kp^{-1}(J)) \subset (\widetilde\kp^\kz)^{-1}(\widetilde\kp^\kz(J))$. Hence, every complex structure in the fibre $\widetilde\kp(\widetilde\kp^{-1}(J))$ admits a basis of $(1,0)$-forms of the form $\omega^1, \dots, \omega^{n-1}, \omega^n + \overline \alpha$ with $\overline \alpha \in \langle \overline \omega^1, \dots, \overline \omega^k \rangle$. Conversely, every choice of such a basis determines a complex structure in the fibre $\widetilde\kp(\widetilde\kp^{-1}(J))$. Indeed, if $\overline \alpha = \sum_{i=1}^k \lambda_i \overline \omega^i$ and $\beta = \sum_{i=1}^k (-1)^k \lambda_i \omega^1 \wedge \dots \widehat{\omega^k} \wedge \dots \wedge \omega^{n-1} \wedge\omega^n$, then \[
			\omega +d\beta = \omega^1 \wedge \dots \wedge \omega^{n-1} \wedge (\omega^n + \overline \alpha).
			\]
			Therefore, the fibre $\widetilde\kp(\widetilde\kp^{-1}(J))$ is connected and Proposition \ref{prop: Tg exists} concludes this case.
			
			\item[Case 2] $\frg$ is 3-step nilpotent. Again, we can apply Lemma \ref{lem: Kuranishi smooth abelian} with
			\[
			\xi^k =
			\begin{cases}
				\omega^k + \overline \omega^k & k =1,m,n\\
				\overline \omega^k & \textup{else}
			\end{cases}
			\]
			  to conclude that $\kc(\frg)$ is smooth and $\kp_X$ is always an embedding.  Since every complex structure on $\frg$ is abelian, the quotient map $\frg \to \frg/\kz\frg$ induces a map $\kt(\frg) \to \kt(\frg/\kz\frg)$. The Lie algebra $\frg/\kz\frg$ has one-dimensional commutator. So by the same computation as in the previous case the fibre $\widetilde\kp(\widetilde\kp^{-1}(J))$ is connected. Hence, it follows from Proposition \ref{prop: Tg exists} that $\kt(\frg)$ exists as a complex manifold and $\kp$ is an embedding.
			\qedhere
		\end{prooflist}
	\end{proof}
	
	\begin{proof}[Proof of Theorem \ref{thm: Teichmüller 2-dim commutator}]
			As explained above, Theorem \ref{thm: Teichmüller 2-dim commutator} now follows from Propositions \ref{prop: T1 as complex manifold} and \ref{prop: generically injective complex par} together with Example \ref{exam: counterexample dim 10} and Proposition \ref{prop: 3-step 2-dim commutator}.
	\end{proof}

	\section{Almost abelian nilmanifolds}\label{sect: almost abelian nilmanifolds}
	
	The goal of this section is to describe the Teichmüller spaces for almost abelian nilmanifolds.
	
	\subsection{Almost abelian complex nilmanifolds}\label{subsect: almost abelian nilmanifolds}
	An almost abelian Lie algebra is a non-abelian Lie algebra admitting an abelian ideal of codimension one. An almost abelian (complex) nilmanifold is a (complex) nilmanifold whose underlying Lie algebra is almost abelian. Following \cite{andrada2025almost}, we will briefly set up the notation used throughout this section. 
	
	Let $\frg$ be a nilpotent almost abelian Lie algebra of dimension $2n+2$ and let $\gotha \subset \frg$ be an abelian ideal of codimension one. Choose $e_0 \not\in \gotha$ and let $A$ be the matrix representing the map $\ad_{e_0}\colon \gotha \to \gotha$ with respect to a chosen basis of $\gotha$. Then the Lie algebra $\frg$ can be written as the semi-direct product 
	\[\frg_A = \langle e_0 \rangle \ltimes_A \IR^{2n+1},\] 
	and it was shown in \cite{Fr} that two such Lie algebras are isomorphic if and only if their corresponding matrices are conjugate. Thus, the partition of $2n+1$ given by the Jordan normal form of the nilpotent matrix $A$ uniquely determines the isomorphism type of $\frg_A$.  
	
	In terms of this partition, there is a complete description of nilpotent almost abelian Lie algebras admitting a complex structure. Concretely, it was shown in \cite[Theorem~4.10]{ABDGH24} that a $(2n+2)$-dimensional nilpotent almost abelian Lie algebra $\frg_A$ admits a complex structure if and only if there exists an index $j \in \IN$ and a partition
	\begin{equation*} 
		n = \sum_{i>0} p_i \cdot i,
	\end{equation*}
	such that the partition of $2n+1$ given by the Jordan normal form of $A$ is
	\begin{align*}
			2n+1  &= 
			\begin{cases}     
				\displaystyle{\sum_{\stackrel{i>0}{  i\neq j,\, j-1}}} 2p_i \cdot i + (2 p_j +1)\cdot j + (2 p_{j-1}-1)\cdot (j-1), \; &\text{if } j>1,\\ 
				&  \\
				\hspace*{.25cm} \displaystyle{ \sum_{i>1}} \; 2p_i \cdot i + (2 p_1 +1)\cdot 1 , \; &\text{if } j=1.
			\end{cases}
	\end{align*}
	 Throughout this section we also set
	\begin{equation}\label{eq: epsilon}     
		\epsilon = 
		\begin{cases} 1, & \text{if } j>1,\\ 0, & \text{if }j =1.
	\end{cases} \end{equation}

	\subsection{Period maps and Teichmüller spaces}
	
	Let $\frg = \frg_A$ be an almost abelian Lie algebra of dimension $2n+2$ admitting complex structures. Our first goal is to show that we can apply Proposition \ref{prop: Tg exists} to prove that $\kt(\frg)$ always exists as a complex manifold, so we need the following lemma.
	
	\begin{lem}\label{lem: fibres are connected almost abelian}
		The fibres of $\widetilde\kp \colon \kc(\frg) \to \kd$ are connected.
	\end{lem}
	\begin{proof}
		Let $J$ be a complex structure on $\frg$. We will describe the fibre $\widetilde\kp^{-1}(\widetilde \kp(J))$. According to \cite[Corollary~3.13]{andrada2025almost} there exists a basis
		\[\alpha,\beta_1^0, \dots,\beta_{k_0}^0, \dots, \beta_1^r, \dots, \beta_{k_r}^r\]
		of $\frg^{*1,0}$ such that its dual basis
		$z,y_1^0, \dots, y_{k_0}^1, \dots, y_1^r, \dots, y_{k_r}^r$ satisfies the following bracket relations 
		\begin{align}\label{eq: brackets almost abelian}
			\begin{split}
				[z,\overline z] &= -\varepsilon(y_1^0- \overline y_1^0),\\
				[y_i^\ell,z] &=  y_{i+1}^\ell, \qquad [\overline y_i^\ell,z] =  \overline y_{i+1}^\ell,\\
				[y_i^\ell,\overline z] &=  y_{i+1}^\ell, \qquad [\overline y_i^\ell,\overline z] =  \overline y_{i+1}^\ell.
			\end{split}
		\end{align}
		with $\varepsilon$ as in \eqref{eq: epsilon}. Suppose $J'$ is another complex structure on $\frg$ with $\widetilde\kp(J)= \widetilde\kp(J')$. If $\omega = \alpha \wedge\beta_{0}^1 \wedge \dots \wedge \beta_{k_r}^r$ and $\omega'$ is a representative of the class $\widetilde\kp(J')$, then we can assume $\omega' = \omega +d\gamma$ for some $\gamma\in \Wedge^{n}\frg^*_{\IC}$. Since $d\gamma$ does not have a $(n+1,0)$-part with respect to $J$, we have 
		\[
		\omega' = (\alpha +\overline \nu) \wedge (\beta_1^0+ \overline \nu_1^0) \wedge \dots \wedge (\beta_{k_r}^r + \overline \nu_{k_r}^r),
		\]
		where $\overline \nu, \overline\nu_i^\ell \in \frg^{*0,1}$. It follows from the structure equations that we can write
		\[
		d\gamma =  (\alpha+\overline{\alpha})\wedge \eta_1 +\alpha \wedge \overline\alpha \wedge \eta_2  
		\]
		for some $\eta_i \in \Wedge^{n+1-i}\langle \beta_1^0,\dots, \beta_{k_r}^r ,\overline \beta_1^0, \dots, \overline\beta_{k_r}^r \rangle$. Note that the coefficient of $\overline {\alpha}$ in $\overline{\nu}$ cannot be 1 since this would imply that $\omega' \wedge \overline \omega' =0$. Therefore, we have 
		\[
		\omega' = \omega + \alpha \wedge \overline\alpha \wedge \eta_2
		\]
		and thus we can assume $\overline \nu_i^\ell= \lambda_i^\ell\overline\alpha$ for some $\lambda_i^\ell \in \IC$ and $\overline \nu =0$. Moreover, the coefficients $\lambda_{1}^\ell$ have to be $0$ for all $\ell > 0$: if $\lambda_{1}^\ell$ were to be non-zero for some $\ell$, then the form $\omega'$ would have a summand of the form $\lambda_{1}^\ell \alpha \wedge \overline \alpha \wedge ((z \wedge \overline z \wedge y_{1}^\ell) \lrcorner \omega)$, but again the structure equations imply that no exact $(n+1)$-form can have such a summand. Let us denote
		\[
		V = \langle y_1^0, \dots, y_{k_0}^0, y_2^1, \dots, y_{k_1}^1, \dots, y_2^r , \dots, y_{k_r}^r \rangle \subset \frg^{1,0}.
		\]
		Then the class $[\omega'] \in \IP\left(\Wedge^{n+1}\frg^*_\IC\right)$ is contained in the connected subset
		\[
		R = \{ [\omega + \eta \lrcorner \omega]\, | \, \eta \in \langle \overline \alpha \rangle \otimes V \} \subset \IP\left(\Wedge^{n+1}\frg^*_\IC\right).
		\]
		So it remains to see that every element of $R$ determines a complex structure on $\frg$ contained in $\widetilde\kp^{-1}(\widetilde \kp(J))$. The bracket relations of the basis $z,y_1^0, \dots, y_{k_0}^1, \dots, y_1^r, \dots, y_{k_r}^r$ given in \eqref{eq: brackets almost abelian} imply that $\dbar z = \overline \alpha \otimes y_1^0$ and $\dbar y_i^\ell = \overline \alpha \otimes y_{i+1}^\ell$. Thus, we have
		\[
		\langle \overline \alpha \rangle \otimes V = \im \left(\dbar \colon \frg^{1,0} \to \frg^{*0,1} \otimes \frg^{1,0}\right).
		\]
		Hence, every element $[\omega + \eta \lrcorner \omega] \in R$ can be written as 
		\[
		[\omega + \eta \lrcorner \omega] = [\omega + (\dbar x) \lrcorner \omega] = [\omega + \dbar(x \lrcorner \omega)] =  [\omega + d(x  \lrcorner \omega)]
		\]
		for some $x \in \frg^{1,0}$. Here the second to last equality comes from the general formula $\dbar(x \lrcorner \omega) = x \lrcorner (\dbar \omega) + (\dbar x) \lrcorner \omega$ \cite[Lemma~2.4 (1)]{Xia22} and the fact that $\dbar \omega =0$. Therefore, the fibre $\widetilde\kp^{-1}(\widetilde \kp(J)) =R$ is connected.
	\end{proof}
	
	We can now prove the main result of this section.
	
	\begin{thm}\label{thm: global period map almost abelian}
		The Teichmüller space $\kt(\frg)$ exists as a complex manifold and the period map $\kp \colon \kt(\frg) \to \kd$ is an embedding.
	\end{thm}
	\begin{proof}
		Let $X = (\Gamma \backslash G,J)$ be an almost abelian complex nilmanifold. By Theorem \cite[Theorem~1.2]{andrada2025almost} the Kuranishi space of $X$ is a smooth universal deformation space. Moreover, the Frölicher spectral sequence of X degenerates at the first page \cite[Theorem~1.1]{andrada2025almost}, so Lemma \ref{lem: local period map} implies that the local period map $\kp_X$ is an embedding. Now, Lemma \ref{lem: fibres are connected almost abelian} allows us to apply Proposition \ref{prop: Tg exists} to conclude that $\kt(\frg)$ exists as a complex manifold and the period map $\kp$ is an embedding.
	\end{proof}

	\begin{rem}\label{rem: dimension Teichmüller almost abelian}
		The dimension of $\kt(\frg)$ equals $ \dim \Kur(X) = \dim H_{\dbar}^{n-1,1}(X)$ for any almost abelian complex nilmanifold $X$ of type $(\frg, \Gamma)$, which can be computed with the combinatorial formula given in \cite[Proposition~4.11]{andrada2025almost}. 
	\end{rem}
	
	Next, we study the period map associated to the filtration
	\begin{equation}\label{eq: sptbs almost abelian}
		0  \subset \kz\frg =\kz^1\frg \subset \dots \subset \kz^{j-1}\frg \subset \kz^{j-1}\frg +\kc^{\nu-j}\frg \subset \dots \subset \kz^{j-1}\frg +\kc^1\frg \subset \frg,
	\end{equation}
	where $\nu$ is the nilpotency index of $\frg = \frg_A$ and $j\in \IN$ is defined as in Section \ref{subsect: almost abelian nilmanifolds}. It was shown in \cite[Corollary~3.15]{andrada2025almost} that this filtration, which we will denote by $\ks = (\ks^i\frg)_{i=1,\dots, \nu}$, defines a SPTBS on $\frg$. So for every complex structure on $\frg$ we get a tower of principal torus bundles:
	\begin{equation}\label{eq: tower almost abelian}
		\begin{tikzcd}
			X= X_1 \rar{\pi_1} & X_2 \rar{\pi _2} & {\cdots}  \rar{\pi_{\nu-2}} &X_{\nu-1} \rar{\pi_{\nu-1}} & X_{\nu} = T_\nu.
		\end{tikzcd} 
	\end{equation}
	The simplest examples of almost abelian complex nilmanifolds are Kodaira surfaces, that is, nilmanifolds with underlying Lie algebra $\frg = \frh_3 \oplus \IR$. In this case, $\ks$ is just $0 \subset \kz\frg \subset \frg$ and $X \to X_2$ is a bundle of elliptic curves over an elliptic curve. It turns out that $\frh_3 \oplus \IR$ is the only almost abelian Lie algebra for which the period map $\kp^\ks$ is injective.
	
	\begin{prop}\label{prop: fibre map almost abelian}
		If $\frg$ is almost abelian, then the period map $\kp^\ks \colon \kt(\frg) \to \ku$ is injective if and only if $\frg = \frh_3 \oplus \IR$.
	\end{prop}
	\begin{proof}
		Since every Kodaira surface has a one-dimensional automorphism group it follows from Theorem \ref{thm: fibre map not injective} that $\kp^\ks$ is injective if $\frg = \frh_3 \oplus \IR$.
		Let $X$ be an almost abelian complex nilmanifold with complex structure $J$ and associated bundle as in \eqref{eq: tower almost abelian}. Let us denote by $T$ the fibre of the map $\pi_1 \colon X \to X_2$ and let $\gotht^{1,0}$ be the space of left-invariant holomorphic vector fields on $X$ tangent to $T$. Recall from Section \ref{subsect: twist deformations} that every element of the space
		\[
		A_T = \coker\left(\gamma_1 \colon H^0(X_2, \Theta_{X_2}) \to H^{0,1}_{\dbar}(X_2) \otimes \gotht^{1,0} \right)
		\]
		defines a deformation of $J$ which is in the same fibre of $\widetilde\kp^\ks$ as $J$. In particular, $\kp^\ks$ can only be injective if $A_T=0$, i.e., if $\gamma_1$ is surjective. In what follows, we will compute the dimensions of $H^0(X_2, \Theta_{X_2})$ and $H^{0,1}_{\dbar}(X_2)$ and see that $\gamma_1$ can only be surjective if $X$ is a Kodaira surface. We separate the proof into two cases depending on the value of $\varepsilon\in \{0,1\}$ defined as in \eqref{eq: epsilon}. As in the proof of Lemma \ref{lem: fibres are connected almost abelian}, $\alpha,\beta_1^0, \dots,\beta_{k_0}^0, \dots, \beta_1^r, \dots, \beta_{k_r}^r$ denotes a basis of left-invariant $(1,0)$-forms on $X$ whose dual basis $z,y_1^0, \dots, y_{k_0}^1, \dots, y_1^r, \dots, y_{k_r}^r$ satisfies the bracket relations given in \eqref{eq: brackets almost abelian}.
		\begin{prooflist}
			\item[Case 1] $\varepsilon =0$. In this case we have $k_0=0$. Since $\frg$ is $\nu$-step nilpotent the indices $k_i$ are at most $\nu$ and for $i=1,\dots, r$ we define
			\[
			m_i = \begin{cases}
				k_i & \textup{ for $k_i < \nu$},\\
				k_i-1 & \textup{ for $k_i = \nu$}.
			\end{cases}
			\]
			The first vector space in the filtration $\ks$, i.e., the real Lie algebra underlying the fibre $T$, is the commutator $\kc^{\nu-1}\frg$. Therefore, the forms $\alpha,\beta_1^1, \dots,\beta_{m_1}^1, \dots, \beta_1^r, \dots, \beta_{m_r}^r$ are a basis of left-invariant $(1,0)$-forms on $X_2$. Since $\dbar\overline \alpha =0$ and $\dbar \overline\beta_i^\ell = \overline\alpha \wedge \overline \beta_{i-1}^\ell$ we get
			\[
			H^{0,1}_{\dbar}(X_2) = \langle \overline \alpha, \overline \beta_1^1, \dots, \overline \beta_1^r \rangle.
			\]
			Similarly, $z,y_1^0, \dots, y_{m_1}^1, \dots, y_1^r, \dots, y_{m_r}^r$ form a basis of left-invariant $(1,0)$-vector fields on $X_2$. The bracket relations \eqref{eq: brackets almost abelian} imply $\dbar z = 0$ and $\dbar y_i^\ell = \overline \alpha \otimes y_{i+1}^\ell$ and thus we get
			\[
			H^0(X_2, \Theta_{X_2}) = \langle z, y_{m_1}^1, \dots, y_{m_r}^r \rangle.
			\]  
			Therefore, we have $\dim H^{0,1}_{\dbar}(X_2) = \dim H^0(X_2, \Theta_{X_2})$ but since $\gamma_1(z)=0$, the map $\gamma_1$ can never be surjective if $\varepsilon=0$.
			
			\item[Case 2] $\varepsilon =1$. The real Lie algebra underlying $T$ is the centre $\kz\frg$. Hence, the forms $\alpha,\beta_1^0, \dots,\beta_{k_0-1}^0, \dots, \beta_1^r, \dots, \beta_{k_r-1}^r$ are a basis of $(1,0)$-forms and the vectors $z,y_1^0, \dots, y_{k_0-1}^1, \dots, y_1^r, \dots, y_{k_r-1}^r$ form a basis of $(1,0)$-vector fields on $X_2$. As in the case $\varepsilon=0$, we have $\dbar\overline \alpha =0$, $\dbar \overline\beta_i^\ell = \overline\alpha \wedge \overline \beta_{i-1}^\ell$ and $\dbar y_i^\ell = \overline \alpha \otimes y_{i+1}^\ell$, but unless $k_0=1$ we have $\dbar z = \overline \alpha \otimes y_1^0$. So for $k_0 \geq 2$ we get
			\[
			H^{0,1}_{\dbar}(X_2) = \langle \overline \alpha, \delta_0\overline \beta_1^0, \dots, \delta_r\overline \beta_1^r \rangle
			\]
			and 
			\[
			H^0(X_2, \Theta_{X_2}) = \langle \delta_0y_{k_0-1}^0, \dots,  \delta_ry_{k_r-1}^r \rangle,
			\]
			where we define $\delta_i =0$ if $k_i=1$ and $\delta_i =1$ if $k_i>1$.
			Therefore, we have $\dim H^0(X_2, \Theta_{X_2}) < H^{0,1}_{\dbar}(X_2)$ and thus $\gamma_1$ is never surjective. If $k_0=1$, then $y_1^0$ is in the centre of $\frg_\IC$ and thus $\dbar z=0$ if $z$ is considered as a vector field on $X_2$. In this case, we have $\dim H^0(X_2, \Theta_{X_2}) = \dim H^{0,1}_{\dbar}(X_2)$. Hence, if $\gamma_1$ is surjective, then $\dim\gotht^{1,0}=1$ and thus $\dim \kz\frg =2$. But this implies that the matrix $A$ corresponding to $\frg$ has only two Jordan blocks and thus $k_i=0$ for all $i>0$. Since we already showed that $k_0=1$ if $\gamma_1$ is surjective, the forms $\alpha, \beta_1^0$ are a basis of left-invariant $(1,0)$-forms on $X$ with structure equations $d\alpha=0$ and $d\beta_1^0= \alpha \wedge \overline \alpha$. Therefore, $X$ is a Kodaira surface.\qedhere
		\end{prooflist}
	\end{proof}

	\section{Nilmanifolds of dimension six}\label{sect: nilmanifolds of dimension six}

	In this section, we will apply our previous results to the the class of six-dimensional nilmanifolds. Here, we will only consider those nilmanifolds for which the description of the Teichmüller or moduli space follows directly from one of the preceding theorems. The remaining cases will be treated in a separate article.

	Following Salamon's notation \cite{salamon01}, we describe a nilpotent Lie algebra $\mathfrak{g}$, or rather its isomorphism class, by writing $\mathfrak{g} = (de^1, de^2, \ldots, de^{2n})$ for a given basis $e^1, \dots, e^{2n}$ of $\frg^*$, and abbreviating $e^{ij} = e^i \wedge e^j$ to $ij$. For example, the tuple $(0,0,0,12)$ represents the four-dimensional Lie algebra with the single non-trivial differential $de^4 = e^{12}$. 
	
	Real nilpotent Lie algebras are classified up to dimension seven, which is the smallest dimension where infinite families start to appear (see for instance \cite{Magnin86}). The analysis which Lie algebras up to dimension six admit a complex structure was carried out by Salamon in \cite{salamon01}: up to dimension four there are two non-abelian nilpotent Lie algebras, namely, $(0,0,0,12)$ and $(0,0,12,13)$. The former is the real Lie algebra underlying a Kodaira surface and the latter does not admit a complex structure \cite[Proposition~2.3]{salamon01}.
	
	In dimension six the situation becomes more involved. There are 34 isomorphism classes of real nilpotent Lie algebras and Salamon's classification states that a six-dimensional nilpotent Lie algebra $\mathfrak{g}$ admits a complex structure if and only if $\mathfrak{g}$ is isomorphic to one of the following Lie algebras (cf. \cite[Table~A.1]{salamon01} or \cite[Theorem~8]{Ugarte07}):
	
	\begin{enumerate}
		\begin{multicols}{2}
			\item[] $\mathfrak{h}_1 = (0,0,0,0,0,0)$
			\item[] $\mathfrak{h}_2 = (0,0,0,0,12,34)$
			\item[] $\mathfrak{h}_3 = (0,0,0,0,0,12+34)$
			\item[] $\mathfrak{h}_4 = (0,0,0,0,12,14+34)$
			\item[] $\mathfrak{h}_5 = (0,0,0,0,13+42,14+23)$
			\item[] $\mathfrak{h}_6 = (0,0,0,0,12,13)$
			\item[] $\mathfrak{h}_7 = (0,0,0,12,13,23)$
			\item[] $\mathfrak{h}_8 = (0,0,0,0,0,12)$
			\item[] $\mathfrak{h}_9 = (0,0,0,0,12,14+25)$
			\item[] $\mathfrak{h}_{10} = (0,0,0,12,13,14)$
			\item[] $\mathfrak{h}_{11} = (0,0,0,12,13,14+23)$
			\item[] $\mathfrak{h}_{12} = (0,0,0,12,13,24)$
			\item[] $\mathfrak{h}_{13} = (0,0,0,12,13+14,24)$
			\item[] $\mathfrak{h}_{14} = (0,0,0,12,14,13+42)$
			\item[] $\mathfrak{h}_{15} = (0,0,0,12,13+42,14+23)$
			\item[] $\mathfrak{h}_{16} = (0,0,0,12,14,24)$
			\item[] $\mathfrak{h}_{19}^- = (0,0,0,12,23,14-35)$
			\item[] $\mathfrak{h}_{26}^+ = (0,0,12,13,23,14+25)$
		\end{multicols}
	\end{enumerate}
	
	Note that all of these Lie algebras have rational structure constants, so an associated simply connected Lie group always admits a lattice. The following theorem can readily be deduced from our previous results.
	
	\begin{thm}\label{thm: dimension six}
		Let $\frg$ be a nilpotent Lie algebra of dimension six admitting complex structures, and let $\Gamma$ be a lattice in an associated simply connected Lie group.
		\begin{enumerate}
			\item If $\frg \in \{ \frh_1, \frh_2, \frh_3, \frh_4, \frh_6, \frh_8, \frh_9, \frh_{10}\}$, then the global period map 
			\[
			\kp \colon \kt(\frg) \to \IP H^3(\frg, \IC)
			\]
			is injective and the Teichmüller space $\kt(\frg)$ exists as a complex manifold.
			\item If $\frg = \frh_5$, then $\kt(\frg)$ and $\km(\frg, \Gamma)$ are non-Hausdorff.
			\item If $\frg \in \{\frh_{13}, \frh_{14}, \frh_{15}\}$, then $\kt(\frg)$ does not exist as a complex space.
		\end{enumerate}
	\end{thm}
	\begin{proof}
		We begin with the first item. The Lie algebras $\frh_2$ and $\frh_4$ admit a SPTBS of length 2 \cite[Theorem B]{rollenske09}, and by the classification of complex structures given in \cite{ceballos2016invariant} every complex structure $J$ on $\frh_2$ or $\frh_4$ satisfies $h^0(J) =1$, that is, every complex nilmanifold of type $(\frh_2, \Gamma)$ or $(\frh_4, \Gamma)$ has one-dimensional automorphism group. Hence, for the Lie algebras $\frh_2$ and $\frh_4$ the first item follows from Theorem~\ref{thm: fibre map not injective}. The Lie algebras $\frh_3$ and $\frh_8$ have one-dimensional commutator and are thus covered by the first item of Theorem \ref{thm: Teichmüller 2-dim commutator}. The Lie algebra $\frh_9$ is 3-step nilpotent and has two-dimensional commutator, so the result follows from the last item of Theorem \ref{thm: Teichmüller 2-dim commutator}. Finally, the almost abelian Lie algebras $\frh_6$ and $\frh_{10}$ are covered by Theorem \ref{thm: global period map almost abelian}.
		
		The Lie algebra $\frh_5$ is non-abelian and admits complex parallelisable structures. Thus, $\kt(\frh_5)$ and $\km(\frh_5, \Gamma)$ are non-Hausdorff by Proposition \ref{prop: Tg non hausdorff complex par}. 
		
		For the Lie algebras listed in the last item, it again follows from the classification of complex structures proved in \cite{ceballos2016invariant} that the function $h^0$ is not constant on the connected components of $\kc(\frg)$ and thus the Teichmüller space does not exist as a complex space by Proposition \ref{prop: fibre dimension criterion}.
	\end{proof}

	\begin{rem}
		The Lie algebra $\frh_5$ is the real Lie algebra underlying the Iwasawa manifold. We already mentioned in Example \ref{exam: Iwasawa} that the period map $\kp$ on the Teichmüller space of the Iwasawa manifold $\kt(\frh_5)$ is generically injective. This allows one to replace $\kt(\frh_5)$ by its largest Hausdorff quotient to
		obtain a complex manifold of dimension four, parametrising complex structures on
		$\frh_5$, which coincides with $\kt(\frh_5)$ on a dense open subset. This approach will be further explored in a separate article.
	\end{rem}
	



	\printbibliography

\end{document}